\documentclass{amsart}

\usepackage{amssymb, amsmath, amsxtra, ulem}
\usepackage{xcolor}

\definecolor{azul}{rgb}{0,0,1}

\definecolor{nuovo}{rgb}{0.75,0.05,0.55}

\usepackage[colorlinks=false]{hyperref}
\usepackage{tikz-cd}
\usepackage{tikz}
\usetikzlibrary{patterns}
\usepackage{palatino}

\usepackage{indentfirst}
\usepackage{graphicx}

\newtheorem{theorem}{Theorem}[section]
\newtheorem{proposition}[theorem]{Proposition}

\newtheorem{lemma}[theorem]{Lemma}
\newtheorem{remark}[theorem]{Remark}
\newtheorem{definition}[theorem]{Definition}
\newtheorem{example}[theorem]{Example}
\usepackage[left=2.9cm,right=2.9cm,top=2.8cm,bottom=2.8cm]{geometry}

\def\re{\mathbb{R}}

\def\dem{\noindent{\it Proof. }}
\def\dys{\displaystyle}
\newcommand{\fim}{\hfill $\Box$}

\usepackage{booktabs,amsmath}

\def\XXint#1#2#3{{\setbox0=\hbox{$#1{#2#3}{\int}$}
    \vcenter{\hbox{$#2#3$}}\kern-.5\wd0}}

\title{Sharp existence and non-existence for singular $1$-Laplace equations with Hardy potentials}

\author{J. C. Ortiz Chata}
\address{Departamento de Matem\'atica, Universidade Federal de São Carlos - UFSCar, 13565-905, S\~ao Carlos - SP, Brazil}
\address{juan.carlos@unesp.br}
\author[F. Petitta]{Francesco Petitta}
\address{Dipartimento di Scienze di Base e Applicate per l' Ingegneria, Sapienza Universit\`a di Roma, Via Scarpa 16, 00161 Roma, Italia}
\address{francesco.petitta@uniroma1.it}

\date{}

\begin{document}
\pretolerance10000

\maketitle

\begin{abstract}

In this paper, we investigate the singular Dirichlet problem
\begin{equation}\label{abs1}
\left\{
\begin{array}{rclr}
\dys -\Delta_1u&=&\dfrac{\lambda}{|x|}\hbox{Sgn}\,(u)+\dfrac{f}{u^\gamma}&\quad\mbox{in}\; \Omega,\\[1ex]
u&=&0&\quad\mbox{on}\; \partial\Omega,
\end{array}
\right.
\end{equation}
where $\Delta_1u=\hbox{div}\left(\frac{Du}{|Du|}\right)$ denotes the $1$-Laplacian operator, the parameters are $\lambda\in \re$ and $\gamma>0$, and $f$ is a non-negative function belonging to the Lorentz space $L^{N,\infty}(\Omega)$. Our main goal is to establish the existence of non-trivial solutions to \eqref{abs1} under {the sharp restriction $\lambda<N-1$}, regardless of the magnitude of the datum $f$ or the value of the singular exponent $\gamma>0$. {We also show that these solutions are globally bounded and that, on the contrary, no solution exists as soon as $\lambda\geq N-1$ and $f$ is positive.} These results are achieved by rigorously analyzing the asymptotic behavior, as $p\to 1^+$, of the solutions to the approximating $p$-Laplace problems
 
\begin{equation}
\left\{
\begin{array}{rclc}
-\Delta_pu&=&\dfrac{\lambda}{|x|^p}|u|^{p-2}u+\dfrac{f}{u^\gamma}&\quad\mbox{ in }\; \Omega,\\[1ex]
              u&=&0&\quad\mbox{ on }\; \partial\Omega.
\end{array}
\right.
\end{equation}
Finally, we provide {a family of} explicit examples designed to illustrate the sharp optimality of our main assumptions.
\end{abstract}

\tableofcontents
\numberwithin{equation}{section}

\section{Introduction}
In this paper, we consider the Dirichlet problem
\begin{equation}
\label{Pintro} 
\left\{
\begin{array}{rclr}
-\Delta_1u&=&\dfrac{\lambda}{|x|}\dfrac{u}{|u|}+\dfrac{f}{u^\gamma}&\quad \mbox{ in } \Omega,\\[1ex]
u&=&0 &\quad \mbox{ on }\partial\Omega,
\end{array}
\right.
\end{equation} 
where $\Delta_1u=\hbox{div}\left(\frac{Du}{|Du|}\right)$ denotes the 1-Laplacian operator, $\Omega \subset \mathbb{R}^N$ $(N\geq2)$ is a bounded open set with Lipschitz boundary $\partial\Omega$ containing the origin $0 \in \Omega$, $\lambda < N-1$, and $f$ is a non-negative datum belonging to the Lorentz space $L^{N,\infty}(\Omega)$.

The unperturbed case $\lambda=\gamma=0$ is by now well understood (\cite{Kawohl1990, ct, KawohlSchuricht2007, mst}) and has traditionally been addressed either variationally or via an asymptotic analysis of solutions $u_p$ to approximating problems involving the $p$-Laplace operator as $p\to1^+$, namely
\begin{equation}
\label{f1}
\left\{
\begin{array}{rclc}
-\Delta_pu_p&=&f&\quad\hbox{in}\;\Omega,\\[0.5ex]
 u_p&=&0&\quad\hbox{on}\;\partial\Omega,
\end{array}
\right.
\end{equation}
where $\Delta_p u=\hbox{div}\left(|\nabla u|^{p-2}{\nabla u}\right)$ is the standard $p$-Laplacian. In \cite{Kawohl1990}, the author investigated the torsion problem associated with \eqref{f1} for $p=1$ (i.e., $f=1$), highlighting how the existence and shape of solutions strongly depend on the geometry and size of $\Omega$, which in turn leads to a free boundary problem (\cite[Section 6]{Kawohl1990}). Furthermore, in \cite{ct, mst}, the authors studied general data $f$ in $L^N(\Omega)$ (as well as in $L^{N,\infty}(\Omega)$ or $W^{-1,\infty}(\Omega)$). Under sharp smallness conditions on the norm of $f$, existence results were established. However, a severe degeneracy phenomenon occurs: the sequence of approximating solutions $u_p$ is shown to vanish identically in the limit as $p \to 1^+$ for \textit{almost all small data}. Nevertheless, for certain specific data with suitably tailored norms, non-trivial solutions, often referred to as \textit{almost 1-harmonic functions} (see \cite{dem}), do exist.

The case $0<\lambda<N-1$ with $\gamma=0$ was recently addressed in \cite{OrtizPetitta2024}. There, the authors proved the existence of solutions in the sense of \cite{acm2001, AndreuMazonMollCaselles2004, AndreuCasellesMazon2004} via an approximation limit $p \to 1^+$. Moreover, they provided a complete characterization of the asymptotic behavior of $u_p$, showing its sharp dependence on the interplay between the parameter $\lambda$ and the size of $f$. Related issues were investigated in \cite{OrtizPetittaP} in the presence of a critical Hardy-type drift term.

Due to the extreme degeneracy inherent to $1$-Laplacian-type operators, considerable attention has recently been devoted to understanding the regularizing properties of singular lower-order terms as a mechanism to prevent triviality and guarantee the existence of non-trivial solutions. Inspired by the case $p>1$, the authors in \cite{DeCiccoGiachettiSegura2019, dgop} observed that singular reaction terms exert a strong regularizing effect. More precisely, regardless of the norm of $f$, they proved the existence of non-trivial solutions to the Dirichlet problem
\begin{equation}
\left\{
\begin{array}{rclr}
-\Delta_1u&=&\dfrac{f}{u^\gamma}&\quad \mbox{ in } \Omega,\\[1ex]
u&=&0 &\quad \mbox{ on }\partial\Omega,
\end{array}
\right.
\end{equation} 
where $\Omega\subset\mathbb{R}^N$ is a bounded Lipschitz domain, $\gamma> 0$, and $f$ is a positive function in $L^{N,\infty}(\Omega)$.

Singular elliptic equations and their regularizing phenomena have been widely studied since the classical works \cite{crt77, Stu} and the landmark paper \cite{lm}, experiencing a renewed surge of interest starting with \cite{BoccardoOrsina2010}{; see also \cite{DDO} for the case of measure data}. For a comprehensive survey on the topic, we refer to \cite{survey}.

To better understand the regularity properties driven by singular terms, let us briefly recall the model problem
\begin{equation}\label{pbintro}
\left\{
\begin{array}{rclr}
-\Delta_p u&=&\dfrac{f}{u^\gamma}&\quad \mbox{ in } \Omega,\\[1ex]
u&=&0 &\quad \mbox{ on }\partial\Omega.
\end{array}
\right.
\end{equation}

When $p=2$, it was shown in \cite{lm} that $u\notin C^{1}(\overline{\Omega})$ whenever $\gamma> 1$, whereas $u\in H^{1}_0(\Omega)$ if and only if $\gamma<3$.

For non-negative data $f \in L^{m}(\Omega)$, the authors in \cite{BoccardoOrsina2010} established the following regularity properties:
\begin{itemize}
    \item[\textit{i)}] If $\gamma<1$ and $m= \left(\frac{2^*}{1-\gamma}\right)'$, then $u \in H^1_0(\Omega)$; if $1 \le m <\left(\frac{2^*}{1-\gamma}\right)'$, then $u \in W^{1,\frac{Nm(\gamma+1)}{N-m(1-\gamma)}}_0(\Omega)$.
    \item[\textit{ii)}] If $\gamma=1$ and $m = 1$, then $u \in H^1_0(\Omega)$.
    \item[\textit{iii)}] If $\gamma>1$ and $m = 1$, then $u \in H^1_{\rm{loc}}(\Omega)$, and 
    \begin{equation}\label{potenza}
    u^{\frac{\gamma+1}{2}}\in H^1_0(\Omega).
    \end{equation}
\end{itemize}

Property \eqref{potenza} is not merely a technical artifact; it provides the precise sense in which the Dirichlet boundary condition is attained. Indeed, as counterexamples show, distributional solutions to \eqref{pbintro} may fail to belong to $H^1_0(\Omega)$ while still satisfying \eqref{potenza}.

In this framework, the classical Lazer–McKenna threshold for the existence of finite energy solutions with $\gamma>1$ was extended to general $f\in L^{m}(\Omega)$: problem \eqref{pbintro} admits a (unique) solution in $H_0^1(\Omega)$ for every positive $f\in L^{m}(\Omega)$ ($m> 1$) if and only if $\gamma < {3-\frac{2}{m}}$ (see \cite{OP}).

Several of these results have been extended to non-linear $p$-Laplacian operators ($p>1$) of the form
\begin{equation} \label{op}
\begin{cases}
- \Delta_p u= f(x)u^{-\gamma} & \text{in } \Omega, \\
u \ge 0 & \text{in } \Omega, \\
u = 0 & \text{on } \partial \Omega.
\end{cases}
\end{equation}

In this context, for smooth positive data, the sharp Lazer–McKenna condition ensuring that $u \in W^{1,p}_0(\Omega)$ becomes (\cite{S})
\begin{equation}\label{condLMp}
\gamma < \frac{2p-1}{p-1}.
\end{equation}
For generic data $f \in L^m(\Omega)$, it was shown in \cite{DCA} that there exist distributional solutions $u \in W^{1,p}_{\rm loc}(\Omega)$ satisfying
\begin{equation} \label{formally}
u^{ \max\left(1,\frac{\gamma-1+p}{p}\right)} \in W^{1,p}_0(\Omega).
\end{equation}
Thus, as in the linear case, when $\gamma>1$ only a suitable power of the solution lies in the Sobolev energy space $W_0^{1,p}(\Omega)$.

In \cite{DeCiccoGiachettiSegura2019}, for $0<\gamma\le 1$, by passing to the limit in approximating problems with $p>1$, the authors proved the existence and (for $f>0$) uniqueness of non-negative bounded solutions to
\begin{equation}\label{now}
\begin{cases}
- \Delta_1 u = f(x){u^{-\gamma}} & \text{in } \Omega, \\
u = 0 & \text{on } \partial \Omega,
\end{cases}
\end{equation}
where $0\le f\in L^{N,\infty}(\Omega)$. For $\gamma>1$, these solutions satisfy $u^{\gamma}\in BV(\Omega)$. Recently, in \cite{maop}, this result was significantly improved by establishing the existence of a true finite energy solution (i.e., $u\in BV(\Omega)$) for problem \eqref{now}, regardless of how large $\gamma>0$ is. Notice that this result perfectly matches the limit of threshold \eqref{condLMp} as $p\to1^+$. It is worth noting that, in \cite{dgop}, the authors studied existence and regularity results for singular elliptic equations with general nonlinearities that include \eqref{now}. 

In \cite{AbdellaouiAttar2013}, for $p>1$, the authors studied the existence and optimal summability of solutions to the following problem
\begin{equation}
\label{op1}
\left\{
\begin{array}{rclr}
-\Delta_p u&=&\lambda\frac{|u|^{p-2}u}{|x|^p}+\frac{f}{u^\gamma}&\mbox{ in }\;\Omega\\
u&=&0&\mbox{ on }\;\partial\Omega
\end{array}
\right.
\end{equation}
with respect to the summability of $f$ and the value of the parameter $\lambda$; {there the notion of entropy solution introduced in \cite{BenilanBoccardoGallouetGariepyPierreVazquez1995} is used}. Moreover, it is well {known} that problem {\eqref{op1}} has a solution in $W^{1,p}_0(\Omega)$  when $\gamma=0$ (for instance, see \cite{AzoreroAlonso1998, LeonoriMartinezPrimo2011, OrtizChata2025}).

Now, let us focus on problem \eqref{op1} in the limiting regime $p=1$, where the natural finite energy space is $BV(\Omega)$, the space of functions of bounded variation.  In order to do so, we study the asymptotic behavior of the solutions of the approximate $p-$Laplace problem \eqref{op1}, as $p$ goes to 1, which, in turn, is considered with the concept of distributional solution as defined in \cite{BoccardoOrsina2010,DCA}. More precisely, the main goal of the present paper is to establish optimal existence and global regularity results for non-trivial \textit{finite energy solutions} to problem \eqref{Pintro}, under {the sharp restriction $\lambda<N-1$} and for any arbitrary $\gamma>0$ and size of $f$. Furthermore, whenever $f$ is strictly positive, we show that strictly positive solutions can be obtained.

{Let us emphasize a structural feature of our result. If $\gamma=0$, whether the limit of the approximating sequence $(u_p)$ is trivial or not is governed by a sharp balance between $\lambda$ and the size of $\|f\|_{L^{N,\infty}(\Omega)}$, and for {\it almost all small data} degeneracy occurs (see \cite{ct, mst, OrtizPetitta2024}). Here, on the contrary, no smallness (nor largeness) assumption on $f$ is needed, for a reason which is structural rather than technical: the requirement $fu^{-\gamma}\in L^1_{\rm loc}(\Omega)$ appearing in the  definition of solution (see $(a)$ of Definition \ref{DS} below) rules out the trivial function as soon as $f\not\equiv0$. Accordingly, the a priori bound \eqref{L15} below, which is central in  the whole argument, does not depend on $f$: it is the singular term which, no matter how small $f$ is, prevents the solution from vanishing.}

{Finally, we show that the restriction $\lambda<N-1$ is sharp: if $f$ is positive and $\lambda\geq N-1$, then problem \eqref{Pintro} admits no solution at all (Theorem \ref{NoEx} below). This is in agreement with the fact that $N-1$ is the best constant in the Hardy inequality in $BV(\Omega)$, which is known not to be attained (\cite{crt, WangWillem}).}

{The paper is organized as follows: in Section \ref{Prel} we collect some preliminary material on Lorentz spaces, functions of bounded variation, and an extended version of Anzellotti's pairing theory. In Section \ref{main} we formally set up the problem, we state our main assumptions and results, and we describe the structure of the proof. Section \ref{ApproxP} is devoted to the approximating $p$-Laplace problems and to the uniform estimates on their solutions, while in Section \ref{PT} we perform the limit as $p\to1^+$. In Section \ref{PosDat}, under the assumption of positive data, we construct strictly positive solutions, and we prove that the limit function has finite energy also in the strongly singular regime. Section \ref{Br} addresses global regularity properties of the solutions obtained and contains the conclusion of the proof of Theorem \ref{mainTh}. Finally, in Section \ref{Opt} we prove that the assumption $\lambda<N-1$ is optimal, and we exhibit a family of explicit solutions illustrating our results.}

\section{Preliminaries}
\label{Prel}

We start with some notations: throughout the paper  $\Omega$  is a bounded open subset of $\mathbb{R}^N$, $B_r(x)$ the ball centering in $x$ with radius $r>0$ and 
$$
\Omega_\epsilon=\{x\in \Omega\colon \hbox{dist}\,(x, \partial\Omega)>\epsilon\}.
$$
 Moreover, we denote the $\mathcal{H}^{N-1}-$dimensional Hausdorff measure on $\mathbb{R}^N$ and, for each Borel set $E\subset \Omega$, we denote its $N-$ dimensional Lebesgue measure on $\mathbb{R}^N$ by $|E|$. Define the truncation $T_k\colon\mathbb{R}\to \mathbb{R}$ by
\begin{equation}
\label{Tk}
T_k(s)=\left\{
\begin{array}{lc}
s&\hbox{if }\; |s|\leq k,\\
k\frac{s}{|s|}&\hbox{if}\;|s|>k,
\end{array}
\right.
\end{equation}
and $G_k\colon \mathbb{R}\to \mathbb{R}$ as
\begin{equation}
\label{Gk}
G_k(s)=s-T_k(s)\qquad\hbox{for all }\;s\in\mathbb{R}.
\end{equation}
 
\subsection{Basics on Lorentz spaces} 
Let $\Omega\subset \mathbb{R}^N$ be a bounded open set. Let $u\colon\Omega\to \mathbb{R}$ be a measurable function. 

The non-increasing rearrangement of $u$ on $(0,|\Omega|)$ is defined by
$$
u^*(t)=\inf\{s>0\colon |\{x\in \Omega\colon |u(x)|>s\}|\leq t\}\quad\hbox{ for all }\; t\in (0,|\Omega|).
$$
For any $1\leq p<\infty$ and $1\leq q\leq \infty$ and $u$ a measurable function in $\Omega$. We define $\|u\|_{L^{p,q}(\Omega)}$ as
\begin{equation}
\label{SN}
\|u\|_{L^{p,q}(\Omega)}=\left\{
\begin{array}{lcr}
\left(\int_{0}^{|\Omega|}(t^{1/p}u^*(t))^q\frac{dt}{t}\right)^{1/q}&\hbox{ if }& 1\leq q<\infty\\
\sup_{0<t<|\Omega|}t^{1/p}u^*(t)& \hbox{ if} & q=\infty.
\end{array}
\right.
\end{equation}
\begin{definition}
Let $1\leq p<\infty$ and $1\leq q\leq \infty$. The Lorentz space $L^{p,q}(\Omega)$ is the set of all measurable functions $u\colon\Omega\to \mathbb{R}$ such that the quantity \eqref{SN} is finite, i.e., 
$$
L^{p,q}(\Omega)=\{u\colon \Omega\to \mathbb{R}\quad\hbox{ measurable function}\colon \|u\|_{L^{p,q}(\Omega)}<\infty\}.
$$
\end{definition}

Observe that, for  $1\leq r\leq \infty$,  the previous definition implies that
$$
L^{r,r}(\Omega)=L^r(\Omega),
$$
while the following inclusions   hold for $q<N$:
$$
L^N(\Omega)\subset L^{N,\infty}(\Omega)\subset L^q(\Omega).
$$

The proof of the following proposition can be found in \cite{ONeil1968} (see also \cite[Appendix]{DAnconaFanelli2007}).

\begin{proposition}[H\"older inequality]
Let $ 1<p<\infty$ and $1\leq q\leq \infty$. For any $u\in L^{p,q}(\Omega)$ and  $v\in L^{p',q'}(\Omega)$ with $1/p+1/p'=1$ and $1/q+1/q'=1$, the following estimates holds:
$$
\|uv\|_{L^1(\Omega)}\leq \|u\|_{L^{p,q}(\Omega)}\|v\|_{L^{p',q'}(\Omega)}.
$$
\end{proposition}

\begin{proposition}[Sobolev's inequality]
\label{SLp}
Let $1\leq p <N$ and $p^*=Np/(N-p)$. There exists a constant $\zeta_{N,p}>0$ such that the following inequality holds:
$$
\|u\|_{L^{ p^*,p}(\Omega)}\leq \zeta_{N,p}\|\nabla u\|_{L^p(\Omega)}\quad\hbox{ for all }\; u\in C^\infty_c(\Omega).
$$
Furthermore, the constant below is optimal 
$$
\zeta_{N,p}=\frac{p}{(N-p)|B_1(0)|^{1/N}}.
$$
\end{proposition}

The proof of the previous result can be found in \cite{Alvino1977}.

{Throughout the paper we also denote by $S_{N,p}$ the best constant in the classical Sobolev inequality, namely
\begin{equation}
\label{SobC}
\|u\|_{L^{p^*}(\Omega)}\leq S_{N,p}\|\nabla u\|_{L^p(\Omega)}\qquad\hbox{ for all }\; u\in W^{1,p}_0(\Omega),\quad 1\leq p<N.
\end{equation}
Let us recall that $p\mapsto S_{N,p}$ is continuous on $[1,N)$; in particular $S_{N,p}$ stays bounded as $p\to 1^+$, a fact that will be repeatedly used when passing to the limit in the estimates of Section {\ref{ApproxP}}.}

\subsection{The space $BV(\Omega)$}
 We say that $u\in L^1(\Omega)$ is a bounded variation function in $\Omega$ if the distributional derivative of $u$ is a finite Radon measure in $\Omega$, i.e., if 
$$
\int_{\Omega}u\frac{\partial\phi}{\partial x_i}\,dx=-\int_{\Omega}\phi\,dD_iu\quad\hbox{ for all }\; \varphi\in C^\infty_c(\Omega),\quad i=1,...,N
$$
for some $\mathbb{R}^N-$valued Radon measure $Du=(D_1u,...,D_Nu)$ in $\Omega$. We denote the space of such functions ($BV$ functions) by $BV(\Omega)$.

\begin{proposition}
Let $u\in BV(\Omega)$. Then, the total variation of the $\mathbb{R}^N-$valued measure $Du$ satisfies the following identity
\begin{equation}
\label{TV}
 \int_{\Omega}|Du|=\sup\left\{\int_{\Omega}u\mbox{div}\,\varphi\,dx\>:\>\varphi\in C^1_c(\Omega,\mathbb{R}^N),\quad|\varphi|\leq 1\right\}.
 \end{equation}
 Moreover, the map $BV(\Omega)\ni u\mapsto \int_{\Omega}|Du|$ is lower semicontinuous in $BV(\Omega)$ with respect to the $L^1_{loc}(\Omega)$ topology.
\end{proposition}

The space $BV(\Omega)$ is a Banach space endowed with the norm
\begin{equation}
\label{N}
\|u\|_{BV(\Omega)}=\int_{\Omega}|Du|+\int_{\Omega}|u|\,dx.
\end{equation}

 We say that $u\in BV_{loc}(\Omega)$ if for each open set $\omega \subset\subset \Omega$, 
$$
u\in BV(\omega).
$$

In order to state the embedding theorem in $BV(\Omega)$, we recall the following result.

\begin{theorem}[Embedding theorem]
Let $\Omega$ be an open subset in $\mathbb{R}^N$ with bounded Lipschitz boundary. Then, the embedding $BV(\Omega)\hookrightarrow L^{1^*}(\Omega)$ is continuous and the embeddings $BV(\Omega)\hookrightarrow L^q(\Omega)$ are compact for $1\leq q <1^*$, $1^*=N/(N-1)$.
\end{theorem}

An easy extension of Proposition \ref{SLp}  is the following (see \cite{pe,Alvino1977,ct}): 
\begin{proposition}[Sobolev's inequality in $BV$]
\label{LSI}
There exists a constant $\zeta>0$ such that the following inequality holds for all $u\in BV(\Omega)$:
\begin{equation}
\|u\|_{L^{1^*, 1}(\mathbb{R}^N)}\leq \zeta \int_{\mathbb{R}^N}|Du|.
\end{equation}
Furthermore, the better constant is given by
\begin{equation}
\label{bestl}
\zeta_N = \frac{1}{(N-1)|B_1(0)|^\frac{1}{N}}.
\end{equation}
\end{proposition}

\subsection{Fine properties   of $BV$ functions }

In order to show some fine properties of $BV$ functions, we define approximate limits. 

\begin{definition}
Let $u\in L^1_{loc}(\Omega)$. We say that $u$ has a approximate limit at $x\in\Omega$ if there exists $\tilde{u}(x)\in \mathbb{R}$ such that
\begin{equation}\label{limit}
\lim_{\rho\downarrow0} \frac{1}{|B_{\rho}(x)|}\int_{B_{\rho}(x)}|u(y)-\tilde{u}(x)|\,dx=0,
\end{equation}
\end{definition}

We denote by $S_u$ the set of points where the limit in \eqref{limit} does  not exist; this called the singular set ot the  approximate discontinuity set of $u$.
We also consider  among the approximate discontinuity points those corresponding to an approximate jump discontinuity between two values $u^+(x)$ and $u^-(x)$ along a direction $\nu$.  More precisely:

\begin{definition}
Let $u\in L^1_{loc}(\Omega)$. We say that $x\in S_u$ is an approximate jump point of $u$ if there exists $u^+(x),\;u^-(x)\in \mathbb{R}$ and $\nu \in S^{N-1}$ such that 
\begin{equation}
\lim_{\rho\downarrow0} \frac{1}{|B^+_{\rho}(x,\nu)|}\int_{B_{\rho}^+(x,\nu)}|u(y)-u^+(x)|\,dx=0,\qquad \lim_{\rho\downarrow0} \frac{1}{|B^-_{\rho}(x,\nu)|}\int_{B_{\rho}^-(x,\nu)}|u(y)-u^-(x)|\,dx=0,
\end{equation}
where
$$
B^+_\rho(x,\nu)=\{y\in B_\rho(x)\colon \langle y-x,\nu\rangle>0\}, \qquad B^-_\rho(x,\nu)=\{y\in B_\rho(x)\colon \langle y-x, \nu\rangle<0\}.
$$
We denote by $J_u$ the set of approximate jump points.
\end{definition}
We define the precise representative $u^*\colon \Omega\setminus(S_u\setminus J_u)\to \mathbb{R}$ of $u$ by 
\begin{equation}
u^*(x)=
\left\{
\begin{array}{lcr}
\tilde{u}(x)&\hbox{if}&x\in \Omega\setminus S_u,\\
\frac{u^++u^-}{2}&\hbox{if}& x\in J_u.
\end{array}
\right.
\end{equation}

Now, let us define the standard mollifiers $\rho_\epsilon\colon \mathbb{R}^N\to \mathbb{R}$ by
\begin{equation}
\label{mll}
\rho_{\epsilon}(x)=\left\{
\begin{array}{lc}
a\epsilon^{-N}\exp{(-\epsilon^2/(\epsilon^2-|x|^2)}&\hbox{ if }\; |x|<\epsilon\\
0&\hbox{ if }\; |x|\geq \epsilon,
\end{array}
\right.
\end{equation}
where 
$$
a^{-1}=\int_{\{|x|\leq 1\}}\exp{(-1/(1-|x|^2))}\,dx.
$$

We have the following: 
\begin{proposition}
Let $u\in BV(\Omega)$ and let $(\rho_\epsilon)_{\epsilon>0}$ be the mollifiers defined as in \eqref{mll}. Then
$$
u\ast \rho_\epsilon(x)\to u^*(x)\quad \hbox{ for all }\; x\in \Omega\setminus(S_u\setminus J_u)\quad\hbox{ as }\; \epsilon\to 0.
$$
\end{proposition}

\begin{proposition}[Properties of trace operator]
\label{BT}
Let $\Omega\subset \mathbb{R}^N$ be an open set with bounded Lipschitz boundary $\partial\Omega$. For each $u\in BV(\Omega)$ there exists $\tau(u)\in L^1(\partial\Omega)$ such that
$$
\lim_{\rho\downarrow 0}\rho^{-N}\int_{\Omega\cap B_\rho(x)}|u(y)-\tau(u(x))|\,dy=0\quad\mathcal{H}^{N-1}-\hbox{a.e.} \;x\in \partial\Omega.
$$
Moreover, $\tau\colon BV(\Omega)\to L^1(\partial\Omega)$ is a bounded linear operator that satisfies
\begin{equation}
\label{TrIn}
 \int_{\Omega}|u|\,dx\leq N\left(\frac{|\Omega|}{|B_1(0)|}\right)^{1/N}\left(\int_{\Omega}|Du|+\int_{\partial\Omega}|\tau(u)|\,d\mathcal{H}^{N-1}\right)\quad\hbox{ for all }\; u\in BV(\Omega)
\end{equation}
and it is also continuous with respect to the topology induced by strict convergence, i.e., $(u_n)_{n\geq1}$ strictly converges in $BV(\Omega)$ to $u$ if 
$$
u_n\to u\quad\hbox{ in }\;L^1(\Omega)\quad\hbox{ and }\quad\int_{\Omega}|Du_n|\to\int_{\Omega}|Du|,
$$
as $n\to\infty$.
\end{proposition}

The proof of previous proposition can be found in \cite[Theorems 3.87, 3.88]{AmbrosioFuscoPallara} and \cite[Proposition 2]{Miranda1974}. From now on we simply denote $u\in L^1(\partial\Omega)$ instead of $\tau(u)$.

\begin{remark}
Clearly, the norm $\|\cdot\|_{BV(\Omega)}$ is equivalent to the following
$$
\|u\|:=\int_{\Omega}|Du|+\int_{\partial\Omega}|u|\,d\mathcal{H}^{N-1},
$$
by previous proposition.
 \end{remark}

More details about space $BV(\Omega)$ can be found in \cite{AmbrosioFuscoPallara} (see also \cite{AttouchButtazzoMichaille2006, EvansGariepy1992}) from which we mainly inherit our notation.

\subsection{An extension of Anzellotti's theory} 

In this subsection, we recall an extension of Anzellotti's theory \cite{Anzellotti1983}, given in \cite{Caselles2011} and \cite{DeCiccoGiachettiSegura2019}{; we refer to \cite{CrastaDeCicco} for a systematic treatment of the pairing theory}.

Let 
$$
X_{\mathcal{M}}(\Omega)=\{{\bf z}\in L^{\infty}(\Omega,\mathbb{R}^N)\>:\> \mbox{div}\,{\bf z} \; \mbox{ is a finite Radon measure in }\;\Omega\}.
$$

\begin{definition}
\label{zDu}
Let ${\bf z} \in X_{\mathcal{M}}(\Omega)$ and $u\in BV(\Omega)$ be such that $ u^*\in L^1(\Omega, \hbox{div}\,{\bf z})$. We define the linear functional $({\bf z}, Du)\>:\>C^\infty_c(\Omega)\to \mathbb{R}$ by 
$$
\langle ({\bf z}, Du),\varphi\rangle=-\int_{\Omega}u^*\varphi\,\mbox{div}\,{\bf z}-\int_{\Omega}u {\bf z}\cdot \nabla \varphi\,dx\quad\hbox{ for all }\;\varphi\in C^\infty_c(\Omega).
$$
\end{definition}

\begin{proposition}
\label{Dist}
Let ${\bf z} \in X_{\mathcal{M}}(\Omega)$ and $u\in BV(\Omega)$ be such that $ u^*\in L^1(\Omega, \hbox{div}\,{\bf z})$ and let $A\subset \Omega$ be a open set. For all functions $\varphi\in C_c(A)$ the following inequality holds                                                                                                                                                                                                                                                                                                                                                                                                                
$$
|\langle ({\bf z}, Du), \varphi\rangle |\leq \sup_{A}|\varphi|\|{\bf z}\|_{L^\infty(A)}\int_{A}|Du|,
$$
which means that $({\bf z}, Du)$ is a finite Radon measure on $\Omega$. Furthermore, $({\bf z}, Du)$, $|({\bf z}, Du)|$ are absolutely continuous with respect to the measure $|Du|$ in $\Omega$ and 
$$
\left|\int_{B}({\bf z}, Du)\right|\leq \int_{B} |({\bf z}, Du)|\leq \|{\bf z}\|_{L^\infty(A)}\int_{B}|Du|,
$$
for all Borel sets $B$ and for all opens sets $A$ such that $B\subset A\subset \Omega$.
\end{proposition}

\dem 
Let ${\bf z}\in X_{\mathcal{M}}(\Omega)$ and $u\in BV(\Omega)$ be such that $u^*\in L^1(\Omega, \hbox{div}\,{\bf z})$. Let $(\rho_n)$ be the mollifiers  defined as in \eqref{mll} with $\epsilon=1/n$  and $u_n^{(k)}=T_k(u)\ast \rho_n$, where $T_k$ is the truncation function defined in \eqref{Tk}.

Since, by  \cite[Corollary 3.80]{AmbrosioFuscoPallara},
$$
u_n^{(k)}\to (T_k(u))^*\quad\mathcal{H}^{N-1}-\hbox{ a.e. in }\;\Omega
$$
and, by \cite[Proposition 3.1]{ChenFried1999}, 
$$
u_n^{(k)}\to (T_k(u))^*\quad\hbox{div}\,{\bf z}-\hbox{ a.e. in }\;\Omega\quad\hbox{ as }\; n\to \infty.
$$

Let $\varphi\in C^1_c(A)$ with $\hbox{supp}\,\varphi\subset V\subset\subset A\subset \Omega$ and $\int_{\partial V}|Du|=0$. Then  

\begin{align}
\langle ({\bf z}, Du_n^{(k)}), \varphi\rangle &=\int_{V}\varphi {\bf z}\cdot \nabla u_n^{(k)}\,dx\\
&\leq \sup_{A}|\varphi|\|{\bf z}\|_{L^\infty(A,\mathbb{R}^N)}\int_{\Omega}|\nabla u_n^{(k)}|\,dx\\
&\leq \sup_{A}|\varphi|\|{\bf z}\|_{L^\infty(A,\mathbb{R}^N)}\int_{\Omega}|DT_k(u)\ast \rho_n|\,dx\\
&\leq \sup_{A}|\varphi|\|{\bf z}\|_{L^\infty(A, \mathbb{R}^N)}|DT_k(u)|(I_n(V)), 
\end{align}
where $I_n(V)=\{x\in A\colon \hbox{dist}\,(x, V)<1/n$\}. Letting $n\to \infty $ above, we obtain
$$
\langle ({\bf z}, DT_k(u)), \varphi\rangle\leq \|{\bf z}\|_{L^\infty(A, \mathbb{R}^N)}|DT_k(u)|(\overline{V})
$$
 and, by \cite[Proposition 2.20]{OrtizPetitta2024}, letting $k\to \infty$ , this becomes
$$
\langle ({\bf z}, Du),\varphi \rangle \leq \sup_{A}|\varphi|\|{\bf z}\|_{L^\infty(A,\mathbb{R}^N)}\int_{V}|Du|,
$$
for all $\varphi \in C_c^1(A)$. Clearly, the functional $\langle ({\bf z}, Du),\varphi\rangle$ can be extended continuously to $C_c(A)$,  the uniform closure of $C_c^1(A)$, and to $C_0(A)$, the uniform closure of $C_c(A)$. Hence
$$
({\bf z}, Du)
$$
is a finite Radon measure on $A$.
\fim 

\begin{proposition}
Let ${\bf z}\in X_{\mathcal{M}}(\Omega)$ and $u\in BV(\Omega)$ be such that $u^*\in L^1(\Omega, \hbox{div}\,{\bf z})$. Then the following Green formula holds
$$
\int_{\Omega}u^*\,\hbox{div}\,{\bf z}+\int_{\Omega}({\bf z}, Du)=\int_{\partial\Omega}u[{\bf z},\nu]\,\mathcal{H}^{N-1},
$$
where $[{\bf z},\nu]$ is the weak trace on $\partial\Omega$ of the normal component of ${\bf z}$ as defined in \cite{Anzellotti1983, Caselles2011}.
\end{proposition}
The proof of the previous result can be found in \cite[Proposition 2.5]{DeCiccoGiachettiSegura2019}.

We set
$$
X_{\mathcal{M}_{loc}}(\Omega)=\{{\bf z}\in L^\infty(\Omega,\mathbb{R}^N)\colon \hbox{div}\,{\bf z}\,\hbox{ is a locally finite Radon measure}\}.
$$
\begin{remark}
\label{R1}
Definition \ref{zDu} is correct even if we assume ${\bf z}\in X_{\mathcal{M}_{ \rm loc}}(\Omega)$ and $u\in BV_{\rm loc}(\Omega)$. In this case, Proposition \ref{Dist} implies that 
$({\bf z}, Du)$ is a Radon measure on $\Omega$ and 
$$
\left|\int_{B}({\bf z}, Du)\right|\leq \int_{B} |({\bf z}, Du)|\leq \|{\bf z}\|_{L^\infty(A)}\int_{B}|Du|,
$$
for all Borel sets $B$ and for all opens sets $\omega$ such that $B\subset \omega\subset\subset \Omega$.
\end{remark}
\begin{proposition} 
\label{uz}
Let ${\bf z}\in X_{\mathcal{M}_{loc}}(\Omega)$ and $u\in BV(\Omega)$ such that $u^*\in L^1(\Omega, \hbox{div}\,{\bf z})$. Then 
$$
\hbox{div}\,(u{\bf z})=u^*\,\hbox{div}\,{\bf z}+({\bf z}, Du)\quad\hbox{ as Radon measures in}\; \Omega.
$$
Furthermore, if $u$ also belongs to $L^\infty(\Omega)$, $u{\bf z}\in X_{\mathcal{M}}(\Omega)$ and the following Green formula holds:
$$
\int_{\Omega}u^*\hbox{div}\,{\bf z}+\int_{\Omega}({\bf z}, Du)=\int_{\partial\Omega}[u{\bf z},\nu]\,d\mathcal{H}^{N-1}.
$$
\end{proposition}
\dem Let $(\rho_n)$ be the mollifiers defined as in \eqref{mll} with $\epsilon=1/n$ and let the truncation $T_k\colon\mathbb{R}\to \mathbb{R}$ defined as in \eqref{Tk}. Denote 
$$
u_n^{(k)}=T_k(u)\ast \rho_n\quad\hbox{ for all }\; k,n\in \mathbb{N}.
$$
Note that, for $\varphi\in C^\infty_c(\Omega)$,
$$
-\int_{\Omega}u_n^{(k)}{\bf z}\cdot\nabla\varphi\,dx=\int_{\Omega}\varphi u_n^{(k)}\,\hbox{div}\,{\bf z}+\int_{\Omega}\varphi {\bf z}\cdot\nabla u_n^{(k)}\,dx.
$$
Letting $n\to \infty$ above, using $u_n^{(k)}(x)\to T_k(u(x))$ $\hbox{div}\,{\bf z}-$a.e. $x\in \hbox{supp}\,{\varphi}$, we obtain 
$$
-\int_{\Omega}T_k(u){\bf z}\cdot\nabla \varphi\,dx=\int_{\Omega}\varphi (T_k(u))^*\,\hbox{div}\,{\bf z}+\int_{\Omega}\varphi ({\bf z}, DT_k(u)).
$$
Hence, now letting $k\to\infty$, by \cite[Proposition 2.20]{OrtizPetitta2024}, we obtain
\begin{equation}
\label{Id1}
-\int_{\Omega}u{\bf z}\cdot \nabla \varphi\,dx=\int_{\Omega}\varphi u^*\,\hbox{div}\,{\bf z}+\int_{\Omega}\varphi({\bf z}, Du).
\end{equation}
Since $({\bf z}, Du)$ is a finite Radon measure in $\Omega$ and $u^*\in L^1(\Omega, \hbox{div}\,{\bf z})$,  we can infer from \eqref{Id1} that $\hbox{div}\,(u{\bf z})$ is finite Radon measure in $\Omega$ and
$$
\hbox{div}\,(u{\bf z})=u^*\hbox{div}\,{\bf z}+({\bf z}, Du)\quad\hbox{ as Radon measures in }\; \Omega.
$$
Moreover, if we assume that $u\in L^\infty(\Omega)$, then $u{\bf z}\in X_{\mathcal{M}}(\Omega)$ and so, the following Green formula holds:
$$
\int_{\Omega}u^*\hbox{div}\,{\bf z}+\int_{\Omega}({\bf z}, Du)=\int_{\partial\Omega}[u{\bf z},\nu]\,d\mathcal{H}^{N-1}.
$$
\fim

\section{Main assumptions and results}
\label{main}

The main result  of this section  deals with the existence of a solution of  problem 
\begin{equation}
\label{P} 
\left\{
\begin{array}{rclr}
-\Delta_1u&=&\frac{\lambda}{|x|} s(x)+\frac{f}{u^\gamma}&\quad \mbox{ in } \Omega,\\
u&=&0 &\quad \mbox{ on }\partial\Omega,
\end{array}
\right.
\end{equation} 
where  $ \Delta_1u=\mbox{div}\,\left(\frac{Du}{|Du|}\right)$ and $\Omega$ is a bounded, open set of $\mathbb{R}^N$ with Lipschitz boundary $\partial\Omega$ containing the origin, $\lambda<N-1$, $\gamma> 0$,  $f$ is a nonnegative function in $L^{N,\infty}(\Omega)$, the measurable selection $s(x)\in \mbox{Sgn}(u(x))$ a.e. $x\in\Omega$ with set-valued sign function $\hbox{Sgn}\colon \mathbb{R}\to \mathcal{P}({\mathbb{R}})$ defined as
\begin{equation}
\hbox{Sgn}\,(s)=\left\{
\begin{array}{lcr}
\frac{s}{|s|}& \hbox{if}& s\neq0\\
{[-1,1]}&\hbox{if}& s=0.
\end{array}
\right.
\end{equation}

We first stress that, as $Du/|Du|$ loses its  meaning when $Du=0$, a precise definition of solution for \eqref{P} is needed; this notion  will be given in accordance with the one  introduced in \cite{acm2001, AndreuMazonMollCaselles2004, AndreuCasellesMazon2004, DeCiccoGiachettiSegura2019,dgop}. {We would like to emphasize that a condition similar to $\lambda<N-1$ was established in \cite{WangWillem}}.

Moreover, problem \eqref{P} represents a   borderline case with respect to the  previous works  (\cite{DeCiccoGiachettiSegura2019, LatorreOlivaPetittaSegura2021}) due to the presence of the perturbation term 
$$
\frac{\lambda}{|x|}s(x),
$$
where $1/|x|$ belongs to Marcinkiewicz (or Lorentz) space $L^{N,\infty}(\Omega)$ but not to $L^N(\Omega)$. The main feature of this work is to analyse the influence of this perturbation term on the existence of the solutions for \eqref{P} and, in the spirit of \cite{DeCiccoGiachettiSegura2019,dgop}, the relation with the  regularizing effect provided by the term ${f}{u^{-\gamma}}$.  

\begin{definition}
\label{DS}
Let $f\in L^{N,\infty}(\Omega)$, $f\geq 0$.  We say that a nonnegative function $u\in BV(\Omega)$  is a solution to \eqref{P} if there exist $ {\bf z}\in X_{\mathcal{M}_{loc}}(\Omega)$ with $\|{\bf z}\|_{L^\infty(\Omega)}\leq 1$ such that
\begin{itemize}
\item[$(a)$] $ \frac{f}{u^\gamma}\in L^1_{loc}(\Omega)$;
\item[$(b)$] $ \chi_{\{u>0\}}\in BV_{loc}(\Omega)$;
\item[$(c)$] $-\chi_{\{u>0\}}^*\hbox{div}\,{\bf z}=\frac{\lambda}{|x|}\chi_{\{u>0\}}+\frac{f}{u^\gamma}\quad\hbox{ in }\; \mathcal{D}'(\Omega)$;
\item[$(d)$]   $ ({\bf z}, Du)=|Du|,$  as Radon measures on $\Omega$; 
\item[$(e)$]  {for every $0\leq \phi\in C^1(\overline{\Omega})\cap L^1(\Omega, {\rm div}\,{\bf z})$ with $ \sup_{\Omega} \phi\leq\phi\lfloor_{\partial\Omega}$, and for $\mathcal{H}^{N-1}$-a.e. $x\in \partial\Omega$, one of the following holds:
$$
\lim_{\rho\downarrow0}\rho^{-N}\int_{\Omega\cap B_\rho(x)} u(y)\,dy=0\qquad\hbox{ or }\qquad \phi(x)+[\phi {\bf z},\nu](x)=0 .
$$}
 
\end{itemize}
\end{definition}

{
\begin{remark}
\label{Rnontriv}
Condition $(a)$ deserves to be emphasized: it is not a technical requirement but the very reason why no smallness assumption on $f$ is needed here. Indeed, if $|\{u=0\}\cap\{f>0\}|>0$ then $fu^{-\gamma}=+\infty$ on a set of positive measure and $(a)$ fails; hence {\it every} solution in the sense of Definition \ref{DS} satisfies $u>0$ a.e. in $\{f>0\}$ and, in particular, it is non-trivial as soon as $f\not\equiv0$. This is in sharp contrast with the case $\gamma=0$, where the trivial function solves the problem for all sufficiently small data (see \cite{ct, mst, OrtizPetitta2024}).
\end{remark}
}

\begin{remark}
\label{RTk}
Regarding condition $(e)$ above, it is worth noting that, as noted in \cite[Remark 4.6]{DeCiccoGiachettiSegura2019}, since ${\bf z}\in X_{\mathcal{M}_{loc}}(\Omega)$, then not necessarily $[{\bf z}, \nu]\in L^1(\partial\Omega)$ and so, the usual weak trace condition, i.e., $[{\bf z}, \nu]\in \hbox{Sgn}(-u)$ a.e. in $\partial\Omega$ does not need to hold; to get rid of this fact the weaker  condition $(e)$  is requested and  this  is in accordance with \cite{DeCiccoGiachettiSegura2019, dgop}. 
\end{remark}

\begin{remark}
In case of a positive datum $f$ Definition \ref{DS} looks cleaner: 
 
Let $0<f\in L^{N,\infty}(\Omega)$. We say that a nonnegative function $u\in BV(\Omega)$ is a solution to \eqref{P} if there exist {${\bf z}\in X_{\mathcal{M}}(\Omega)$}  with $\|{\bf z}\|_{L^\infty(\Omega)}\leq 1$ and {${\rm div}\,{\bf z}\in L^1(\Omega)$} such that
\begin{itemize}
\item[$(1)$] ${\frac{f}{u^\gamma}\in L^1(\Omega)}$;
\item[$(2)$] $- \hbox{div}\,{\bf z}=\frac{\lambda}{|x|}+\frac{f}{u^\gamma}\quad\hbox{ in }\; \mathcal{D}'(\Omega)$;
\item[$(3)$]   $ ({\bf z}, Du)=|Du|,$  as Radon measures on $\Omega$; 
\item[$(4)$] {for $\mathcal{H}^{N-1}$-a.e. $x\in\partial\Omega$ one of the following holds:
$$
\lim_{\rho\downarrow0}\rho^{-N}\int_{\Omega\cap B_\rho(x)}u(y)\,dy=0\qquad\hbox{ or }\qquad [{\bf z}, \nu](x)=-1 .
$$
}
 
\end{itemize}

\end{remark}

The  main result of this paper   is the following: 
\begin{theorem} 
\label{mainTh}
Let $\gamma>0$ and $\lambda<N-1.$  Suppose that $f$ is a nonnegative function in $L^{N,\infty}(\Omega)$.  Then there exists a nonnegative solution $u$ to \eqref{P} in the sense of Definition \ref{DS}. {Moreover:
\begin{itemize}
\item[$(i)$] $u>0$ a.e. in $\{f>0\}$; in particular, $u\not\equiv0$ whenever $f\not\equiv0$;
\item[$(ii)$] $u\in L^\infty(\Omega)$ and
$$
\|u\|_{L^\infty(\Omega)}\leq\left(\frac{(N-1)\,\zeta_N\,\|f\|_{L^{N,\infty}(\Omega)}}{N-1-\lambda^+}\right)^{1/\gamma},
$$
where $\zeta_N$ is the optimal constant given in \eqref{bestl}.
\end{itemize}}
\end{theorem}

{

\begin{remark}
\label{RemNeg}
No restriction is needed when $\lambda\leq0$: all our assumptions only involve $\lambda^+$ and are therefore automatically satisfied in this case, in which the term $\lambda|x|^{-1}s(x)$ acts as an absorption term. The relevant case is thus $0<\lambda<N-1$, in which the Hardy term competes with the diffusion.
\end{remark}
}
{Since the argument is somewhat involved, let us describe beforehand the structure of the proof:
\begin{itemize}
\item[$1.$] in Section \ref{ApproxP} we study the approximating $p$-Laplace problems \eqref{Pp} and we collect the estimates on their solutions, which are uniform with respect to $p\in(1,\bar p]$;
\item[$2.$] in Section \ref{PT} we let $p\to 1^+$: this produces the vector field ${\bf z}$ and the limit function $u$, and we check conditions $(a)$--$(e)$ of Definition \ref{DS} (Lemmas \ref{ConvZ}--\ref{chiSol}). If $0<\gamma\leq1$, the estimates of Proposition \ref{S} also give $u\in BV(\Omega)$, so that the existence part of Theorem \ref{mainTh} is already proved in this case;
\item[$3.$] if $\gamma>1$, the estimates only give $u\in BV_{\rm loc}(\Omega)$ together with $u^\gamma\in BV(\Omega)$. The finiteness of the energy of $u$ itself is obtained in Section \ref{PosDat}, first for a positive datum (Theorem \ref{FE}) and then, by approximating $f$ with $f+1/n$, for a general nonnegative one;
\item[$4.$] Section \ref{Br} is devoted to the $L^\infty$ bound $(ii)$, and the proof of Theorem \ref{mainTh} is completed at its end;
\item[$5.$] finally, Section \ref{Opt} shows that the restriction $\lambda<N-1$ cannot be improved.
\end{itemize}}

\section{{The approximating $p$-Laplace problems}}
\label{ApproxP}

{Let    $1<p<N$. We will prove Theorem \ref{mainTh} as a consequence  of the study of the asymptotic behaviour of the solutions to   problems
\begin{equation}
\label{Pp}
\left\{
\begin{array}{rclr}
-\Delta_p u&=&\dys \frac{\beta}{|x|^p}|u|^{p-2}u+\frac{f}{u^\gamma}&\mbox{ in  }  \Omega,\\
u&=&0&\mbox{ on }\partial\Omega,
\end{array}
\right.
\end{equation}
as $p$ go to $1^+$ for some ${  \beta}<\left(\frac{N-p}{p}\right)^p$, $\gamma>0$ and, for the moment, let us restrict to the case $f$ is in $L^m(\Omega)$ with $m=\left(\frac{p^*}{1-\gamma}\right)'$. In order to prove the existence of solutions of  \eqref{P} we  first need to prove the existence of solutions to the problem  \eqref{Pp}  in sense of the concept of solution given in \cite{BoccardoOrsina2010,DCA} ; in fact we will show that there exists  a nonnegative function $u\in W^{1,p}_0(\Omega)$ if $0<\gamma\leq1$, and $u\in W^{1,p}_{\rm loc}(\Omega)$ with $u^\frac{\gamma+p-1}{p}\in W^{1,p}_0(\Omega)$ if $\gamma>1$, such that 
\begin{align}
\nonumber&\forall \omega\subset\subset\Omega\quad\exists c_\omega>0\;\hbox{ s.t.}\quad u\geq c_\omega\quad\hbox{ in }\quad\omega;\\
\label{WSp}
&\int_{\Omega}|\nabla u|^{p-2}\nabla u\cdot\nabla \varphi\,dx=\beta\int_{\Omega}\frac{1}{|x|^p}|u|^{p-2}u\varphi\,dx+\int_{\Omega}\frac{f}{u^\gamma}\varphi\,dx\quad\hbox{ for all }\;\varphi\in C^\infty_c(\Omega).
\end{align}
We also need find uniform estimates of the family of solutions $(u_p)_{p>1}$, each solution $u_p$ obtained of \eqref{Pp}. In order to do so, for each $n\in \mathbb{N}$, we consider  the auxiliary Dirichlet problem
\begin{equation}
\label{Ppn}
\left\{
\begin{array}{rclr}
\dys -\Delta_p u&=& \dys \frac{\beta}{|x|^{p} +1/n}|u|^{p-2}u+\frac{f_n}{(u+1/n)^\gamma}&\mbox{ in  }  \Omega,\\
u&=&0&\mbox{ on }\partial\Omega,
\end{array}
\right.
\end{equation}
where $f_n=T_n(f)$.

 In order to show the existence of solutions to \eqref{Ppn}, for each $v\in L^p(\Omega)$, we shall prove that there exists $w_n\in W^{1,p}_0(\Omega)$ solution to the problem
\begin{equation}
\label{PpnV}
\left\{
\begin{array}{rclr}
-\Delta_p w_n&=&\dys \frac{\beta}{|x|^{p} +1/n}|w_n|^{p-2}w_n+\frac{f_n}{(|v|+1/n)^\gamma}&\mbox{ in  }  \Omega,\\
w_n&=&0&\mbox{ on }\partial\Omega.
\end{array}
\right.
\end{equation}
But this problem can be formulate as the following minimization problem
\begin{equation}
\label{Min}
M_n=\min_{w\in W^{1,p}_0(\Omega)}\left\{\frac{1}{p}\int_{\Omega}|\nabla w|^p\,dx-\frac{\beta}{p}\int_{\Omega}\frac{1}{|x|^p+1/n}|w|^p\,dx-\int_{\Omega}\frac{f_nw}{(|v|+1/n)^\gamma}\,dx\right\}     .
\end{equation}

Let $(w_{n, m})_{m\in\mathbb{N}}$ be a minimizing sequence to \eqref{Min}. Use Hardy's, H\"older's and Poincar\'e's inequalities to obtain,
\begin{equation}\label{38}
M_n+o_m(1)\geq \frac{1}{p}\left[1-{\beta^+}\left(\frac{p}{N-p}\right)^p\right]\int_{\Omega}|\nabla w_{n,m}|^p\,dx-n^{\gamma+1} C|\Omega|^\frac{p-1}{p}\left(\int_{\Omega}|\nabla w_{n,m}|^p\,dx\right)^\frac{1}{p},
\end{equation}
Hence,  $n\in \mathbb{N}$ fixed, $(w_{n,m})_{m\in\mathbb{N}}$ is a bounded sequence in $W^{1,p}_0(\Omega)$. Consequently, there exists $w_n\in W^{1,p}_0(\Omega)$ such that, up to subsequences, 
\begin{eqnarray}
\label{Conv}
w_{m, n}\rightharpoonup w_n &\hbox{ in }\; W^{1,p}_0(\Omega)\\
\nonumber w_{m, n}\to w_n&\hbox{ in }\; L^q(\Omega)\; \hbox{ for all } \; q\in [p, p^*),
\end{eqnarray}
as $m\to \infty$.

Note that, by \eqref{Conv}, one has that
\begin{align}
\nonumber M_n&=\liminf_{m\to \infty}\left(\frac{1}{p}\int_{\Omega}|\nabla w_{n, m}|^p\,dx-\frac{\beta}{p}\int_{\Omega}\frac{1}{|x|^p+1/n}|w_{n,m}|^p\,dx-\int_{\Omega}\frac{f_nw_{n,m}}{(|v|+1/n)^\gamma}\,dx\right)\\
&\geq \frac{1}{p}\int_{\Omega}|\nabla w_n|^p\,dx-\frac{\beta}{p}\int_{\Omega}\frac{1}{|x|^p+1/n}|w_n|^p\,dx-\int_{\Omega}\frac{f_nw_n}{(|v|+1/n)^\gamma}\\
\nonumber&\geq M_n,
\end{align}
and so, since
$$
W^{1,p}_0(\Omega)\ni u\mapsto \frac{1}{p}\int_{\Omega}|\nabla u|^p\,dx-\frac{\beta}{p}\int_{\Omega}\frac{1}{|x|^p+1/n}|u|^p\,dx-\int_{\Omega}\frac{f_nu}{(|v|+1/n)^\gamma}\,dx
$$
is Frechet differentiable, $w_n$ satisfies the following identity
\begin{equation}
\label{WSnV}
\int_{\Omega}|\nabla w_n|^{p-2}\nabla w_n\cdot\nabla \varphi\,dx=\beta\int_{\Omega}\frac{1}{|x|^p+1/n}|w_n|^{p-2}w_n\varphi\,dx+\int_{\Omega}\frac{f_n\varphi}{(|v|+1/n)^\gamma}\,dx\quad\hbox{ for all }\; \varphi\in W^{1,p}_0(\Omega),
\end{equation}
that is, $w_n$ is a solution to \eqref{PpnV}. 

Use $w_n$ as test function to obtain 
$$
\int_{\Omega}|\nabla w_n|^p\,dx=\beta\int_{\Omega}\frac{1}{|x|^p+1/n}|w_n|^p\,dx+\int_{\Omega}\frac{f_nw_n}{(|v|+1/n)^\gamma}\,dx.
$$
Hence, by Hardy, H\"older and Poincar\'e inequalities, 
$$
\|w_n\|_{L^p(\Omega)}\leq C(\Omega,p)\left[\frac{C(\Omega,p) n^{\gamma+1}|\Omega|^{(p-1)/p}}{1-{\beta^+}\left(\frac{p}{N-p}\right)^p}\right]^{1/(p-1)}=: r(\beta, N,\gamma, p, |\Omega|, n)=r.
$$
{
Let us now observe that every solution of \eqref{PpnV} is nonnegative. Indeed, using $w_n^- $ as test function in \eqref{WSnV} we obtain
$$
-\int_{\Omega}|\nabla w_n^-|^p\,dx+\beta\int_{\Omega}\frac{|w_n^-|^p}{|x|^p+1/n}\,dx=\int_{\Omega}\frac{f_nw_n^-}{(|v|+1/n)^\gamma}\,dx\geq0.
$$
If $\beta\leq0$ this forces $w_n^-=0$ a.e. in $\Omega$, while if $0<\beta<\left(\frac{N-p}{p}\right)^p$ Hardy's inequality yields
$$
\beta\int_{\Omega}\frac{|w_n^-|^p}{|x|^p}\,dx\geq\beta\int_{\Omega}\frac{|w_n^-|^p}{|x|^p+1/n}\,dx\geq\int_{\Omega}|\nabla w_n^-|^p\,dx\geq\left(\frac{N-p}{p}\right)^p\int_{\Omega}\frac{|w_n^-|^p}{|x|^p}\,dx,
$$
which is impossible unless $w_n^-=0$ a.e. in $\Omega$. In both cases
\begin{equation}
\label{signw}
w_n\geq0\qquad\hbox{ a.e. in }\;\Omega.
\end{equation}

The solution of \eqref{PpnV} is, in fact, unique. Indeed, in view of \eqref{signw}, problem \eqref{PpnV} reads as $-\Delta_pw=h_n(x,w)$, where
$$
h_n(x,\sigma)=\frac{\beta\,\sigma^{p-1}}{|x|^p+1/n}+\frac{f_n(x)}{(|v(x)|+1/n)^\gamma},\qquad \sigma>0,
$$
and, since the second summand is nonnegative and $p>1$,
$$
\frac{h_n(x,\sigma)}{\sigma^{p-1}}=\frac{\beta}{|x|^p+1/n}+\frac{f_n(x)}{(|v(x)|+1/n)^\gamma}\,\sigma^{1-p}\qquad\hbox{ is nonincreasing with respect to }\;\sigma>0 .
$$
This is precisely the structure condition under which the uniqueness of the nonnegative solution holds, by the Brezis--Oswald theorem \cite{BrezisOswald} for $p=2$ and by the D\'iaz--Saa inequality \cite{DiazSaa} for a general $p>1$. Let us stress that the Hardy term is harmless here: being $(p-1)$-homogeneous, it cancels out in the D\'iaz--Saa argument. 
}

Then we  can consider the application $S_n\colon B_r(0)\subset L^p(\Omega)\to B_r(0)$ defined as $S_n(v)=w_n$, which, in turn, is invariant in $B_r(0)$. Moreover, it is not difficult to show that $S_n$ is both continuous and compact on $L^p(\Omega)$. Hence, by Schauder's fixed point theorem there exists $u_n$ in $W^{1,p}_0(\Omega)$ such that $S_n(u_n)=u_n$. In other  words $u_n$ solves 
\begin{equation}
\label{AuxP}
\left\{
\begin{array}{rclc}
-\Delta_ p u&=& \dys \frac{\beta}{|x|^p+1/n}|u|^{p-2}u+\frac{f_n}{(|u|+1/n)^\gamma}&\hbox{ in }\; \Omega\\
           u&=&0&\hbox{ on }\;\partial\Omega.
\end{array}
\right.
\end{equation}

{Since $u_n=S_n(u_n)$ solves \eqref{PpnV} with $v=u_n$, estimate \eqref{signw} gives $u_n=|u_n|\geq0$ a.e. in $\Omega$.} Consequently, one can drop  the absolute value of $u_n$ in \eqref{AuxP}, i.e.,
\begin{equation}
\label{Un}
\int_{\Omega}|\nabla u_n|^{p-2}\nabla u_n\cdot \nabla\varphi\,dx=\beta\int_{\Omega}\frac{|u_n|^{p-2}u_n\varphi}{|x|^p+1/n}\,dx+\int_{\Omega}\frac{f_n\varphi}{(u_n+1/n)^\gamma}\,dx\quad\hbox{ for all }\,\varphi\in C^1_c(\Omega).
\end{equation}
Moreover, if $p<N<\frac{p(2p-1)}{p-1}$, then the solutions $u_n$ belongs to $L^\infty(\Omega)$. Indeed, use $G_k(u_n)$ as test function in \eqref{AuxP} to obtain 
\begin{align}
\int_{\Omega}|\nabla G_k(u_n)|^p\,dx&=\beta\int_{\Omega}\frac{u_n^{p-1}G_k(u_n)}{|x|^p+1/n}\,dx+\int_{\Omega}\frac{f_nG_k(u_n)}{(u_n+1/n)^\gamma}\,dx\\
&\leq \beta^+ n \|u_n\|_{L^r(\Omega)}^{p-1}\|G_k(u_n)\|_{L^{p^*}(\Omega)}|A(k)|^{1-\frac{p-1}{r}-\frac{1}{p^*}}+n^{\gamma+1}\|G_k(u_n)\|_{L^{p^*}(\Omega)}|A(k)|^{1-\frac{1}{p^*}}\\
&\leq n(\beta^+\|u_n\|_{L^r(\Omega)}^{p-1}+n^\gamma)\|G_k(u_n)\|_{L^{p^*}(\Omega)}|A(k)|^{1-\frac{p-1}{r}-\frac{1}{p^*}},
\end{align}
since $|A(k)|\leq k^{-p^*}\int_{\Omega}u_n^{p^*}\,dx<1$ for $k\geq k_0$. Using the Sobolev inequality on the left-side hand above, we obtain
$$
S_{N,p}^{-p}\|G_k(u_n)\|_{L^{p^*}(\Omega)}^{p-1}\leq n(\beta^+\|u_n\|_{L^r(\Omega)}^{p-1}+n^\gamma)|A(k)|^{1-\frac{p-1}{r}-\frac{1}{p^*}}
$$
But this, for $h\geq k\geq k_0$, implies that
$$
|A(h)|\leq\frac{[n(\beta^+ \|u_n\|^{{p-1}}_{L^r(\Omega)}+n^\gamma)]^\frac{p^*}{p-1}}{(h-k)^{p^*}}|A(k)|^{\frac{p^*}{p-1}(1-\frac{p-1}{r}-\frac{1}{p^*})}
$$
Let $r>\frac{N(p-1)}{p}$. Note that 
$$
\frac{p^*}{p-1}\left(1-\frac{p-1}{r}-\frac{1}{p^*}\right)>1
$$
Thus, by \cite[Lemma 4.1]{Stampacchia1965}, one has
$$
|A(k_0+d)|=0, 
$$
where
$$
d=[n(\beta^+ \|u_n\|^{{p-1}}_{L^r(\Omega)}+n^\gamma)]^\frac{1}{p-1} |A(k_0)|^{\frac{p^*-p}{p^*(p-1)}-\frac{1}{r}} 2^{{\frac{\frac{p^*}{p-1}\left(1-\frac{p-1}{r}-\frac{1}{p^*}\right)}{\frac{p^*-p}{p-1}-\frac{p^*}{r}}}}.
$$
Hence
\begin{equation}
\label{Unb}
u_n\in L^\infty(\Omega).
\end{equation}
Note that, if $\beta^+=0$, one can prove that $u_n$ belongs to $L^\infty(\Omega)$ for any $p<N$. {Let us also observe that the restriction $p<N<\frac{p(2p-1)}{p-1}$, which is nothing but $p^*>2p-1$, is not restrictive for our purposes: for every fixed $N\geq2$ it is satisfied for all $p>1$ close enough to $1$, since $\frac{p(2p-1)}{p-1}\to+\infty$ as $p\to1^+$, and we shall only let $p\to1^+$ (see \eqref{RKey} below).} Consequently, by standard regularity theory, (see for instance \cite[Theorem 6.2.7]{GasinskiPapageorgiou2006}), one has that $u_n\in C^{1,\alpha}_0(\overline{\Omega})$ for some  $\alpha\in (0,1)$. Hence, by the strong maximum principle 
\begin{equation}
\label{UnP}
u_n>0\quad\hbox{ in }\; \Omega\quad\hbox{ for all }\; n\in \mathbb{N}.
\end{equation}
On the other hand, reasoning as in  \cite[Lemma 4.2]{AbdellaouiPeral2003}, it follows that
$$
u_{n+1}\geq u_n \quad\hbox{ in }\;\Omega\quad\hbox{ for all }\; n\in \mathbb{N}.
$$
Hence, for each $\omega\subset \subset \Omega$, there exists $c_\omega>0$ such that
\begin{equation}
\label{UnC}
u_n\geq c_\omega\qquad\hbox{ in }\; \omega\quad\hbox{ for all }\;n\in \mathbb{N}.
\end{equation}
 
}

\subsection{Estimates in the mildly singular case}
{In order to show the existence of solution and to find uniform estimates of the family of solutions $(u_p)_{p>1}$, each $u_p$ obtained of \eqref{Pp}, for the case $0<\gamma\leq1$, the following result is established.}
\begin{proposition}
\label{S}
 Let $0<\gamma\leq 1$ and  $\beta<((N-p)/p)^p$.  Let $f\in L^m(\Omega)$, $m=\left(\frac{p^*}{1-\gamma}\right)'$,  be a nonnegative function. Then there exists a solution to \eqref{Pp} {in sense of \eqref{WSp} and the following holds: if $ 0<\gamma<1 $,
 \begin{equation}
 \label{E6}
 \int_{\Omega}|\nabla u_p|^p\,dx\leq \left[\frac{\|f\|_{L^m(\Omega)}S_{N,p}^{1-\gamma}}{1-\beta^+(p/(N-p))^p}\right]^{p/(p-1+\gamma)}
 \end{equation}
 and if $\gamma=1$,
 \begin{equation}
 \label{E7}
 \int_{\Omega}|\nabla u_p|^p\,dx\leq \frac{\|f\|_{L^1(\Omega)}}{1-\beta^+(p/(N-p))^p}.
 \end{equation}
 }
\end{proposition}
\dem Use $u_n$ as test function in \eqref{Un} to obtain
\begin{equation}
\int_{\Omega}|\nabla u_n|^p\,dx=\beta \int_{\Omega}\frac{1}{|x|^p+1/n}|u_n|^p\,dx+\int_{\Omega}\frac{f_nu_n}{(u_n+1/n)^\gamma}\,dx.
\end{equation}
Hence, using the Hardy's inequality one has
\begin{equation}
\label{I1}
\left(1-\beta^+\left(\frac{p}{N-p}\right)^p\right)\int_{\Omega}|\nabla u_n|^p\,dx\leq\int_{\Omega}\frac{f_nu_n}{(u_n+1/n)^\gamma}\,dx.
\end{equation}
 
Let $0<\gamma<1$ and   $m=\left(\frac{p^*}{1-\gamma}\right)'$. Note that, using H\"older inequality, we have that 
\begin{equation}
\label{Est3}
\int_{\Omega}\frac{f_nu_n}{(u_n+1/n)^\gamma}\,dx\leq \|f_n\|_{L^m(\Omega)}  \left(\int_{\Omega}u_n^{p^*}\,dx\right)^{(1-\gamma)/p^*},
\end{equation}
which in \eqref{I1}, using Sobolev inequality on the left-hand side of both estimates, it allows us  to conclude that
\begin{equation}
\label{Est4}
 \left(\int_{\Omega}u_n^{p^*}\,dx\right)^{(1-\gamma)/{p^*}}\leq \left[\frac{\|f\|_{L^m(\Omega)} S_{N,p}^p}{1-\beta^+(p/(N-p))^p}\right]^{(1-\gamma)/(p-1+\gamma)}\quad\hbox{ for all }\, n\in \mathbb{N},
\end{equation}
respectively, where $S_{N,p}$ is the Sobolev constant {defined in \eqref{SobC}}. Thus, from \eqref{I1}, \eqref{Est3} and \eqref{Est4}, we obtain that
\begin{equation}
\label{Est6}
\int_{\Omega}|\nabla u_n|^p\,dx\leq  \left[\frac{\|f\|_{L^m(\Omega)} S_{N,p}^{1-\gamma}}{1-\beta^+(p/(N-p))^p}\right]^{p/(p-1+\gamma)}.
\end{equation}

In the case $\gamma=1$ the estimate is even easier as the right-hand side of  \eqref{I1} is bounded for any $f\in L^1(\Omega)$, {i.e.,
\begin{equation}
\label{Est7}
\int_{\Omega}|\nabla u_n|^p\,dx\leq\frac{\|f\|_{L^1(\Omega)}}{1-\beta^+(p/(N-p))^p}.
\end{equation}
}
Hence $u_n$ is a bounded sequence in $W^{1,p}_0(\Omega)$ and so, by   compact embeddings, there exists $u_p\in W^{1,p}_0(\Omega)$ such that, up to not relabeled subsequences, one has 
\begin{equation}
\label{CUn}
\begin{array}{lcr}
u_n\rightharpoonup u_p\ \ \hbox{ in } W^{1,p}_0(\Omega);\\
u_n\to u_p  \ \ \ \hbox{in}\ \  L^s(\Omega),\; s\in [p,p^*),\quad\hbox{a.e. in }\Omega;
\end{array}
\end{equation}
as {$n\to\infty$}. Letting $n \to \infty$ in \eqref{Est6} and \eqref{Est7} yields \eqref{E6} and \eqref{E7}, respectively. Furthermore, from \eqref{UnC}, for each $\omega \subset\subset \Omega$, there exists $c_\omega > 0$ such that
$$
u_p\geq c_\omega\qquad\hbox{ in }\;\omega.
$$

Furthermore, by  \cite[Theorem 2.1]{BoccardoMurat1992} one has that, up to subsequences, $\nabla u_n \to \nabla u_p$ a.e. in $\Omega$, and so 
\begin{equation}
\label{0}
\lim_{n\to \infty}\int_{\Omega}|\nabla u_n|^{p-2}\nabla u_n\cdot \nabla \varphi\,dx=\int_{\Omega}|\nabla u_p|^{p-2}\nabla u_p\cdot\nabla \varphi\,dx\quad\hbox{ for all }\; \varphi\in C^1_c(\Omega).
\end{equation}
For the sake of completeness let us briefly show the argument:   for $m\ge1$, we define
$$
\Omega_{\frac{1}{m}}=\left\{x\in\Omega\colon \hbox{dist}\,(x,\partial\Omega)>\frac{1}{m}\right\}.
$$  

Let $\varphi\in C_c^\infty(\Omega)$ be such that $\varphi=1$ in $\overline{\Omega}_{\frac{1}{m}}$ and $0\leq \varphi\leq 1$ in $\Omega$. Use $\varphi T_k(u_n-u_p) $ as test function in \eqref{Ppn} to obtain
\begin{align*}
&\int_{\Omega}\varphi|\nabla u_n|^{p-2}\nabla u_n\cdot \nabla T_k(u_n-u_p)\,dx\\
&={\beta}\int_{\Omega}\frac{u_n^{p-1} \varphi T_k(u_n-u_p)}{|x|^p+1/n}\,dx+\int_{\Omega}\frac{f_n\varphi T_k(u_n-u_p)}{(u_n+1/n)^\gamma}\,dx-\int_{\Omega}T_k(u_n-u_p)|\nabla u_n|^{p-2}\nabla u_n\cdot \nabla \varphi\,dx.
 \end{align*}
 
Thus, using  \eqref{CUn} one can show that
\begin{equation}
\label{A}
\int_{\Omega}\varphi|\nabla u_n|^{p-2}\nabla u_n\cdot \nabla T_k(u_n-u_p)\,dx=o_n(1).
\end{equation}

On the other hand, since $T_k(u_n-u_p)\rightharpoonup 0$ in $W^{1,p}_0(\Omega)$, one has
 \begin{equation}
 \label{B}
  \int_{\Omega}\varphi|\nabla u_p|^{p-2}\nabla u_p\cdot \nabla T_k(u_n-u_p)\,dx=o_n(1).
  \end{equation}
  
  Thus, from \eqref{A} and \eqref{B}, for $0<\theta<1$, 
  $$
  \int_{\Omega_{\frac{1}{m}}}[(|\nabla u_n|^{p-2}\nabla u_n-|\nabla u_p|^{p-2}\nabla u_p)\cdot (\nabla u_n-\nabla u_p)]^\theta\,dx=o_n(1).
  $$
  But this, in turn, implies that
 \begin{align}
 \label{Dg}
 \int_{\Omega_{\frac{1}{m}}}|\nabla u_n-\nabla u_p|^\theta\,dx&\leq \left(\int_{\Omega_{\frac{1}{m}}}\frac{|\nabla u_n-\nabla u_p|^{2\theta}}{(|\nabla u_n|+|\nabla u_p|)^{(2-p)\theta}}\,dx\right)^{\frac{1}{2}}\left(\int_{\Omega_{\frac{1}{m}}}(|\nabla u_n|+|\nabla u_p|)^{(2-p)\theta}\,dx\right)^{\frac{1}{2}}\\
 \nonumber&\leq \left(\int_{\Omega_{\frac{1}{m}}}[(|\nabla u_n|^{p-2}\nabla u_n-|\nabla u_p|^{p-2}\nabla u)\cdot(\nabla u_n-\nabla u_p)]^\theta\,dx\right)^{\frac{1}{2}}\left(\int_{\Omega_{\frac{1}{m}}}(|\nabla u_n|+|\nabla u_p|)^{(2-p)\theta}\,dx\right)^{\frac{1}{2}}\\
 \nonumber&=o_n(1).
 \end{align}
 
Consequently, from \eqref{Dg}, using a diagonal argument, we have, up to subsequences, 
 $$
 |\nabla u_n-\nabla u_p|^\theta \to 0\quad\hbox{a.e. in }\;\Omega.
 $$
Moreover, since $(\nabla u_n)$ is a bounded sequence in $(L^p(\Omega))^N$, by Vitali's theorem, one has 
$$
\nabla u_n\rightarrow \nabla u_p\quad\hbox{ in }\; (L^q(\Omega))^N\;\quad\hbox{ for all}\;q<p,
$$
as $n\to \infty$, and so, \eqref{0} holds.

In order  to prove the convergences of the terms of right-hand  side of \eqref{Un}, we shall use Lebesgue's dominated convergence theorem. 

{\bf Claim 1.} 
\begin{equation}
\label{1}
\lim_{n\to \infty}\int_{\Omega}\frac{1}{|x|^p+1/n}|u_n|^{p-2}u_n\varphi\,dx=\int_{\Omega}\frac{1}{|x|^p}|u_p|^{p-2}u_p\varphi\,dx.
\end{equation}

Since $u_n\to u_p$ in $L^q(\Omega)$ for all $q\in [p, p^*)$, up to subsequences, 
\begin{align}
&u_n(x)\to u_p(x)\quad\hbox{ a.e. }\; x\in \Omega,\\
\exists\; g\in L^q(\Omega) \;\hbox{ s.t. }&|u_n(x)|\leq g(x)\quad\hbox{ a.e. }\; x\in \Omega.
\end{align}
Thus, if we consider $1<s<\frac{N}{(p-1)(N-1)}$ and $\frac{N}{N-(p-1)(N-1)}<s'<\frac{N}{p}$, by Young's inequality, then
$$
\left|\frac{1}{|x|^p+1/n}|u_n|^{p-2}u_n\varphi\right|\leq \|\varphi\|_{\infty}\left(\frac{1}{s'}
\frac{1}{|x|^{ps'}}+\frac{1}{s}g^{(p-1)s}\right)\in L^1(\Omega).
$$
Consequently, by dominated convergence theorem, we obtain
$$
\lim_{n\to \infty}\int_{\Omega}\frac{1}{|x|^p+1/n}|u_n|^{p-2}u_n\varphi\,dx=\int_{\Omega}\frac{1}{|x|^p}|u_p|^{p-2}u_p\varphi\,dx.
$$

{\bf Claim 2.} 
\begin{equation}
\label{2}
\lim_{n\to \infty}\int_{\Omega}\frac{f_n}{(u_n+1/n)^\gamma}\varphi\,dx=\int_{\Omega}\frac{f}{u_p^\gamma}\varphi\,dx.
\end{equation}
Using \eqref{UnC} we have that 
$$
\left|\frac{f_n\varphi}{(u_n+1/n)^\gamma}\right|\leq \frac{|f||\varphi|}{( c_{\hbox{supp}\varphi})^\gamma}  \quad\hbox{ in }\; \Omega.
$$
Hence, by Lebesgue's dominated convergence theorem, Claim 2 holds.

Therefore, letting $n\to \infty$ in \eqref{Un}, from \eqref{0}, Claim 1 and Claim 2,  the limit $u_p$ is a solution to \eqref{Pp} in sense of  \eqref{WSp}.
\fim

\subsection{Estimates in the strongly singular case}
  In order to show the existence of solution for the case $\gamma>1$, the following estimates are found. {In this regime the equation has to be tested with $u^\gamma$ rather than with $u$ itself, and consequently the Hardy term has to be absorbed by a different constant; this leads to the assumption
\begin{equation}
\label{HypG}
\beta^+<\gamma\left(\frac{N-p}{\gamma+p-1}\right)^p ,
\end{equation}
which is discussed in Remark \ref{RemH} below.}

\begin{lemma} 
\label{L1}
Let $u_n$ be the solution of \eqref{Ppn} with $\gamma>1$ and $f$ be a  nonnegative function in $L^1(\Omega)$. {Assume that \eqref{HypG} holds and that $p<N<\frac{p(2p-1)}{p-1}$.}
Then $(u_n)_{n\in\mathbb{N}}$ is bounded in $W^{1,p}_{loc}(\Omega)$ and in $L^s(\Omega)$ with $s=\frac{p^*(\gamma+p-1)}{p}$.
\end{lemma}

{
\begin{remark}
\label{RemH}
Two comments on the assumptions of Lemma \ref{L1} are in order.

$(i)$ Assumption \eqref{HypG} is the natural counterpart, in the strongly singular regime, of the condition $\beta<\left(\frac{N-p}{p}\right)^p$ of Proposition \ref{S}, to which it reduces for $\gamma=1$; it is equivalent to the positivity of the constant appearing in \eqref{EstUE} below.  Observe that 
$$
\gamma\left(\frac{N-p}{\gamma+p-1}\right)^p<\left(\frac{N-p}{p}\right)^p ,
$$
so that \eqref{HypG} is strictly stronger than the assumption of Proposition \ref{S}, and the admissible range of $\beta$ shrinks as $\gamma$ grows. This restriction, however, disappears in the limit $p\to1^+$, which is the only regime we are interested in: indeed
\begin{equation}
\label{LimCost}
\lim_{p\to1^+}\gamma\left(\frac{N-p}{\gamma+p-1}\right)^p=\gamma\cdot\frac{N-1}{\gamma}=N-1\qquad\hbox{ for every }\;\gamma>0,
\end{equation}
independently of $\gamma$, and $N-1$ is precisely the threshold appearing in \eqref{RKey} and in Theorem \ref{mainTh}. See Figure \ref{fig1}.

$(ii)$ The restriction $p<N<\frac{p(2p-1)}{p-1}$ is only needed in order for $u_n^\gamma$ to be an admissible test function, since it guarantees $u_n\in L^\infty(\Omega)$ (see \eqref{Unb}); as observed there, it is automatically satisfied for every fixed $N$ as soon as $p$ is close enough to $1$, and it is void if $\beta^+=0$.
\end{remark}
}

\begin{figure}[ht]
\centering
\begin{tikzpicture}[x=1cm,y=1cm]
  \draw[->] (0,0) -- (10.50,0) node[below right] {\small $p$};
  \draw[->] (0,0) -- (0,6.50);
  \foreach \x/\l in {0.00/1.0, 1.67/1.1, 3.33/1.2, 5.00/1.3, 6.67/1.4, 8.33/1.5, 10.00/1.6} \draw (\x,0) -- (\x,-0.1) node[below] {\small $\l$};
  \foreach \y/\l in {0.33/1, 2.00/2, 3.67/3, 5.33/4} \draw (0,\y) -- (-0.1,\y) node[left] {\small $\l$};
  \draw[gray] (0,5.333) -- (10.00,5.333) node[right,black] {\small $N-1$};
  \draw[very thick] (0.000,5.333) -- (0.333,5.349) -- (0.667,5.361) -- (1.000,5.369) -- (1.333,5.373) -- (1.667,5.373) -- (2.000,5.369) -- (2.333,5.361) -- (2.667,5.349) -- (3.000,5.333) -- (3.333,5.313) -- (3.667,5.289) -- (4.000,5.262) -- (4.333,5.231) -- (4.667,5.197) -- (5.000,5.159) -- (5.333,5.117) -- (5.667,5.073) -- (6.000,5.025) -- (6.333,4.974) -- (6.667,4.920) -- (7.000,4.863) -- (7.333,4.803) -- (7.667,4.740) -- (8.000,4.675) -- (8.333,4.607) -- (8.667,4.537) -- (9.000,4.464) -- (9.333,4.389) -- (9.667,4.313) -- (10.000,4.234);
  \node[right] at (10.05,4.234) {\small $\gamma=1$};
  \draw[thick,dashed] (0.000,5.333) -- (0.333,5.324) -- (0.667,5.311) -- (1.000,5.295) -- (1.333,5.275) -- (1.667,5.252) -- (2.000,5.226) -- (2.333,5.197) -- (2.667,5.164) -- (3.000,5.128) -- (3.333,5.089) -- (3.667,5.047) -- (4.000,5.003) -- (4.333,4.955) -- (4.667,4.904) -- (5.000,4.851) -- (5.333,4.795) -- (5.667,4.737) -- (6.000,4.676) -- (6.333,4.613) -- (6.667,4.547) -- (7.000,4.479) -- (7.333,4.409) -- (7.667,4.338) -- (8.000,4.264) -- (8.333,4.188) -- (8.667,4.111) -- (9.000,4.032) -- (9.333,3.952) -- (9.667,3.870) -- (10.000,3.787);
  \node[right] at (10.05,3.787) {\small $\gamma=2$};
  \draw[thick,dotted] (0.000,5.333) -- (0.333,5.243) -- (0.667,5.151) -- (1.000,5.059) -- (1.333,4.965) -- (1.667,4.871) -- (2.000,4.775) -- (2.333,4.679) -- (2.667,4.582) -- (3.000,4.484) -- (3.333,4.386) -- (3.667,4.288) -- (4.000,4.188) -- (4.333,4.089) -- (4.667,3.989) -- (5.000,3.890) -- (5.333,3.790) -- (5.667,3.690) -- (6.000,3.590) -- (6.333,3.490) -- (6.667,3.390) -- (7.000,3.291) -- (7.333,3.192) -- (7.667,3.093) -- (8.000,2.995) -- (8.333,2.897) -- (8.667,2.800) -- (9.000,2.703) -- (9.333,2.607) -- (9.667,2.512) -- (10.000,2.417);
  \node[right] at (10.05,2.417) {\small $\gamma=5$};
  \draw[thick,dashdotted] (0.000,5.333) -- (0.333,5.083) -- (0.667,4.840) -- (1.000,4.605) -- (1.333,4.377) -- (1.667,4.156) -- (2.000,3.943) -- (2.333,3.736) -- (2.667,3.536) -- (3.000,3.343) -- (3.333,3.156) -- (3.667,2.976) -- (4.000,2.801) -- (4.333,2.632) -- (4.667,2.470) -- (5.000,2.312) -- (5.333,2.161) -- (5.667,2.014) -- (6.000,1.873) -- (6.333,1.737) -- (6.667,1.606) -- (7.000,1.479) -- (7.333,1.357) -- (7.667,1.240) -- (8.000,1.127) -- (8.333,1.018) -- (8.667,0.913) -- (9.000,0.813) -- (9.333,0.716) -- (9.667,0.623) -- (10.000,0.533);
  \node[right] at (10.05,0.533) {\small $\gamma=20$};
\end{tikzpicture}
\caption{The map $p\mapsto\gamma\left(\frac{N-p}{\gamma+p-1}\right)^{p}$ appearing in \eqref{HypG}, for $N=5$ and $\gamma=1,2,5,20$ (the case $\gamma=1$ being the classical Hardy constant $\left(\frac{N-p}{p}\right)^{p}$). For $p>1$ the admissible range of $\beta$ shrinks as $\gamma$ grows, but all the curves converge to the same value $N-1$ as $p\to1^{+}$: this is why the threshold in Theorem \ref{mainTh} does not depend on $\gamma$.}
\label{fig1}
\end{figure}
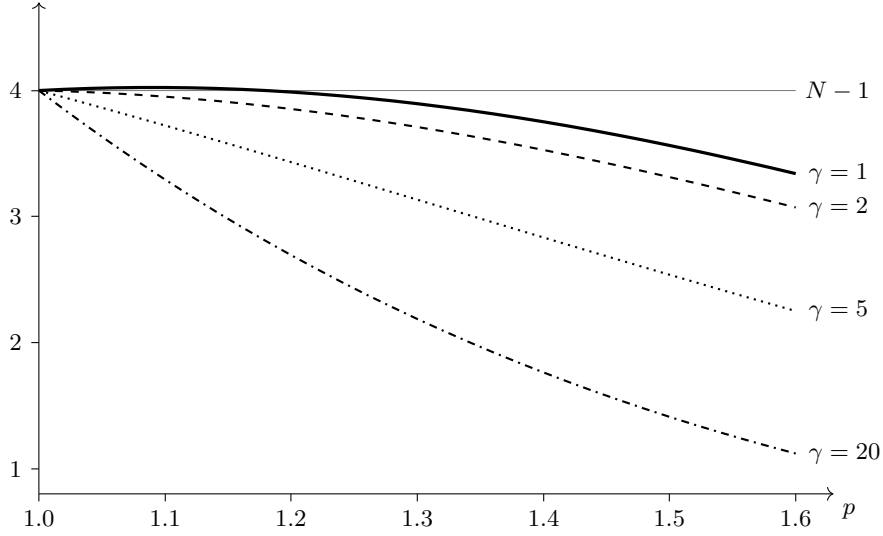

\dem
Use $u_n^\gamma$ as test function in \eqref{Ppn} to obtain 
$$
\gamma \left(\frac{p}{\gamma+p-1}\right)^p\int_{\Omega}|\nabla u_n^{(\gamma+p-1)/p}|^p\,dx=\beta\int_{\Omega}\frac{u_n^{\gamma+p-1}}{|x|^p+1/n}\,dx+\int_{\Omega}\frac{f_nu_n^\gamma}{(u_n+1/n)^\gamma}\,dx.
$$
Hence, by Hardy's inequality, we get
\begin{equation}
\label{EstUE}
{\left(\frac{p}{N-p}\right)^p\left[\gamma \left(\frac{N-p}{\gamma+p-1}\right)^p-\beta^+\right]}\int_{\Omega}|\nabla u_n^{(\gamma+p-1)/p}|^p\,dx\leq \int_{\Omega}f\,dx,
\end{equation}
 {observe  that the constant on the left-hand side is positive   because of \eqref{HypG};  this constant is nothing but $\gamma\left(\frac{p}{\gamma+p-1}\right)^p-\beta^+\left(\frac{p}{N-p}\right)^p$, and it will be used in this second form in the computations below.} But \eqref{EstUE}  implies that, by Sobolev inequality, $(u_n)_{n\in\mathbb{N}}$ is bounded in $L^s(\Omega)$.

In order to prove that $(u_n)_{n\in\mathbb{N}}$ is bounded in $W^{1,p}_{loc}(\Omega)$, observe that, by the H\"{o}lder and the Sobolev inequalities and using \eqref{EstUE}, one has
\begin{align}
\label{Et2}
\nonumber\int_{\Omega}u_n^p\,dx&\leq\left(\int_{\Omega}u_n^{p^*(p+\gamma-1)/p}\,dx\right)^{p^2/(p^*(p+\gamma-1))}|\Omega|^{(N(\gamma-1)+p^2)/(N(p+\gamma-1))}\\
&\leq \left(S_{N,p}^p\int_{\Omega}|\nabla u_n^{(\gamma+p-1)/p}|^p\,dx\right)^{p/(\gamma+p-1)}|\Omega|^{(N(\gamma-1)+p^2)/(N(p+\gamma-1))}\\
\nonumber&\leq \left(\frac{S_{N,p}^p\|f\|_{L^1(\Omega)}}{\gamma\left(\frac{p}{\gamma+p-1}\right)^p-\beta^+\left(\frac{p}{N-p}\right)^p}\right)^{p/(\gamma+p-1)}|\Omega|^{(N(\gamma-1)+p^2)/(N(p+\gamma-1))}.
\end{align}
Regarding the gradient term of $u_n$, first we choose $\varphi \in C^1_c(\Omega)$ such that $0\leq \varphi\leq 1$, $\varphi=1$ in $\omega\subset\subset \Omega$ and $|\nabla \varphi|\leq 2/d$, where $d=\hbox{dist}\,(\omega, \partial\Omega)$ (this is always possible, see for instance   \cite[page 185]{GilbargTrudinger1998}). Use $\varphi^p u_n$ as test function in \eqref{Ppn} to obtain
\begin{align}
\label{Eq2}\int_{\Omega}\varphi^p|\nabla u_n|^p\,dx&=\beta\int_{\Omega}\frac{(u_n\varphi)^p}{|x|^p+1/n}\,dx+\int_{\Omega}\frac{f_nu_n\varphi^p}{(u_n+1/n)^\gamma}-p\int_{\Omega}u_n\varphi^{p-1}|\nabla u_n|^{p-2}\nabla u_n\cdot \nabla \varphi\,dx\\
\nonumber&\leq|\beta|\int_{\Omega}\frac{u_n^p}{|x|^p}\,dx+\int_{\Omega}\frac{f_nu_n\varphi^p}{(u_n+1/n)^\gamma}+\frac{2p}{d}\int_{\Omega}(\varphi|\nabla u_n|)^{p-1} u_n\,dx.
\end{align}
We estimate each term on the second right-hand side above. By Young's and Hardy's inequalities and \eqref{EstUE}, we get
\begin{align}
\label{Cal}
\nonumber\int_{\Omega}\frac{u_n^p}{|x|^p}\,dx&\leq \frac{p}{p+\gamma-1}\int_{\Omega}\frac{u_n^{p+\gamma-1}}{|x|^p}\,dx+\frac{\gamma-1}{p+\gamma-1}\int_{\Omega}\frac{1}{|x|^p}\,dx\\
&\leq \left(\frac{p}{p+\gamma-1}\right)\left(\frac{p}{N-p}\right)^p\int_{\Omega}|\nabla u_n^{(p+\gamma-1)/p}|^p\,dx+\frac{\gamma-1}{p+\gamma-1}\int_{\Omega}\frac{1}{|x|^p}\,dx\\
\nonumber&\leq \left(\frac{p}{p+\gamma-1}\right)\left(\frac{p}{N-p}\right)^p\frac{\|f\|_{L^1(\Omega)}}{\gamma\left(\frac{p}{p+\gamma-1}\right)^p-\beta^+\left(\frac{p}{N-p}\right)^p}+\frac{\gamma-1}{p+\gamma-1}\int_{\Omega}\frac{1}{|x|^p}\,dx.
\end{align}
We write
$$
\int_{\Omega}\frac{f_nu_n\varphi^p}{(u_n+1/n)^\gamma}\,dx=\int_{\{u_n>1\}}\frac{f_nu_n\varphi^p}{(u_n+1/n)^\gamma}\,dx+\int_{\{u_n\leq 1\}}\frac{f_nu_n\varphi^p}{(u_n+1/n)^\gamma}\,dx.
$$
Note that, in the first term on right-hand side above, one has
\begin{equation}
\label{Cal1}
\int_{\{u_n>1\}}\frac{f_nu_n\varphi^p}{(u_n+1/n)^\gamma}\,dx\leq \|f\|_{L^1(\Omega)};
\end{equation}
 in the second  we use  a similar argument to the one  used for  \eqref{Cal};  we use   $u_n\varphi^p $ as test function in \eqref{Ppn} and,  for each $M>0$, we also use Young's inequality  to  obtain 
\begin{align}
\label{Cal2}
\nonumber\int_{\{u_n\leq 1\}}&\frac{f_nu_n\varphi^p}{(u_n+1/n)^\gamma}\,dx\leq  \int_{\Omega}|\nabla u_n|^{p-2}\nabla u_n\cdot\nabla \varphi^p\,dx-\beta\int_{\Omega}\frac{u_n^{p-1}\varphi^p}{|x|^p+1/n}\,dx\\
\nonumber&\leq  \frac{p}{M^{1/(p-1)}}\int_{\Omega}\varphi^p|\nabla u_n|^p\,dx+pM\int_{\Omega}|\nabla \varphi|^p\,dx+|\beta|\int_{\Omega}\frac{u_n^{p-1}}{|x|^p+1/n}\,dx\\
&\leq   \frac{p}{M^{1/(p-1)}}\int_{\Omega}\varphi^p|\nabla u_n|^p\,dx+pM|\Omega|\left(\frac{2}{d}\right)^p\\
\nonumber&\quad + |\beta|\left(\left(\frac{p-1}{p+\gamma-1}\right)\left(\frac{p}{N-p}\right)^p\frac{\|f\|_{L^1(\Omega)}}{\gamma\left(\frac{p}{p+\gamma-1}\right)^p-\beta^+\left(\frac{p}{N-p}\right)^p}+ \frac{\gamma}{p+\gamma-1}\int_{\Omega}\frac{1}{|x|^p}\,dx\right).
\end{align}

On the other hand, by the  Young inequality and \eqref{Et2}, for each $M>0$, one has
\begin{align}
\label{Cal3}
\nonumber\int_{\Omega}(\varphi|\nabla u_n|)^{p-1}u_n\,dx&\leq \frac{1}{M^{1/(p-1)}}\int_{\Omega}\varphi^p|\nabla u_n|^p\,dx+M\int_{\Omega} u_n^p\,dx\\
&\leq  \frac{1}{M^{1/(p-1)}}\int_{\Omega}\varphi^p|\nabla u_n|^p\,dx\\
\nonumber&\quad+M\left(\frac{S_{N,p}^p\|f\|_{L^1(\Omega)}}{\gamma\left(\frac{p}{\gamma+p-1}\right)^p-\beta^+\left(\frac{p}{N-p}\right)^p}\right)^{p/(\gamma+p-1)}|\Omega|^{(N(\gamma-1)+p^2)/(N(p+\gamma-1))}.
\end{align}
Thus, from \eqref{Cal}-\eqref{Cal3} in \eqref{Eq2}, for $M$ sufficiently large, we obtain
\begin{align}
\label{Et3}
&\int_{\Omega}\varphi^p|\nabla u_n|^p\,dx\leq \left\{\|f\|_{L^1(\Omega)}\left( \frac{\frac{(2p-1)|\beta|}{p+\gamma-1}\left(\frac{p}{N-p}\right)^p}{\gamma\left(\frac{p}{p+\gamma-1}\right)^p-\beta^+\left(\frac{p}{N-p}\right)^p}+1\right)+\frac{(2\gamma-1)|\beta|}{p+\gamma-1}\int_{\Omega}\frac{1}{|x|^p}\,dx+pM|\Omega|\left(\frac{2}{d}\right)^p\right.\\
\nonumber&\left.+\frac{2pM}{d}\left(\frac{S_{N,p}^p\|f\|_{L^1(\Omega)}}{\gamma\left(\frac{p}{\gamma+p-1}\right)^p-\beta^+\left(\frac{p}{N-p}\right)^p}\right)^{p/(\gamma+p-1)}|\Omega|^{(N(\gamma-1)+p^2)/(N(p+\gamma-1))}\right\}\times\left(\frac{dM^{1/(p-1)}}{dM^{1/(p-1)}-p(d+2)}\right),
\end{align}
which together with \eqref{Et2} implies that $(u_n)_{n\in\mathbb{N}}$ is bounded in $W^{1,p}_{loc}(\Omega)$ as desired.

\fim

\begin{proposition}
\label{Sl}
{Let $\gamma>1$ and assume that \eqref{HypG} holds and that $p<N<\frac{p(2p-1)}{p-1}$.} Let $f$ be a nonnegative measurable function in $L^1(\Omega)$. Then there exists a solution $u\in W^{1,p}_{loc}(\Omega)$ of \eqref{Pp} such that $u^{(\gamma+p-1)/p}\in W^{1,p}_0(\Omega)$ and
\begin{equation}
\label{Estp}
\int_{\Omega}|\nabla u^{(\gamma+p-1)/p}_p|^p\,dx\leq {\left(\frac{N-p}{p}\right)^p\frac{\|f\|_{L^1(\Omega)}}{\gamma\left(\frac{N-p}{\gamma+p-1}\right)^p-\beta^+}}.
\end{equation}
\end{proposition}
\dem First of all, note that, by Lemma \ref{L1}, there exists $u_p\in W^{1,p}_{loc}(\Omega)$ such that, up to subsequences,
\begin{align}
\label{ConvLoc}
u_n\rightharpoonup u_p\quad\hbox{ in }\; W^{1,p}_{loc}(\Omega),\\
\nonumber u_n\to u_p\quad\hbox{ in } L^q_{loc}(\Omega),\; q\in [p, p^*),\quad\hbox{ a.e. in }\; \Omega,
\end{align}
as $n\to \infty$.

We shall prove that $u_p$ is a solution to \eqref{Pp} in sense of \eqref{WSp}. In order to do so, we pass to limit in the following identity
\begin{equation}
\label{Estn}
\int_{\Omega}|\nabla u_n|^{p-2}\nabla u_n\cdot\nabla \varphi\,dx=\beta\int_{\Omega}\frac{u_n^{p-1}\varphi}{|x|^p+1/n}\,dx+\int_{\Omega}\frac{f_n\varphi}{(u_n+1/n)^\gamma}\,dx\quad\hbox{for all}\;\varphi\in C^1_c(\Omega).
\end{equation}

Note that, arguing as in \eqref{0}, one can prove 
\begin{equation}
\label{C-1}
\lim_{n\to \infty}\int_{\Omega}|\nabla u_n|^{p-2}\nabla u_n\cdot \nabla \varphi\,dx=\int_{\Omega}|\nabla u_p|^{p-2}\nabla u_p\cdot\nabla \varphi\,dx\quad\hbox{ for all }\; \varphi\in C^1_c(\Omega)
\end{equation}
and, since $u_n\rightarrow u_p$ a.e. in $\Omega$, $u_n(x)\leq u_p(x)$ a.e. $x\in \Omega$ and, for all $\omega\subset\subset\Omega$, there exists $c_\omega$ such that $u_n\geq c_\omega$ in $\omega$, one has, by Lebesgue's dominated convergence theorem, 
\begin{equation}
\label{C-2}
\lim_{n\to\infty}\left(\beta\int_{\Omega}\frac{u_n^{p-1}\varphi}{|x|^p+1/n}\,dx+\int_{\Omega}\frac{f_n\varphi}{(u_n+1/n)^\gamma}\,dx\right)=\beta\int_{\Omega}\frac{u_p^{p-1}\varphi}{|x|^p}\,dx+\int_{\Omega}\frac{f\varphi}{u_p^\gamma}\,dx.
\end{equation}

Letting $n\to \infty$ in \eqref{Estn}, by \eqref{C-1} and \eqref{C-2}, we obtain 
\begin{align}
&\forall\, \omega\subset\subset \Omega,\quad\exists\; c_\omega>0\colon u_p\geq c_\omega\quad\hbox{ in }\; \omega,\\
&\int_{\Omega}|\nabla u_p|^{p-2}\nabla u_p\cdot \nabla \varphi\,dx=\beta\int_{\Omega}\frac{u_p^{p-1}\varphi}{|x|^p}\,dx+\int_{\Omega}\frac{f\varphi}{u_p^\gamma}\,dx\quad\hbox{ for all}\;\varphi\in C^1_c(\Omega).
\end{align}

On the other hand, from  \eqref{EstUE}, it follows that $(u_n^{(p+\gamma-1)/p})_{n\in\mathbb{N}}$ is a bounded sequence in $W^{1,p}_0(\Omega)$ and so, up to subsequences,
\begin{align}
\label{CuG}
&u_n^{(\gamma+p-1)/p}\rightharpoonup u^{(\gamma+p-1)/p}_p\quad\hbox{ in }\; W^{1,p}_0(\Omega),\\
\nonumber&u_n^{(\gamma+p-1)/p}\rightarrow u^{(\gamma+p-1)/p}_p\quad\hbox{ in }\; L^q(\Omega),\; q\in [p,p^*),
\end{align}
as $n\to\infty$. But this, again in \eqref{EstUE}, implies that
$$
\left[\gamma \left(\frac{p}{\gamma+p-1}\right)^p-\beta^+\left(\frac{p}{N-p}\right)^p\right]\int_{\Omega}|\nabla u^{(\gamma+p-1)/p}_p|^p\,dx\leq \int_{\Omega}f\,dx,
$$
as desired.
\fim

By appealing to an argument in  \cite{BoccardoOrsinaPeral2006, OrtizChata2025}, we establish the following quite surprising result:

\begin{proposition}
\label{Unbdd}
{Let $\beta>0$ and assume that $\beta<\left(\frac{N-p}{p}\right)^p$ if $0<\gamma\leq1$, and that \eqref{HypG} holds if $\gamma>1$.} Then the solutions $u_p$ found in Proposition \ref{S}  and Proposition \ref{Sl} are unbounded in $\Omega$. 
\end{proposition}
\dem
 Let $u_n$ be a solution of \eqref{Ppn}. Note that, by \eqref{UnC}, there exists $c>0$ such that
$$
u_n\geq c\quad\hbox{ in }\; B_{\delta}(0)\quad\hbox{ for all }\;n\in \mathbb{N}\quad\hbox{ for some }\;\delta>0.
$$
Moreover, by weak maximum principle one has
\begin{equation}
\label{wmp}
u_n\geq v_n\quad\hbox{ in }\; B_{\delta}(0)\quad\hbox{ for all }\;n\in\mathbb{N},
\end{equation}
where $v_n$ is the solution of the following problem 
\begin{equation}
\left\{
\begin{array}{rclr}
-\Delta_p v&=& \dys \beta c^{p-1}\left(\frac{1}{|x|^p+1/n}-\frac{1}{\delta ^{p}+1/n}\right)&\hbox{ in }\;B_\delta(0)\\
v&=&0&\hbox{on}\;\partial B_\delta(0).
\end{array}
\right.
\end{equation}
Note that the solution is unique and radial decreasing for this problem,
$$
v_n(x)=c\beta^{1/(p-1)}\int_{|x|}^{\delta}\frac{1}{t^{(N-1)/(p-1)}}\left(\int_{0}^{t}s^{N-1}\left(\frac{1}{s^p+1/n}-\frac{1}{\delta^p+1/n}\right)\,ds\right)^{1/(p-1)}\,dt
$$
and, letting $n\to\infty$, one has
\begin{align}
\liminf_{n\to\infty}v_n(x)&\geq c\beta^{1/(p-1)}\int_{|x|}^{\delta}\frac{1}{t^{(N-1)/(p-1)}}\left(\int_{0}^{t}s^{N-1}\left(\frac{1}{s^p}-\frac{1}{\delta^p}\right)\,ds\right)^{1/(p-1)}\,dt\\
&\geq c\left(\frac{\beta p}{(N-p)N}\right)^{1/(p-1)}(\ln \delta-\ln|x|).
\end{align}

Consequently, passing to the limit in \eqref{wmp}, we have
\begin{equation}
\label{EUp}
u_p(x)\geq c\left(\frac{\beta p}{(N-p)N}\right)^{1/(p-1)}(\ln \delta-\ln|x|)\quad\hbox{ a.e. }\;x\in B_\delta(0),
\end{equation}
since $u_p(x)=\sup_{n\in\mathbb{N}}u_n(x)\quad\hbox{a.e. }\; x\in B_\delta(0)$.
\fim

\begin{remark}
Observe that, if $0<\beta<N-1$ and so, $0<\beta<((N-p)/p)^p$ for $p>1$ small enough, one has 
$$
\lim_{p\to 1^+} \left(\frac{\beta p}{(N-p)N}\right)^{1/(p-1)}=0; 
$$
therefore even if  $\lim_{p\to 1^+}u_p(x)$ exists a.e. $x\in B_\delta(0)$, we could not  conclude that this limit $u$ is unbounded by \eqref{EUp}. In Section \ref{Br} we shall prove that this limit is indeed bounded in $\Omega$.
\end{remark}

\section{{The limit as $p\to1^+$}}
\label{PT}
 
{The asymptotic analysis of the family $(u_p)_{p>1}$, which is the core of the proof of Theorem \ref{mainTh}, will be divided into various lemmas.} The overall procedure is  inspired by  \cite{DeCiccoGiachettiSegura2019,dgop} and \cite{OrtizPetitta2024}; although, the arguments are different in some crucial technical points that   we shall highlight.

 First of all,  since $\lambda<N-1$, there exists $\bar p>1$ such that
\begin{equation}\label{RKey}
\lambda<\left(\frac{N-p}{p}\right)^p\quad\hbox{ if }\; 0<\gamma\leq 1\quad \hbox{ and }\quad \gamma \left(\frac{p}{\gamma+p-1}\right)^p-\lambda^+\left(\frac{p}{N-p}\right)^p>0\quad\hbox{if}\;\gamma>1,
\end{equation}
for all $p\in (1,\bar p]$. Thus, we may consider $\lambda$ instead of $\beta$ in  the  $p$-Laplace problem \eqref{Pp}. We also consider 
$$
f\in L^m(\Omega),\quad m=\left(\frac{p^*}{1-\gamma}\right)',
$$
since $L^{N,\infty}(\Omega)\subset L^m(\Omega)$ for $ m=\left(\frac{p^*}{1-\gamma}\right)'<N$ if $\gamma<1$ and $f\in L^{1}(\Omega)$ for $\gamma\geq 1$.

For $0<\gamma\leq1$,  from \eqref{Est6}, the following estimate holds 
\begin{equation}
\label{En}
\int_{\Omega}|\nabla u_n|^p\,dx\leq \left[\frac{\|f\|_m S_{N,p}^{1-\gamma}}{1-\lambda^+(p/(N-p))^p}\right]^{p/(p-1+\gamma)}\quad\hbox{ for all } \;p\in (1,\bar p].
\end{equation}
Letting $n\to\infty$, by \eqref{CUn}, we have that
\begin{equation}
\label{Ep}
\int_{\Omega}|\nabla u_p|^p\,dx\leq \left[\frac{\|f\|_m S_{N,p}^{1-\gamma}}{1-\lambda^+(p/(N-p))^p}\right]^{p/(p-1+\gamma)}\quad\hbox{ for all } \;p\in (1,\bar p].
\end{equation}
But this, in turn, as $u_p=0$ on $\partial\Omega$, implies that
$$
\int_{\Omega}|\nabla u_p|\,dx+\int_{\partial\Omega}|u_p|\,d\mathcal{H}^{N-1}\leq\frac{1}{p}\left[\frac{\|f\|_m S_{N,p}^{1-\gamma}}{1-\lambda^+(p/(N-p))^p}\right]^{p/(p-1+\gamma)}+\frac{p-1}{p}|\Omega|,
$$
for all $p\in (1,\bar p]$. Hence $u_p$ is a bounded sequence in $BV(\Omega)$ and so, by compact embedding, up to subsequences,
\begin{eqnarray}
\label{u2}
\nonumber \nabla u_p\rightharpoonup Du&*-\hbox{ weak in }\;BV(\Omega),\\
u_p\to u&\hbox{ in }\; L^q(\Omega)\; \hbox{ for all }\; q\in [1, 1^*),\quad\hbox{ a.e. in }\;\Omega,\\
\nonumber u_p\rightharpoonup u&\hbox{ in}\; L^{1^*}(\Omega),
\end{eqnarray}
as $p\to 1^+$.

 On the other hand, let $\omega\subset\subset\Omega$, for $\gamma>1$, from \eqref{Et2} and  \eqref{Et3}, both with $\lambda$ instead of $\beta$, the following estimate is also established
\begin{align}
\label{Cal4}
\nonumber&\|u_n\|^p_{W^{1,p}(\omega)}\leq \left\{\|f\|_{L^1(\Omega)}\left( \frac{\frac{(2p-1)|\lambda|}{p+\gamma-1}\left(\frac{p}{N-p}\right)^p}{\gamma\left(\frac{p}{p+\gamma-1}\right)^p-\lambda^+\left(\frac{p}{N-p}\right)^p}+1\right)+\frac{(2\gamma-1)|\lambda|}{p+\gamma-1}\int_{\Omega}\frac{1}{|x|^p}\,dx+pM|\Omega|\left(\frac{2}{d}\right)^p\right.\\
\nonumber&\left.+\frac{2pM}{d}\left(\frac{S_{N,p}^p\|f\|_{L^1(\Omega)}}{\gamma\left(\frac{p}{\gamma+p-1}\right)^p-\lambda^+\left(\frac{p}{N-p}\right)^p}\right)^{p/(\gamma+p-1)}|\Omega|^{(N(\gamma-1)+p^2)/(N(p+\gamma-1))}\right\}\times\left(\frac{dM^{1/(p-1)}}{dM^{1/(p-1)}-p(d+2)}\right)\\
&+ \left(\frac{S_{N,p}^p\|f\|_{L^1(\Omega)}}{\gamma\left(\frac{p}{\gamma+p-1}\right)^p-\lambda^+\left(\frac{p}{N-p}\right)^p}\right)^{p/(\gamma+p-1)}|\Omega|^{(N(\gamma-1)+p^2)/(N(p+\gamma-1))}=:C(p, \lambda, \gamma, d, M),
\end{align}
where 
$$
\|v\|_{W^{1,p}(\omega)}=\left(\int_{\omega}|\nabla v|^p\,dx+\int_{\omega}|v|^p\,dx\right)^{1/p}.
$$

But this, in turn, implies that, 
\begin{align*}
\int_{\omega}|\nabla u_p|\,dx+\int_{\omega}u_p\,dx&\leq \frac{1}{p}\|u_p\|_{W^{1,p}(\omega)}^p+\frac{2(p-1)}{p}|\omega|\\
&\leq \frac{1}{p}\liminf_{n\to\infty}\|u_n\|^p_{W^{1,p}(\omega)}+\frac{2(p-1)}{p}|\Omega|\\
&\leq \frac{1}{p}C(p,\lambda,\gamma,d,M)+\frac{2(p-1)}{p}|\Omega|,
\end{align*}
where was used the fact that $u_n\rightharpoonup u_p$ in $W^{1,p}_{loc}(\Omega)$. Moreover, since $C(p,\lambda,\gamma,d,M)\to C(1,\lambda,\gamma,d,M)$ as $p\to1^+$,  $(u_p)_{p>1}$ is bounded in $BV_{loc}(\Omega)$ and so, by \cite[Theorem 3.23]{AmbrosioFuscoPallara}, there exists $u\in BV_{loc}(\Omega)$ such that, up to subsequences,
\begin{equation}
\label{u3}
 u_p\rightarrow u\quad\hbox{ in }\;L^1_{loc}(\Omega),
\end{equation}
as $p\to 1^+$.

On the other hand, since $u_p^{(\gamma+p-1)/p}\in {W^{1,p}_0(\Omega)}$, by Young inequality and \eqref{Estp}, one has
\begin{align}
\int_{\Omega}|\nabla u_p^{(\gamma+p-1)/p}|\,dx+\int_{\partial\Omega}|u^{(\gamma+p-1)/p}|\,d\mathcal{H}^{N-1}&\leq \frac{1}{p}\int_{\Omega}|\nabla u_p^{(\gamma+p-1)/p}|^p\,dx+\frac{p-1}{p}|\Omega|\\
&\leq \frac{1}{p}\frac{\|f\|_{L^1(\Omega)}}{\gamma\left(\frac{p}{\gamma+p-1}\right)^p-\lambda^+\left(\frac{p}{N-p}\right)^p}+\frac{p-1}{p}|\Omega|.
\end{align}
Hence $(u_p^{(\gamma+p-1)/p})_{p>1}$ is a bounded sequence in $BV(\Omega)$ and so, up to subsequences,
\begin{eqnarray}
\label{u4}
\nonumber \nabla u_p^{(\gamma+p-1)/p}\rightharpoonup Du^\gamma&*-\hbox{ weak in }\;BV(\Omega),\\
u_p^{(\gamma+p-1)/p}\to u^\gamma&\hbox{ in }\; L^q(\Omega)\; \hbox{ for all }\; q\in [1, 1^*),\quad\hbox{ a.e. in }\;\Omega,\\
\nonumber u_p^{(\gamma+p-1)/p}\rightharpoonup u^\gamma&\hbox{ in}\; L^{1^*}(\Omega),
\end{eqnarray}
as $p\to 1^+$.

\begin{lemma} 
\label{ConvZ}
Let $q>1$ and $\gamma>0$. Suppose that $u_p$ is a solution of \eqref{Pp}. Then, there exists  ${\bf z}\in L^\infty(\Omega, \mathbb{R}^N)$ with $\|{\bf z}\|_{L^\infty(\Omega)}\leq 1$ such that, for $0<\gamma\leq 1$, 

\begin{align}
\label{Cz}
|\nabla u_p|^{p-2}\nabla u_p\rightharpoonup{\bf z}\quad \mbox{ in }\;L^q(\Omega,\mathbb{R}^N)\quad\mbox{ as }\;p\to 1^+
\end{align}
and, for $\gamma>1$, 
\begin{align}
\label{Cz1}
|\nabla u_p|^{p-2}\nabla u_p\rightharpoonup{\bf z}_\epsilon\quad \mbox{ in }\;L^q(\Omega_\epsilon,\mathbb{R}^N)\quad \mbox{ as }\;p\to 1^+\quad\hbox{ and }\quad \chi_{\Omega_\epsilon}{\bf z}_\epsilon\rightharpoonup {\bf z}\quad\hbox{ in }\; L^q(\Omega, \mathbb{R}^N)\; \hbox{ as }\;\epsilon\to 0.
\end{align}
 Furthermore, ${\bf z}\in X_{\mathcal{M}_{loc}}(\Omega)$.
\end{lemma}

\dem   Suppose that $0<\gamma\leq 1$. Use H\"older's inequality and \eqref{Est6} to obtain
\begin{align}
\int_{\Omega}||\nabla u_p|^{p-2}\nabla u_p|^q\,dx&\leq \left(\int_{\Omega}|\nabla u_p|^p\,dx\right)^{q(p-1)/p}|\Omega|^{1-q(p-1)/p}\\
&\leq  \left[\frac{\|f\|_m S_{N,p}^{1-\gamma}}{1-\lambda^+(p/(N-p))^p}\right]^{(p-1)q/(p-1+\gamma)} |\Omega|^{1-q(p-1)/p},
\end{align}
for all $p\in (1, \bar{p}]$.  Hence
\begin{equation}
\label{NUp}
\left(\int_{\Omega}||\nabla u_p|^{p-2}\nabla u_p|^q\,dx\right)^{1/q}\leq\left[\frac{\|f\|_m S_{N,p}^{1-\gamma}}{1-\lambda^+(p/(N-p))^p}\right]^{(p-1)/(p-1+\gamma)}|\Omega|^{1/q-(p-1)/p},
\end{equation}
for all $p\in (1,\bar p]$. Thus, there exists ${\bf z}_q\in L^q(\Omega,\mathbb{R}^N)$ such that , up to subsequences,
\begin{equation}
\label{FCUp}
|\nabla u_p|^{p-2}\nabla u_p\rightharpoonup {\bf z}_q\quad\hbox{ in }\; L^q(\Omega, \mathbb{R}^N)\quad\hbox{ as }\; p\to 1^+
\end{equation}
and, by a diagonal argument, there exists ${\bf z}$ independent of $q$ such that \eqref{FCUp} yields.

Letting $p\to 1^+$ in \eqref{NUp}, we obtain that
$$
\|{\bf z}\|_{L^q(\Omega, \mathbb{R}^N)}\leq |\Omega|^{1/q}
$$
and so, letting $q\to \infty$, this implies 
$$
\|{\bf z}\|_{L^\infty(\Omega, \mathbb{R}^N)}\leq 1.
$$
Let $\varphi\in C^1_c(\Omega)$ with $\varphi\geq 0$. Use $\varphi$ as test function in \eqref{WSp}. After, letting $p\to 1^+$, using Fatou's lemma, we obtain that
\begin{equation}
\label{SupS}
\int_{\Omega}{\bf z}\cdot\nabla \varphi\,dx\geq \lambda\int_{\Omega}\frac{1}{|x|}s(x)\varphi\,dx+\int_{\Omega}\frac{f}{u^\gamma}\varphi\,dx,
\end{equation} 
where $s(x)\in \hbox{Sgn}\,(u(x))$ a.e. $x\in \Omega$. If we set $g(x) := \lambda \frac{s(x)}{\vert{}x\vert{}} + \frac{f}{u^\gamma} \in L^1_{\text{loc}}(\Omega)$, 
since $\langle -\operatorname{div} z - g, \phi \rangle \ge 0$ for all non-negative $\phi \in C_c^1(\Omega)$, $T = -\operatorname{div} z - g$ defines a positive linear functional, and thus $-\operatorname{div} z$ is a Radon measure on $\Omega$ by the Riesz Representation Theorem  $-\hbox{div}\,{\bf z}$ is a Radon measure (see for instance \cite[Theorem 7.2]{Folland1999}) and so, ${\bf z}\in X_{\mathcal{M}_{loc}}(\Omega)$.

Assume that $\gamma>1$. For each $\epsilon>0$, we define
$$
\Omega_\epsilon=\{x\in \Omega\colon \hbox{dist}\,(x,\partial\Omega)>\epsilon\}.
$$

Let $q>1$. By \eqref{Cal4} and H\"older inequality, one has
\begin{align}
\label{Bd}
\int_{{\Omega_\epsilon}}|\nabla u_p|^{(p-1)q}\,dx&\leq \left(\int_{{\Omega_\epsilon}}|\nabla u_p|^p\,dx\right)^{(p-1)q/p}|\Omega|^{1-(p-1)q/p}\\
\nonumber&\leq C(p, \lambda, \gamma, d, M)^{(p-1)q/p}|\Omega|^{1-(p-1)q/p},
\end{align}
where $d=\hbox{dist}\,(\Omega_\epsilon,\partial\Omega)=\epsilon$. Since $C(p,\lambda,\gamma,\epsilon, M)\to C(1,\lambda, \gamma, \epsilon, M)$ as $p\to 1^+$, there exists ${\bf z}_\epsilon^q$ such that, up to subsequences,
$$
|\nabla u_p|^{p-2}\nabla u_p\rightharpoonup {\bf z}_\epsilon^q\quad\hbox{ in }\;L^q(\Omega_\epsilon, \mathbb{R}^N),
$$
as $p\to 1^+$. By a diagonal argument there exists ${\bf z}_\epsilon$ independent of $q$ such that, up to subsequences,
$$
|\nabla u_p|^{p-2}\nabla u_p\rightharpoonup {\bf z}_\epsilon\quad\hbox{ in }\;L^q(\Omega_\epsilon, \mathbb{R}^N),
$$
as $p\to 1^+$. But this, in turn, in \eqref{Bd} implies that, by lower semicontinuity,
\begin{equation}
\label{Ee}
\left(\int_{\Omega_\epsilon}|{\bf z}_\epsilon|^q\right)^{1/q}\leq |\Omega|^{1/q}
\end{equation}
Note that $(\chi_{\Omega_\epsilon}{\bf z}_\epsilon)_{\epsilon>0}$ is bounded in $L^q(\Omega, \mathbb{R}^N)$. Hence there exists ${\bf z}\in L^q(\Omega, \mathbb{R}^N)$ such that, up to subsequences, 
$$
\chi_{\Omega_\epsilon}{\bf z}_\epsilon\rightharpoonup {\bf z}\quad\hbox{ in }\; L^q(\Omega, \mathbb{R}^N),
$$
as $\epsilon\to 0$. Hence, in \eqref{Ee}, by lower semicontinuity, we get
$$
\left(\int_{\Omega}|{\bf z}|^q\,dx\right)^{1/q}\leq |\Omega|^{1/q}
$$
and, letting $q\to\infty$,
$$
\|{\bf z}\|_{L^\infty(\Omega, \mathbb{R}^N)}\leq 1.
$$
 
On the other hand, using $0\leq \varphi\in C^1_c(\Omega)$ as test function in \eqref{WSp} and letting $p\to1^+$, we obtain
$$
\int_{\Omega_\epsilon}{\bf z}_\epsilon\cdot\nabla\varphi\,dx\geq \lambda\int_{\Omega}\frac{s(x)\varphi}{|x|}\,dx+\int_{\Omega}\frac{f\varphi}{u^\gamma}\,dx
$$
and, since $\chi_{\Omega_\epsilon}{\bf z}_\epsilon\rightharpoonup {\bf z}$ in $L^q(\Omega, \mathbb{R}^N)$, one has
\begin{equation}
\label{SupS1}
\int_{\Omega}{\bf z}\cdot\nabla\varphi\,dx\geq \lambda\int_{\Omega}\frac{s(x)\varphi}{|x|}\,dx+\int_{\Omega}\frac{f\varphi}{u^\gamma}\,dx,
\end{equation}
where $s(x)\in \hbox{Sgn}\,(u(x))$ a.e. $x\in \Omega$ and, by Riesz Theorem, $-\hbox{div}\,{\bf z}$ is a Radon measure and so, ${\bf z}\in X_{\mathcal{M}_{loc}}(\Omega)$.

\fim

\begin{lemma}
\label{Convs}
Let $u_p$ be a sequence of solutions to the $p$-Laplace problem \eqref{Pp}. Then
$$
\lim_{p\to 1^+}\int_{\Omega}\frac{1}{|x|^p}|u_p|^{p-2}u_p\varphi\,dx=\int_{\Omega}\frac{1}{|x|}s(x)\varphi\,dx,
$$
where $s(x)\in \hbox{Sgn}\,(u(x))$ a.e. $x\in \Omega$ where  $u$  is given in  \eqref{u2}.
\end{lemma}
 
\dem
 For $0<\gamma\leq1$, one can argue exactly as in \cite[Theorem 3.3]{OrtizPetitta2024} to obtain the result.

 In order to prove it for $\gamma>1$, we write
\begin{equation}
\label{D1}
\int_{\Omega}\frac{u_p^{p-1}\varphi}{|x|^p}\,dx=\int_{\{u>0\}}\frac{u_p^{p-1}\varphi}{|x|^p}\,dx+\int_{\{u=0\}}\frac{u_p^{p-1}\varphi}{|x|^p}\,dx.
\end{equation}
Note that, by H\"older's and Sobolev's  inequalities and \eqref{Estp}, we have
\begin{align}
\label{Et0}
\nonumber\left(\int_{\{u=0\}}u_p^{(p-1)q}\,dx\right)^{1/q}&\leq \left(\int_{\Omega}u_p^{(p+\gamma-1)N/(N-p)}\,dx\right)^{(N-p)(p-1)/(N(p+\gamma-1))}|\Omega|^{1/q-(p-1)(N-p)/(N(p+\gamma-1))}\\
&\leq \left(S_{N,p}^p\int_{\Omega}|\nabla u_p^{(p+\gamma-1)/p}|^p\,dx\right)^{(p-1)/(N(p+\gamma-1))}|\Omega|^{1/q-(p-1)(N-p)/(N(p+\gamma-1))}\\
\nonumber&\leq \left(\frac{S_{N,p}^p\|f\|_{L^1(\Omega)}}{\gamma\left(\frac{p}{p+\gamma-1}\right)^p-\lambda^+\left(\frac{p}{N-p}\right)^p}\right)^{(p-1)/(N(p+\gamma-1))}|\Omega|^{1/q-(p-1)(N-p)/(N(p+\gamma-1))}
\end{align}
Hence $(u_p^{p-1})_{p>1}$ is bounded in $L^q(\{u=0\})$ and so, there exists $v_q\in L^q(\{u=0\})$ such that, up to subsequences,
$$
u_p^{p-1}\rightharpoonup v_q\quad\hbox{ in }\; L^q(\{u=0\})\quad\hbox{ as }\; p\to1^+.
$$
By a diagonal argument, there exists $v$ independent of $q$ in $L^q(\{u=0\})$, such that, up to subsequences,
\begin{equation}
\label{Cv}
u_p^{p-1}\rightharpoonup v\quad\hbox{ in }\; L^q(\{u=0\})\quad\hbox{ as }\; p\to1^+.
\end{equation}
But, in \eqref{Et0}, this implies that, 
$$
\left(\int_{\{u=0\}}|v|^q\,dx\right)^{1/q}\leq |\Omega|^{1/q}
$$
and so, letting $q\to\infty$, 
$$
\|v\|_{ L^\infty(\{u=0\})}\leq 1.
$$
Consequently, in \eqref{D1}, we obtain  
$$
\lim_{p\to 1^+}\int_{\Omega}\frac{u_p^{p-1}\varphi}{|x|}\,dx=\int_{\{u>0\}}\frac{\varphi}{|x|^p}\,dx+\int_{\{u=0\}}\frac{v\varphi}{|x|}\,dx=\int_{\Omega}\frac{s(x)\varphi}{|x|}\,dx,
$$
where $s(x)\in \hbox{Sgn}\,(u(x))$ a.e. $\Omega$.

\fim

\begin{lemma}
\label{IdUZ}
 For $0<\gamma\leq 1$. Let ${\bf z}\in X_{\mathcal{M}_{loc}}(\Omega)$ and $u\in BV(\Omega)$ as   in \eqref{Cz} and \eqref{u2}, respectively. Then
\begin{equation}
\label{G}
-\int_{\Omega}u^*\hbox{div}\,{\bf z}= \lambda\int_{\Omega}\frac{|u|}{|x|}\,dx+\int_{\Omega}fu^{1-\gamma}\,dx.
\end{equation}
In particular $u^*\in L^1(\Omega, \hbox{div}\,{\bf z})$. For $\gamma>1$, let ${\bf z}\in X_{\mathcal{M}_{loc}}(\Omega)$ and $u^\gamma\in BV(\Omega)$ as  in \eqref{Cz1} and \eqref{u4}, respectively, \eqref{G} holds with $u^\gamma$ instead of $u$.

\end{lemma}
\dem
First of all, let $0\leq \varphi\in  C^\infty_c(\Omega)$ and $\rho_\epsilon$ be the mollifiers defined as in \eqref{mll}. Moreover, we know  that there exists $\epsilon_0>0$ such that  
$$
\hbox{supp}\,(\varphi)\subset \{x\in \Omega\>:\> \hbox{dist}\,(x, \partial\Omega)>\epsilon_0\}=:\Omega_{\epsilon_0}.
$$
Note that $\Omega_{\epsilon_0}\subset\subset\Omega$. 

Suppose that $0<\gamma\leq 1$. Use $(T_k(u)*\rho_\epsilon)\varphi$ as test function in \eqref{SupS}. Then, letting $\epsilon$ goes to $0$, one has that
\begin{equation}
\label{1Est}
-\int_{\Omega}T_k(u)^*\varphi\,\hbox{div}\,{\bf z}\geq \lambda\int_{\Omega}\frac{s(x)}{|x|}T_k(u)\varphi\,dx+\int_{\Omega}\frac{f}{u^\gamma}T_k(u)\varphi\,dx.
\end{equation} 
  In order to prove the opposite inequality, let $\epsilon\in (0, \frac{\epsilon_0}{2})$, the following convergence holds (see for instance \cite[Lemma 3.15]{Adams1975}) 
$$
\rho_\epsilon*T_k(u_p)\rightarrow T_k(u_p)\quad\hbox{ in }\; W^{1, p}(\Omega_{\epsilon_0}),
$$
as $\epsilon\to 0$. Thus
\begin{equation}
\label{C1}
\lim_{\epsilon\to 0}\int_{\Omega}\varphi|\nabla u_p|^{p-2}\nabla u_p\cdot \nabla(\rho_\epsilon*T_k(u_p))\,dx=\int_{\Omega}\varphi|\nabla u_p|^{p-2}\nabla u_p\cdot\nabla T_k(u_p)\,dx
\end{equation}
and 
\begin{equation}
\label{C2}
\lim_{\epsilon\to 0}\int_{\Omega}\rho_\epsilon*T_k(u_p)|\nabla u_p|^{p-2}\nabla u_p\cdot \nabla \varphi\,dx=\int_{\Omega}T_k(u_p)|\nabla u_p|^{p-2}\nabla u_p\cdot \nabla \varphi\,dx.
\end{equation}

We use $ (\rho_{\epsilon}* T_k(u_p))\varphi$ as test function in \eqref{WSp} and we pass to the limit as $\epsilon$ goes to 0. Then, by \eqref{C1} and \eqref{C2},
$$
\int_{\Omega}\varphi|\nabla T_k(u_p)|^p\,dx+\int_{\Omega}T_k(u_p)|\nabla u_p|^{p-2}\nabla u_p\cdot \nabla \varphi\,dx{=}\lambda\int_{\Omega}\frac{1}{|x|^p}u_p^{p-1}T_k(u_p)\varphi\,dx+\int_{\Omega}\frac{f}{u_p^\gamma}T_k(u_p)\varphi\,dx.
$$
But this, in turn, using Young's inequality, implies that
\begin{equation}
\label{E1}
\int_{\Omega}\varphi|\nabla T_k(u_p)|\,dx\leq \frac{\lambda}{p}\int_{\Omega}\frac{1}{|x|^p}u_p^{p-1}T_k(u_p)\varphi\,dx+\frac{1}{p}\int_{\Omega}\frac{f}{u_p^\gamma} T_k(u_p)\varphi\,dx+\frac{p-1}{p}\int_{\Omega}\varphi\,dx-\int_{\Omega}T_k(u_p)|\nabla u_p|^{p-2}\nabla u_p\cdot\nabla \varphi\,dx.
\end{equation}
Letting $p\to 1^+$, using the weak lower semicontinuity of $u\mapsto \int_{\Omega}\varphi|Du|$ with respect to the $L^1$ convergence, it becomes
\begin{align}
\label{IneqDu}
\int_{\Omega}\varphi|DT_k(u)|+\int_{\Omega}T_k(u){\bf z}\cdot \nabla \varphi\,dx&\leq \lambda\int_{\Omega}\frac{1}{|x|}s(x)T_k(u)\varphi\,dx+\int_{\Omega}\frac{f}{u^\gamma}T_k(u)\varphi\,dx
\end{align}

Now by Proposition \ref{Dist} (see also Remark \ref{R1}), one has  
\begin{align}
\left|\int_{\Omega}\varphi({\bf z}, T_k(u))\right|\leq  \int_{\Omega}\varphi|DT_k(u)|.\nonumber 
\end{align}
Then,  recalling  Definition \ref{zDu} and using \eqref{IneqDu} and \eqref{1Est}, one has  
\begin{equation}
\label{A1}
-\int_{\Omega}\varphi (T_k(u))^*\hbox{div}\,{\bf z}=\lambda\int_{\Omega}\frac{1}{|x|}s(x)T_k(u)\varphi\,dx+\int_{\Omega}\frac{f}{u^\gamma}T_k(u) \varphi\,dx.
\end{equation}
But this, in turn, implies that $(|- (T_k(u))^*\hbox{div}\,{\bf z}|)_{k\geq 1}$ is a bounded sequence of finite Radon measures on $\Omega$. Hence, by \cite[Theorem 1.54, Proposition 1.62 b)]{AmbrosioFuscoPallara}, up to subsequences, 
$$
- (T_k(u))^*\hbox{div}\,{\bf z}\rightharpoonup -u^*\hbox{div}\,{\bf z}\qquad*-\hbox{weakly in }\quad\Omega,
 $$
as $k\to\infty$. Consequently, letting $k\to \infty$  in \eqref{A1} we have that  
\begin{equation}
\label{B1}
-\int_{\Omega}\varphi u^*\hbox{div}\,{\bf z}  =  \lambda\int_{\Omega}\frac{1}{|x|}s(x)u\varphi\,dx+\int_{\Omega}fu^{1-\gamma}\varphi\,dx,\qquad\hbox{ for all }\;\varphi\in C^1_c(\Omega).
\end{equation}
 
Moreover, this identity holds for all $\varphi\in C_0(\Omega)$ since it is the completion of $C_c(\Omega)$ with respect to sup-norm, which, in turn, is the completion of $C^1_c(\Omega)$ with respect to the same norm. Hence, by \cite[Theorem 1.54]{AmbrosioFuscoPallara}, 
$$
-\int_{\Omega}u^*\hbox{div}\,{\bf z}=\lambda\int_{\Omega}\frac{1}{|x|}s(x)u\,dx+\int_{\Omega}fu^{1-\gamma}\,dx
$$
and
\begin{equation}
\label{UL1}
|-u^*\hbox{div}\,{\bf z}|(\Omega)=\int_{\Omega}u^*|-\hbox{div}\,{\bf z}|<\infty.
\end{equation}
 
Suppose that $\gamma>1$. Use $((T_k(u))^\gamma\ast \rho_\epsilon)\varphi$ as test function in \eqref{SupS} and after, let $\epsilon\to 0$, to obtain
\begin{equation}
\label{1D}
-\int_{\Omega}((T_k(u))^\gamma)^*\varphi\hbox{div}\,{\bf z}\geq\lambda\int_{\Omega}\frac{s(x)}{|x|} (T_k(u))^\gamma\varphi\,dx+\int_{\Omega}\frac{f}{u^\gamma} (T_k(u))^\gamma\varphi\,dx.
\end{equation}
In order to prove the contrary inequality, for $\epsilon\in (0,\frac{\epsilon_0}{2})$, note that 
\begin{equation}
\label{CT}
\rho_\epsilon*(T_k(u_p))^\gamma\to  (T_k(u_p))^\gamma\qquad\hbox{ in }\quad W^{1,p}(\Omega_{\epsilon_0}),
\end{equation}
as $\epsilon\to0$.

We use $ (\rho_{\epsilon}* T_k(u_p))\varphi$ as test function in \eqref{WSp} and we pass to the limit as $\epsilon$ goes to 0 by considering \eqref{CT},  to obtain
\begin{align}
&\gamma\left(\frac{p}{\gamma+p-1}\right)^p\int_{\Omega}\varphi|\nabla (T_k(u_p))^{(\gamma+p-1)/p}|^p\,dx+\int_{\Omega}(T_k(u_p))^\gamma|\nabla u_p|^{p-2}\nabla u_p\cdot \nabla \varphi\,dx\\
&\quad\leq \lambda\int_{\Omega}\frac{1}{|x|^p}u_p^{p-1}(T_k(u_p))^\gamma\varphi\,dx+\int_{\Omega}\frac{f}{u_p^\gamma}(T_k(u_p))^\gamma\varphi\,dx.
\end{align}
Arguing as the previous case,  using the Young inequality and then letting $p\to1^+$,  we get
 $$
\int_{\Omega}\varphi|D (T_k(u))^\gamma|\leq \lambda\int_{\Omega}\frac{s(x)(T_k(u))^\gamma}{|x|}\varphi\,dx+\int_{\Omega}\frac{f}{u^\gamma}(T_k(u))^\gamma \varphi\,dx-\int_{\Omega}(T_k(u))^\gamma{\bf z}_\epsilon\cdot\nabla \varphi\,dx,
$$
for each $\epsilon\in (0,\epsilon_0/2)$, where was used \eqref{Cz1}. Letting $\epsilon\to 0$, by \eqref{Cz1} again one has
\begin{equation}
\label{Ek}
\int_{\Omega}\varphi|D (T_k(u))^\gamma|\leq \lambda\int_{\Omega}\frac{s(x)}{|x|}(T_k(u))^\gamma\varphi\,dx+\int_{\Omega}\frac{f}{u^\gamma}(T_k(u))^\gamma\varphi\,dx-\int_{\Omega}(T_k(u))^\gamma{\bf z}\cdot\nabla \varphi\,dx
\end{equation}
and so, 
$$
-\int_{\Omega}(T_k(u)^\gamma)^*\varphi\hbox{div}\,{\bf z}\leq \lambda\int_{\Omega}\frac{s(x)}{|x|}(T_k(u))^\gamma\varphi\,dx+\int_{\Omega}\frac{f}{u^\gamma}(T_k(u))^\gamma\varphi\,dx,
$$
where was used the fact that $\left|\int_{\Omega}\varphi({\bf z}, (DT_k(u))^\gamma)\right|\leq \int_{\Omega}\varphi|D(T_k(u))^\gamma|$ and Definition \ref{zDu}. Moreover, by \eqref{1D} one has
$$
 -\int_{\Omega}(T_k(u)^\gamma)^*\varphi\hbox{div}\,{\bf z}= \lambda\int_{\Omega}\frac{s(x)}{|x|}(T_k(u))^\gamma\varphi\,dx+\int_{\Omega}\frac{f}{u^\gamma}(T_k(u))^\gamma\varphi\,dx.
$$
Hence, letting $k\to\infty$ one has
\begin{equation}
\label{2D}
-\int_{\Omega}(u^\gamma)^*\varphi\,\hbox{div}\,{\bf z}=\lambda\int_{\Omega}\frac{|u^\gamma|}{|x|}\varphi\,dx+\int_{\Omega}f\varphi\,dx.
\end{equation}
and, by \cite[Theorem 1.54]{AmbrosioFuscoPallara}, we can drop $\varphi$ above and holds
$$
|-(u^\gamma)^*\hbox{div}\,{\bf z}|(\Omega)=\int_{\Omega}(u^\gamma)^*|\hbox{div}\,{\bf z}|<\infty.
$$

\fim

\begin{lemma}
\label{ZDu}
For $0<\gamma\leq 1$, let $u\in BV(\Omega)$ and ${\bf z}\in X_{\mathcal{M}_{loc}}(\Omega)$ be as in \eqref{u2}  and \eqref{Cz}, respectively, we have that
\begin{equation}
\label{Du}
|Du|=({\bf z}, Du)\quad\hbox{ as  finite Radon measures on }\;\Omega.
\end{equation}
For $\gamma>1$, let $u^\gamma\in BV(\Omega)$ and ${\bf z}\in X_{\mathcal{M}_{loc}}(\Omega)$ be as in \eqref{u4} and \eqref{Cz1}, respectively, one has that \eqref{Du} holds too. 
 
\end{lemma}
\dem First of all, let $0\leq\varphi\in C^\infty_c(\Omega)$. Assume that $0<\gamma\leq 1$. Then, from \eqref{IneqDu}, 
$$
\int_{\Omega}\varphi|DT_k(u)|\leq \lambda\int_{\Omega}\frac{1}{|x|}s(x)T_k(u)\varphi\,dx+\int_{\Omega}\frac{f}{u^\gamma}T_k(u)\varphi\,dx-\int_{\Omega}T_k(u){\bf z}\cdot \nabla \varphi\,dx
$$
and, letting $k\to\infty$, we have that
$$
\int_{\Omega}\varphi|Du|\leq\lambda\int_{\Omega}\frac{1}{|x|}s(x)u\varphi\,dx+\int_{\Omega}f u^{1-\gamma}\varphi\,dx -\int_{\Omega}u{\bf z}\cdot\nabla \varphi\,dx
$$
and so, using \eqref{B1}, 
$$
 \int_{\Omega}\varphi|Du|\leq -\int_{\Omega}\varphi u^*\hbox{div}\,{\bf z}-\int_{\Omega} u{\bf z}\cdot \nabla \varphi\,dx=\langle ({\bf z}, Du), \varphi\rangle.
 $$
 
 On the other hand, since $\|{\bf z}\|_\infty\leq 1$, by Proposition \ref{Dist} and Remark \ref{R1}, we obtain that
 $$
 ({\bf z}, Du)\leq |Du|\quad\hbox{ as Radon measure }\;\Omega.
 $$
 So the identity holds for this case.
 
 {Suppose that $\gamma>1$. {From \eqref{Ek}}, by Lemma \ref{IdUZ} with $(T_k(u))^\gamma$ instead of $u^\gamma$, we have 
 \begin{align}
 \int_{\Omega}\varphi|D (T_k(u))^\gamma|&\leq \lambda\int_{\Omega}\frac{(T_k(u))^\gamma}{|x|}\varphi\,dx+\int_{\Omega}f \varphi\,dx-\int_{\Omega}(T_k(u))^\gamma{\bf z}\cdot\nabla \varphi\,dx\\
 &=-\int_{\Omega}(T_k(u)^\gamma)^*\hbox{div}\,{\bf z}-\int_{\Omega}(T_k(u))^\gamma{\bf z}\cdot\nabla \varphi\,dx\\
 &=\int_{\Omega}\varphi({\bf z}, D(T_k(u))^\gamma).
 \end{align}
 By Proposition \ref{Dist}  the inequality contrary holds. Thus
\begin{equation}
\label{IdT}
  \int_{\Omega}\varphi({\bf z}, D(T_k(u))^\gamma)=\int_{\Omega}\varphi |D(T_k(u))^\gamma|\quad\hbox{ for all  }\;0\leq\varphi\in C_c^\infty(\Omega).
\end{equation} 
 On the other hand, since $u\in BV_{\rm loc}(\Omega)$ and the map
  \begin{equation}
  \label{g}
  g(s)=\left\{
  \begin{array}{lcr}
  \left(\frac{s}{k}\right)^{1/\gamma}&\hbox{ if }& s> k,\\
   \left(\frac{s}{k}\right)^\gamma&\hbox{ if }&0\leq s\leq k,\\
    -\left(\frac{-s}{k}\right)^\gamma&\hbox{ if }& -k\leq s<0,\\
  -\left(\frac{-s}{k}\right)^{1/\gamma}&\hbox{ if }&s<-k.
  \end{array}
  \right.
  \end{equation}
 is increasing and Lipschitz, one has that $g(T_k(u))\in BV_{\rm loc}(\Omega)$.   Moreover, by \cite[Proposition 2.8]{Anzellotti1983} we have 
 $$
\theta({\bf z}, Dg(T_k(u)), x)=\theta({\bf z}, DT_k(u), x)\qquad |DT_k(u))|-\hbox{ a.e. in }\; \omega,
$$
for all open set $\omega\subset\subset\Omega$. But this, in turn, using \eqref{IdT}, implies that
$$
({\bf z}, g(T_k(u)))=|Dg(T_k(u))|\quad\hbox{ as Radon measures on }\;\Omega.
$$
Let 
$$
m_k(u):=g'(\widetilde{T_k(u)})\chi_{\Omega\setminus S_{T_k(u)}}+\frac{|g(T_k(u^+))-g(T_k(u^-))|}{|T_k(u^+)-T_k(u^-)|}\chi_{J_{T_k(u)}\cap J_u}\neq 0\qquad|DT_k(u)|-\hbox{ a.e. }x\in\omega,
 $$
 for all $\omega\subset\subset\Omega$; {here we used that $g$ is strictly increasing on $(0,k]$ and that, although $g'(0)=0$ when $\gamma>1$, the measure $|DT_k(u)|$ does not charge the level set $\{\widetilde{T_k(u)}=0\}$, since $\nabla T_k(u)=0$ a.e. there and $|D^cT_k(u)|(\{\widetilde{T_k(u)}\in E\})=0$ whenever $|E|=0$ (see \cite[Proposition 3.92]{AmbrosioFuscoPallara}).} Then
\begin{align}
\int_{\omega}({\bf z}, DT_k(u))&=\int_{\omega}\theta({\bf z}, DT_k(u), x)|DT_k(u)|\\
&=\int_{\omega} \theta({\bf z}, Dg(T_k(u)), x)|DT_k(u)|\\
&=\int_{\omega}\frac{1}{m_k(u)}\theta({\bf z}, Dg(T_k(u)), x)|Dg(T_k(u))|\\
&=\int_{\omega}\frac{1}{m_k(u)}({\bf z}, Dg(T_k(u)))\\
&=\int_{\omega}\frac{1}{m_k(u)}|Dg(T_k(u))|\\
&=\int_{\omega}|DT_k(u)|.
\end{align}
Letting $k\to \infty$, using \cite[Proposition 2.20]{OrtizPetitta2024}, we obtain that
$$
 \int_{\omega}({\bf z}, Du)=\int_{\omega}|Du|,
$$
for all $\omega\subset\subset\Omega$. Hence $({\bf z}, Du)=|Du|$ as Radon measures on $\Omega$.
   }
   
\fim

\begin{lemma}   For $0<\gamma\leq 1$, let $u\in BV(\Omega)$ and ${\bf z}\in X_{\mathcal{M}_{loc}}(\Omega)$  be as in \eqref{u2} and \eqref{Cz}, respectively, {one of the following holds:
\begin{equation}
\label{WT}
\lim_{\rho\downarrow0}\rho^{-N}\int_{\Omega\cap B_{\rho}(x)}u(y)\,dy=0\quad\text{ or }\quad \phi(x)+[\phi{\bf z},\nu](x)=0\quad\mathcal{H}^{N-1}-\text{ a.e.}\;x\in \partial\Omega,
\end{equation}
for all $0\leq \phi\in C^1(\overline{\Omega}) \cap L^1(\Omega,\hbox{div}\,{\bf z})$ with $ \sup_{\Omega}\phi{\leq} \phi\lfloor_{\partial\Omega}$. For $\gamma>1$, let $u^\gamma\in BV(\Omega)$ and ${\bf z}\in X_{\mathcal{M}_{loc}}(\Omega)$ be as in \eqref{u4} and \eqref{Cz1}, respectively, \eqref{WT} also holds .}

\end{lemma}
\dem
  First, let $0\leq \phi\in C^1(\overline{\Omega})\cap L^1(\Omega, \hbox{div}\,{\bf z})$. Suppose that $0<\gamma\leq 1$,. Use  $\phi T_k(u_n)$ as test function in \eqref{AuxP} to obtain
\begin{equation*}
\int_{\Omega}\phi|\nabla T_k(u_n)|^p\,dx+\int_{\Omega}T_k(u_n)|\nabla u_n|^{p-2}\nabla u_n\cdot \nabla \phi\,dx=\lambda\int_{\Omega}\frac{u_n^{p-1}\phi T_k(u_n)}{|x|^p+1/n}\,dx+\int_{\Omega}\frac{f_n\phi T_k(u_n)}{(u_n+1/n)^\gamma}\,dx.
\end{equation*}
By Young's inequality, it follows that
\begin{align}
\label{E}
&\int_{\Omega}\phi|\nabla T_k(u_n)|\leq \frac{1}{p}\int_{\Omega}\phi |\nabla T_k(u_n)|^p\,dx+\frac{p-1}{p}\int_{\Omega}\phi\,dx\\
\nonumber&=\frac{\lambda}{p}\int_{\Omega}\frac{u_n^{p-1}\phi T_k(u_n) }{|x|^p+1/n}\,dx+\frac{1}{p}\int_{\Omega}\frac{f_n\phi T_k(u_n)}{(u_n+1/n)^\gamma}\,dx-\int_{\Omega}T_k(u_n)|\nabla u_n|^{p-2}\nabla u_n\cdot \nabla \phi\,dx+\frac{p-1}{p}\int_{\Omega}\phi\,dx.
\end{align}
Letting $n\to\infty$ in \eqref{E}, we get
\begin{align}
\int_{\Omega}\phi|\nabla T_k(u_p)|\,dx&\leq\frac{\lambda}{p}\int_{\Omega}\frac{1}{|x|^p}u_p^{p-1}\phi T_k(u_p)\,dx+\frac{1}{p}\int_{\Omega}\frac{f\phi T_k(u_p)}{u_p^\gamma}\,dx-\int_{\Omega}T_k(u_p)|\nabla u_p|^{p-2}\nabla u_p\cdot \nabla \phi\,dx\\
&\qquad\qquad\qquad+\frac{p-1}{p}\int_{\Omega}\phi\,dx.
\end{align}
Since $u_p=0$ on $\partial\Omega$, one has
$$
\int_{\Omega}\phi|\nabla T_k(u_p)|+\int_{\partial\Omega}\phi|T_k(u_p)|\,d\mathcal{H}^{N-1}\leq \frac{\lambda}{p}\int_{\Omega}\frac{1}{|x|^p}u_p^{p-1}\phi T_k(u_p)\,dx+\frac{1}{p}\int_{\Omega}\frac{f\phi T_k(u_p)}{u_p^\gamma}\,dx+\frac{p-1}{p}\int_{\Omega}\phi\,dx.
$$
Passing to the limit as $p\to1^+$, we infer
\begin{align}
\int_{\Omega}\phi |DT_k(u)|+\int_{\partial\Omega}\phi |T_k(u)|\,d\mathcal{H}^{N-1}&\leq \lambda\int_{\Omega}\frac{1}{|x|}s(x)\phi T_k(u)\,dx+\int_{\Omega}\frac{f\phi T_k(u)}{u^\gamma}\,dx-\int_{\Omega}T_k(u){\bf z}\cdot\nabla \phi\,dx\\
&=-\int_{\Omega}(\phi T_k(u))^*\hbox{div}\,{\bf z}-\int_{\Omega}T_k(u){\bf z}\cdot\nabla \phi\,dx\\
&=-\int_{\partial\Omega}[T_k(u)\phi {\bf z},\nu]\,d\mathcal{H}^{N-1}+\int_{\Omega}({\bf z}, D(\phi T_k(u)))-\int_{\Omega} T_k(u){\bf z}\cdot \nabla \phi\,dx\\
&=-\int_{\partial\Omega}T_k(u)[\phi {\bf z},\nu]\,d\mathcal{H}^{N-1}+\int_{\Omega}\phi({\bf z}, DT_k(u))\\
&=-\int_{\partial\Omega}T_k(u)[\phi {\bf z},\nu]\,d\mathcal{H}^{N-1}+\int_{\Omega}\phi|DT_k(u)|,
\end{align}
where we used Lemma \ref{IdUZ} with $\phi T_k(u)$ instead of $u$,  {Proposition \ref{uz}}, \cite[Lemma 5.6]{Caselles2011},  and Lemma \ref{ZDu} with $T_k(u)$ instead of $u$. Hence
\begin{equation}
\label{1inq}
\int_{\partial\Omega}(\phi |T_k(u)|+ T_k(u)[\phi {\bf z}, \nu])\,d\mathcal{H}^{N-1}\leq 0.
\end{equation}                                                                                                                          
On the other hand, suppose that $ \sup_{\Omega}\phi\leq \phi\lfloor_{\partial\Omega}$. Thus, using \cite[Theorem 2.1]{Anzellotti1983} and as $\|{\bf z}\|_{L^\infty(\Omega, \mathbb{R}^N)}\leq 1$, it follows that
$$
|T_k(u)[\phi {\bf z}, \nu ]|\leq |T_k(u)|\sup_{\Omega}\phi\leq|T_k(u)|\phi\lfloor_{\partial\Omega}\quad\mathcal{H}^{N-1}-\hbox{a.e. in }\;\partial\Omega 
$$
and so, by \eqref{1inq},
\begin{equation}
\label{Tphi}
\int_{\partial\Omega}(\phi |T_k(u)|+ T_k(u)[\phi {\bf z}, \nu])\,d\mathcal{H}^{N-1}= 0.
\end{equation}

{Letting $k\to\infty$ in \eqref{Tphi}, since $(T_k(u))_{k\geq1}$ strictly converges in $BV(\Omega)$ to $u$,  we obtain
$$
\int_{\partial\Omega}(\phi u+u[\phi {\bf z}, \nu])\,d\mathcal{H}^{N-1}=0.
$$
Hence follows our result. }

Assume that $\gamma>1$. Use $\phi (T_k(u))^\gamma$ as test function in \eqref{AuxP} to obtain
\begin{align}
&\gamma\left(\frac{p}{\gamma+p-1}\right)^p\int_{\Omega}\phi|\nabla (T_k(u_n))^{(\gamma+p-1)/p}|^p\,dx+\int_{\Omega}(T_k(u_n))^\gamma|\nabla u_n|^{p-2}\nabla u_n\cdot \nabla \phi\,dx\\
&\quad=\lambda\int_{\Omega}\frac{u_n^{p-1}\phi (T_k(u_n))^\gamma}{|x|^p+1/n}\,dx+\int_{\Omega}\frac{f_n\phi (T_k(u_n))^\gamma}{(u_n+1/n)^\gamma}\,dx.
\end{align}

By Young inequality, one has that
\begin{align}
&\int_{\Omega}\phi |\nabla (T_k(u_n))^{(\gamma+p-1)/p}|\,dx\leq \frac{1}{p}\int_{\Omega}\phi |\nabla (T_k(u_n))^{(\gamma+p-1)/p}|^p+\frac{p-1}{p}\int_{\Omega}\phi\,dx\\
&=\frac{\lambda}{\gamma p}\left(\frac{\gamma+p-1}{p}\right)^p\int_{\Omega}\frac{u_n^{p-1}\phi (T_k(u_n))^\gamma}{|x|^p+1/n}\,dx+\frac{1}{\gamma p}\left(\frac{\gamma+p-1}{p}\right)^p\int_{\Omega}\frac{f_n\phi (T_k(u_n))^\gamma}{(u_n+1/n)^\gamma}\,dx\\
&\qquad-\frac{1}{\gamma p}\left(\frac{\gamma+p-1}{p}\right)^p\int_{\Omega}(T_k(u_n))^\gamma|\nabla u_n|^{p-2}\nabla u_n\cdot \nabla \phi\,dx+\frac{p-1}{p}\int_{\Omega}\phi\,dx.
\end{align}

Letting $n\to\infty$ above, one has
\begin{align}
\int_{\Omega}\phi |\nabla (T_k(u_p))^{(\gamma+p-1)/p}|\,dx&\leq \frac{1}{\gamma p}\left(\frac{\gamma+p-1}{p}\right)^p\left[\lambda\int_{\Omega}\frac{u_p^{p-1}\phi(T_k(u_p))^\gamma}{|x|^p}\,dx+\int_{\Omega}\frac{f\phi(T_k(u_p))^\gamma}{u_p^\gamma}\,dx\right.\\
&\qquad\left.-\int_{\Omega}(T_k(u_p))^\gamma|\nabla u_p|^{p-2}\nabla u_p\cdot\nabla \phi\,dx\right]+\frac{p-1}{p}\int_{\Omega}\phi\,dx.
\end{align}

Passing to the limit, as $p\to 1^+$, above, and after some simplifications we obtain
\begin{align}
\int_{\Omega}\phi |D(T_k(u))^\gamma|+\int_{\partial\Omega}\phi||(T_k(u))^\gamma|\,d\mathcal{H}^{N-1}&\leq\lambda\int_{\Omega}\frac{s(x)\phi(T_k(u))^\gamma}{|x|}\,dx+\int_{\Omega}\frac{f\phi(T_k(u))^\gamma}{u^\gamma}\,dx\\
&\qquad-\int_{\Omega}(T_k(u))^\gamma{\bf z}\cdot\nabla \phi\,dx,
\end{align}
which, in turn, implies that
\begin{equation}
\label{G1}
\int_{\partial\Omega}\phi|(T_k(u))^\gamma|\,d\mathcal{H}^{N-1}+\int_{\partial\Omega}(T_k(u))^\gamma[\phi{\bf z},\nu]\,d\mathcal{H}^{N-1}\leq0.
\end{equation}

On the other hand, suppose that $\phi\lfloor_{\partial\Omega}\geq \sup_{\Omega}\phi$. Thus, using \cite[Theorem 2.1]{Anzellotti1983} and as $\|{\bf z}\|_{L^\infty(\Omega, \mathbb{R}^N)}\leq 1$, we obtain
$$
|(T_k(u))^\gamma[\phi{\bf z}, \nu]|\leq |(T_k(u))^\gamma| \phi\lfloor_{\partial\Omega}\quad\mathcal{H}^{N-1}-\hbox{ a.e. in }\quad\partial\Omega.
$$
Hence, by \eqref{G1},
$$
\int_{\partial\Omega}\phi|(T_k(u))^\gamma|\,d\mathcal{H}^{N-1}+\int_{\partial\Omega}(T_k(u))^\gamma[\phi{\bf z},\nu]\,d\mathcal{H}^{N-1}=0,
$$
and so, letting $k\to \infty$, 
{
$$
\int_{\partial\Omega}(\phi u^\gamma+u^\gamma[\phi{\bf z}, \nu])\,d\mathcal{H}^{N-1}=0.
$$
But this, in turn, implies
$$
u^\gamma(x)=0\quad\hbox{ or }\quad\phi(x)+[\phi{\bf z}, \nu](x)=0
$$
and if $u^\gamma=0\quad\mathcal{H}^{N-1}a.e.$ in $\partial\Omega$, by H\"{o}lder inequality one has
$$
\rho^{-N}\int_{\Omega\cap B_{\rho}(x)}u(y)\,dy\leq\left(\rho^{-N}\int_{\Omega\cap B_{\rho}(x)}u^\gamma(y)\,dy\right)^{1/\gamma}(\rho^{-N}|\Omega\cap B_{\rho}(x)|)^{1-1/\gamma}\to0\quad\hbox{ as }\;\rho\to0.
$$
}

\fim
 
\begin{lemma}
\label{chiSol}
  For $0<\gamma\leq 1$, let $u\in BV(\Omega)$ and ${\bf z}\in X_{\mathcal{M}_{loc}}(\Omega)$ are as in \eqref{u2} and \eqref{Cz}, respectively, we have that $\chi_{\{u>0\}}\in BV_{loc}(\Omega)$ and the following identity holds
\begin{equation}
\label{chiu}
-\chi^*_{\{u>0\}}\hbox{div}\,{\bf z}=\frac{\lambda}{|x|}\chi_{\{u>0\}}+\frac{f}{u^\gamma}\quad\hbox{ in }\; \mathcal{D}'(\Omega).
\end{equation}
For $\gamma>1$, let $u^\gamma\in BV(\Omega)$ and ${\bf z}\in X_{\mathcal{M}_{loc}}(\Omega)$ be as in \eqref{u4} and \eqref{Cz1}, respectively, one also has that \eqref{chiu} holds.
\end{lemma}

\dem
For $0<\gamma\leq 1$. We start showing that $\chi_{\{u>0\}}$ belongs to $BV_{loc}(\Omega)$. Indeed, for each $n\in \mathbb{N}$, consider the function $h_n\>:\>[0, +\infty)\to [0, +\infty)$ defined by
$$
h_n(s)=nT_{\frac{1}{n}}(s^+).
$$
From \eqref{IneqDu}, it follows that
$$
 \int_{\Omega}\varphi |Dh_n(u)|\leq \lambda\int_{\Omega}\frac{1}{|x|} s(x) h_n(u)\varphi\,dx+\int_{\Omega}\frac{f}{u^\gamma}h_n(u)\varphi\,dx-\int_{\Omega}h_n(u){\bf z}\cdot\nabla\varphi,dx\quad\hbox{ for all }\; n\in \mathbb{N}
$$
and so, letting $n\to \infty$, by lower semicontinuity,
\begin{equation}
\label{KeyInq}
\int_{\Omega}\varphi|D\chi_{\{u>0\}}|\leq \lambda\int_{\Omega}\frac{1}{|x|}\chi_{\{u>0\}}\varphi\,dx+\int_{\Omega}\frac{f}{u^\gamma}\chi_{\{u>0\}}\varphi\,dx-\int_{\Omega}\chi_{\{u>0\}}{\bf z}\cdot\nabla\varphi\,dx.
\end{equation}
Hence $\chi_{\{u>0\}}\in BV_{loc}(\Omega)$.

In order to prove the identity \eqref{chiu}, by Proposition \ref{Dist} and Definition \ref{zDu}, \eqref{KeyInq} becomes
$$
-\int_{\Omega}\chi^*_{\{u>0\}}\varphi\hbox{div}\,{\bf z}\leq \lambda\int_{\Omega}\frac{1}{|x|}\chi_{\{u>0\}}\varphi\,dx+\int_{\Omega}\frac{f}{u^\gamma}\chi_{\{u>0\}}\varphi\,dx.
$$
On the other hand, arguing as for \eqref{1Est}, one can show that
$$
-\int_{\Omega}\chi^*_{\{u>0\}}\varphi\hbox{div}\,{\bf z}\geq \lambda\int_{\Omega}\frac{1}{|x|}\chi_{\{u>0\}}\varphi\,dx+\int_{\Omega}\frac{f}{u^\gamma}\chi_{\{u>0\}}\varphi\,dx.
$$
Thus
$$
-\int_{\Omega}\chi_{\{u>0\}}^*\varphi\hbox{div}\,{\bf z}=\lambda\int_{\Omega}\frac{1}{|x|}\chi_{\{u>0\}}\varphi\,dx+\int_{\Omega}\frac{f}{u^\gamma}\chi_{\{u>0\}}\varphi\,dx\quad\hbox{ for all }\; \varphi\in C^1_c(\Omega).
$$
{Suppose that $\gamma>1$. From \eqref{Ek}, it follows that
\begin{equation}
\label{Enb}
\int_{\Omega}\varphi|D (h_n(u))^\gamma|\leq \lambda\int_{\Omega}\frac{s(x)(h_n(u))^\gamma}{|x|}\varphi\,dx+\int_{\Omega}\frac{f}{u^\gamma}(h_n(u))^\gamma \varphi\,dx-\int_{\Omega}(h_n(u))^\gamma{\bf z}\cdot\nabla \varphi\,dx,
\end{equation}
Letting $n\to \infty$, one has
$$
\int_{\Omega}\varphi|D\chi_{\{u>0\}}|\leq \lambda\int_{\Omega}\frac{1}{|x|}\chi_{\{u>0\}}\varphi\,dx+\int_{\Omega}\frac{f}{u^\gamma}\chi_{\{u>0\}}\varphi\,dx-\int_{\Omega}\chi_{\{u>0\}}{\bf z}\cdot\nabla\varphi\,dx.
$$
Hence, reasoning as the case $0<\gamma\leq1$, one can prove that \eqref{chiu} holds as well.
\fim

\begin{remark}\label{remKey} By \eqref{KeyInq} and {Lemma \ref{chiSol}}, we obtain
$$
\int_{\Omega}\varphi|D\chi_{\{u>0\}}|\leq -\int_{\Omega}\varphi\chi^*_{\{u>0\}}\hbox{div}\,{\bf z}-\int_{\Omega}\chi_{\{u>0\}}{\bf z}\cdot\nabla \varphi\,dx=\int_{\Omega}\varphi({\bf z}, D\chi_{\{u>0\}})\quad\hbox{ for all }\;0\leq \varphi\in C^\infty_c(\Omega).
$$
and, by Proposition \ref{Dist}, 
$$
({\bf z}, D\chi_{\{u>0\}})\leq |D\chi_{\{u>0\}}|.
$$
Thus from both inequalities above infer that
$$
({\bf z}, D\chi_{\{u>0\}})=|D\chi_{\{u>0\}}|\quad\hbox{ as Radon measure on }\;\Omega.
$$
\end{remark}

{Lemmas \ref{ConvZ}--\ref{chiSol} show that the pair $(u,{\bf z})$ constructed above satisfies conditions $(a)$--$(e)$ of Definition \ref{DS}. If $0<\gamma\leq1$, estimate \eqref{u2} also gives $u\in BV(\Omega)$, so that the existence part of Theorem \ref{mainTh} is proved in this case. If $\gamma>1$ we only know, so far, that $u\in BV_{\rm loc}(\Omega)$ and $u^\gamma\in BV(\Omega)$: the finiteness of the energy of $u$ itself is the content of the next section. {The proof of Theorem \ref{mainTh} is then completed at the end of Section \ref{Br}, once the $L^\infty$ bound $(ii)$ is available.}}

\section{{Positive data and finite energy solutions}}
\label{PosDat}

\subsection{{A positive datum implies $u>0$}}
{As we already mentioned, positive solutions to \eqref{P} are found whenever the datum $f\in L^{N,\infty}(\Omega)$ is positive.} In order to show this, we need the following result.

\begin{lemma}
\label{LKey}
Let $f\in  L^{N,\infty}(\Omega)$ be a positive function. Let $u\in BV(\Omega)$ be a solution to \eqref{P} in sense of Definition \ref{DS}. Then  
$$
u(x)>0\quad\hbox{ a.e. }\; x\in \Omega.
$$
\end{lemma}
\dem
Note that, from Definition $\ref{DS}(a)$, one has 
\begin{equation}
\label{FT}
\int_{\Omega}\frac{f}{u^\gamma}\varphi\,dx<\infty\quad\hbox{ for all }\; 0\leq \varphi\in C^\infty_c(\Omega).
\end{equation}
Hence, one can infer 
$$
|\{x\in \Omega\colon \varphi(x)=0=u(x)\}|=0.
$$
On the other hand, by Chebychev inequality, it follows that
$$
\left|\left\{x\in \Omega\colon \frac{f(x)}{u(x)^\gamma}\varphi(x)=+\infty\right\}\right|=0.
$$
But this implies that
$$
|\{x\in \Omega\colon \varphi(x)>0,\; u(x)=0\}|=0,
$$
since $f>0$ a.e. in $\Omega$.

Consequently
$$
u(x)>0\quad\hbox{ a.e. }\; x\in \Omega.
$$
\fim

\begin{theorem}
\label{secTh}
Let $f$ be a positive function in {$L^{N,\infty}(\Omega)$} {and let $u$ be a solution to \eqref{P} in the sense of Definition \ref{DS}}. Then the following identity holds
\begin{itemize}
\item [$(c)'$] $-\hbox{div}\,{\bf z}=\frac{\lambda}{|x|}+\frac{f}{u^\gamma}\quad\hbox{ in }\;\mathcal{D}'(\Omega)$.
\end{itemize}
instead of $(c)$ of Definition \ref{DS}. {Moreover, we have that ${\rm div}\,{\bf z}\in L^1(\Omega).$}
\end{theorem}
\dem  Clearly, the conditions  $(a), (b)$ and $(d)$ of Definition \ref{DS} hold. {We show that $(c)'$ is also valid:}
\begin{itemize}
\item [$(c)'$] From \eqref{KeyInq}, by Lemma \ref{LKey}, one has that
$$
\int_{\Omega}{\bf z}\cdot \nabla\varphi\,dx\leq \lambda\int_{\Omega}\frac{1}{|x|}\varphi\,dx+\int_{\Omega}\frac{f}{u^\gamma}\varphi\,dx,\quad\hbox{ for all }\;0\leq \varphi\in C^\infty_c(\Omega).
$$
Thus, by \eqref{SupS},
\begin{equation}
\label{Mz}
\int_{\Omega}{\bf z}\cdot \nabla \varphi\,dx=\lambda\int_{\Omega}\frac{1}{|x|}\varphi\,dx+\int_{\Omega}\frac{f}{u^\gamma}\varphi\,dx\quad\hbox{ for all }\;\varphi \in C^\infty_c(\Omega),
\end{equation}
or equivalently
$$
-\hbox{div}\,{\bf z}=\frac{\lambda}{|x|}+\frac{f}{u^\gamma}\quad\hbox{ in }\;\mathcal{D}'(\Omega).
$$
Hence  $-\hbox{div}\,{\bf z}$ is a Radon measure in $\Omega$. 
\end{itemize}
{{\bf Claim.} $f/u^\gamma\in L^1(\Omega)$.
In order to prove this claim we follow the arguments in \cite[Theorem 5.1]{DeCiccoGiachettiSegura2019}. By \cite[Theorem 2.16]{Giusti1984}, there exists a sequence $w_n\in W^{1,1}(\Omega)$ such that
\begin{itemize}
\item [$(i)$] $ w_n=1$ on $\partial \Omega$;
\item [$(ii)$] $\|w_n\|_{L^1(\Omega)}\leq 1/n \mathcal{H}^{N-1}(\partial\Omega)$;
\item [$(iii)$] $\int_{\Omega}|\nabla w_n|\,dx\leq A \mathcal{H}^{N-1}(\partial\Omega)$, 
\end{itemize} 
with $A$ depend on $\partial\Omega$, but independent of $n$.
Observe that $1-w_n\in W^{1,1}_0(\Omega)$. Hence, there exists $(\varphi_{n,m})\subset C^\infty_c(\Omega)$ such that $\phi_{n,m}\to 1-w_n$ in $W^{1,1}_0(\Omega)$ as $m\to \infty$. Thus, by Vainberg theorem, up to subsequences, $\phi_{n,m}(x)\to 1-w_n(x)$ a.e. $x\in \Omega$ as $m\to\infty$ and so,
\begin{align}
\int_{\Omega}\frac{f}{u^\gamma}|1-w_n|\,dx&\leq \liminf_{m\to \infty}\int_{\Omega}\frac{f}{u^\gamma}|\phi_{n,m}|\,dx\\
&=\liminf_{m\to\infty}\left(\int_{\Omega}\frac{f}{u^\gamma}\phi_{n,m}^+\,dx+\int_{\Omega}\frac{f}{u^\gamma}\phi_{n,m}^-\,dx\right)\\
&=\liminf_{m\to\infty}\lim_{\epsilon\to0}\left(\int_{\Omega}\frac{f}{u^\gamma}\rho_\epsilon\ast\phi_{n,m}^+\,dx+\int_{\Omega}\frac{f}{u^\gamma}\rho_\epsilon\ast\phi_{n,m}^-\,dx\right)\\
&=\liminf_{m\to\infty}\lim_{\epsilon\to 0}\left(\int_{\Omega_{\epsilon_0}}{\bf z}\cdot \nabla (\rho_\epsilon\ast\phi_{n,m}^+)\,dx-\lambda\int_{\Omega}\frac{\rho_\epsilon\ast\phi_{n,m}^+}{|x|}\,dx\right.\\
&\qquad\left.+\int_{\Omega_{\epsilon_0}}{\bf z}\cdot \nabla (\rho_\epsilon\ast\phi_{n,m}^-)\,dx-\lambda\int_{\Omega}\frac{\rho_\epsilon\ast\phi_{n,m}^-}{|x|}\,dx\right)\\
&=\liminf_{m\to\infty}\left(\int_{\Omega_{\epsilon_0}}{\bf z}\cdot\nabla |\phi_{n,m}|\,dx-\lambda\int_{\Omega}\frac{s(x)}{|x|}|\phi_{n,m}|\,dx\right)\\
&\leq\left(\|{\bf z}\|_{L^\infty(\Omega, \mathbb{R}^N)}+\frac{\lambda^-}{N-1}\right)\int_{\Omega} |\nabla w_n|\,dx\\
&\leq \left(1+\frac{\lambda^-}{N-1}\right)A\mathcal{H}^{N-1}(\partial\Omega),
\end{align}
where
$$
\Omega_{\epsilon_0}=\{x\in \Omega\colon {\rm dist}\,(x,\partial\Omega)>\epsilon_0\},\qquad {\rm supp}\,\phi_{n,m}\subset \Omega_{\epsilon_0} \quad\hbox{ and }\quad 0<\epsilon<\frac{\epsilon_0}{2}.
$$
Letting $n\to \infty$ above, we obtain
\begin{equation}
\label{L15}
\int_{\Omega}\frac{f}{u^\gamma}\,dx\leq\left(1+\frac{\lambda^-}{N-1}\right) A\mathcal{H}^{N-1}(\partial\Omega),
\end{equation}
that is, $f/u^\gamma\in L^1(\Omega)$ and so the Claim holds. Consequently, by \eqref{B}, we have that ${\rm div}\,{\bf z}\in L^1(\Omega)$.

{Let us emphasize that the right-hand side of \eqref{L15} does not depend on $f$: no matter how large (or how small) the datum is, the quantity $\int_{\Omega}fu^{-\gamma}$ is controlled by a constant depending only on $\Omega$, $N$ and $\lambda$. This is precisely the mechanism through which the singular term prevents the degeneracy phenomenon of the case $\gamma=0$ described in \cite{ct, mst, OrtizPetitta2024}, and it is what makes all the estimates of this section independent of the size of $f$.}
}
\fim

\begin{remark}
\label{Ra}
Let $f$ be a positive measurable function in $L^{N,\infty}(\Omega)$ and let $u$ be a solution to \eqref{P} in sense of Definition \ref{DS}. {Since
$$
-\hbox{div}\,{\bf z}\in L^1(\Omega),
$$
in \eqref{WT}, we may take $\phi=1$  to obtain
$$
\lim_{\rho\downarrow0}\rho^{-N}\int_{\Omega\cap B_{\rho}(x)}u(y)\,dy=0\quad\text{ or}\quad[{\bf z}, \nu](x)=-1\quad\mathcal{H}^{N-1}-\hbox{a.e.}\; x\in \partial\Omega.
$$
}
\end{remark}
{
\subsection{{Finite energy: the case of a positive datum}}
{Following closely} \cite{maop} we shall show that the limit of $(u_p)_{p>1}$ in \eqref{u4} has a finite energy. From now on we assume $\gamma>1$. If $f>0$ then we {readily} have the following:  
\begin{theorem}
\label{FE}
Let $\gamma>1$. Let $f$ be a positive measurable function in $L^{N,\infty}(\Omega)$. If $u$ the limit found in \eqref{u4}, then $u\in BV(\Omega)$.
\end{theorem}
\dem Define $\alpha_n\colon \mathbb{R}\to \mathbb{R}$ by
\begin{equation}
\alpha_n(s)=\left\{
\begin{array}{lcr}
(s+1/n)^{1/\gamma}-1/n^{1/\gamma}&\hbox{if}&s\geq 0,\\
1/n^{1/\gamma}-(1/n-s)^{1/\gamma}&\hbox{if}&s<0.
\end{array}
\right.
\end{equation}
Note that $\alpha_n$ is Lipschitz increasing function that belongs to $C^1(\mathbb{R})$. Hence, $\alpha_n(T_k(u)^\gamma)\in BV(\Omega)$ and, by \cite[Proposition 2.8]{Anzellotti1983},
\begin{equation}
\label{fTk}
({\bf z}, D\alpha_n(T_k(u)^\gamma))=|D\alpha_n(T_k(u)^\gamma)| \quad\hbox{ as finite Radon measures on}\;\Omega.
\end{equation}
Use Green Formula and \eqref{fTk} to obtain
$$
\int_{\Omega}|D\alpha_n(T_k(u)^\gamma)|=\int_{\partial\Omega}\alpha_n(T_k(u)^\gamma)[{\bf z}, \nu]\, d\mathcal{H}^{N-1}-\int_{\Omega}\alpha_n(T_k(u)^\gamma){\rm div}\,{\bf z}\,dx\leq -\int_{\Omega}\alpha_n(T_k(u)^\gamma){\rm div}\,{\bf z}\,dx,
$$
since $\int_{\partial\Omega}\alpha_n(T_k(u)^\gamma)[{\bf z}, \nu]\, d\mathcal{H}^{N-1}\leq 0$. Letting $n\to \infty$, we obtain that 
$$
\int_{\Omega}|DT_k(u)|\leq\liminf_{n\to\infty}\int_{\Omega}|D\alpha_n(T_k(u)^\gamma)|\leq \liminf_{n\to \infty}-\int_{\Omega}\alpha_n(T_k(u)^\gamma){\rm div}\,{\bf z}\,dx=-\int_{\Omega}T_k(u){\rm div}\,{\bf z}\,dx.
$$
 Since $u\in L^1(\Omega, {\rm div}\,{\bf z}\,dx)$ (Lemma \ref{IdUZ}), one has
$$
\int_{\Omega}|Du|\leq \liminf_{k\to \infty}\int_{\Omega}|DT_k(u)|\leq \liminf_{k\to\infty}\left(\int_{\Omega}T_k(u) |{\rm div}\,{\bf z}|\,dx\right)=\int_{\Omega}u|{\rm div}\,{\bf z}|\,dx.
$$
Hence $u\in BV(\Omega)$ as desired.
}
\fim

\subsection{{Finite energy: the case of a nonnegative datum}}
{If $f$ is merely nonnegative}, for each $n\in\mathbb{N}$, we consider the following auxiliary problem associated to \eqref{P},
\begin{equation}
\label{AP}
\left\{
\begin{array}{rclr}
-\Delta_1u&=&\lambda\frac{u}{|u|}+\frac{f_n}{u^\gamma}&\hbox{ in }\;\Omega,\\
u&=&0&\hbox{ on }\;\partial\Omega,
\end{array}
\right.
\end{equation}
where $f_n=f+1/n$, $f$ is a nonnegative function in $L^{N,\infty}(\Omega)$. Note that, by Theorem \ref{FE}, there exists a solution $u_n\in BV(\Omega)$ to \eqref{AP}, that is,  there exist ${\bf z}_n\in X_{\mathcal{M}}(\Omega)$ with ${\rm div}\,{\bf z}_n\in L^1(\Omega)$ and $\|{\bf z}_n\|_{L^{\infty}(\Omega, \mathbb{R}^N)}\leq 1$ such that
\begin{itemize}
\item [$(1)$] $f_n/u_n^\gamma\in L^1(\Omega)$;
\item [$(2)$] $-{\rm div}\,{\bf z}_n=\frac{\lambda}{|x|}+\frac{f_n}{u_n^\gamma}$ in $\mathcal{D}'(\Omega)$;
\item [$(3)$] $({\bf z}_n, Du_n)=|Du_n|$ as finite Radon measures on $\Omega$;
\item [$(4)$] one of the following holds:
$$
\lim_{\rho\downarrow0}\rho^{-N}\int_{\Omega\cap B_{\rho}(x)}u_n(y)\,dy=0\quad\text{ or }\quad[{\bf z}_n, \nu]=-1\quad\mathcal{H}^{N-1}-\text{ a.e.}\; x\in\partial\Omega.
$$
\end{itemize}

We shall show that $(u_n)_{n\geq1}$ is a bounded sequence in $BV(\Omega)$ and so, there exists $u\in BV(\Omega)$ such that, up to subsequences, 
\begin{eqnarray}
\label{CUn5}
\nonumber Du_n\rightharpoonup Du&*-\hbox{weak in }\; BV(\Omega),\\
u_n\to u&\hbox{ in }\; L^q(\Omega), \; q\in [1,N/(N-1)),\\
\nonumber u_n\rightharpoonup u&\hbox{in }\; L^{1^*}(\Omega),
\end{eqnarray}
as $n\to\infty$. We shall also show that $u$ is a solution to \eqref{P} in sense of Definition \ref{DS}. 

\begin{lemma} Let $(u_n)_{n\geq1}$ be the sequence of solutions obtained of  \eqref{AP}. Then such sequence is bounded in $BV(\Omega)$ and so, holds \eqref{CUn5} and 
there exists $u^\gamma\in BV(\Omega)$ such that, up to subsequences,
\begin{eqnarray*}
Du_n^\gamma\rightharpoonup Du&*-\text{ weak in }\; BV(\Omega),\\
u_n^\gamma\to u^\gamma&\text{ in }\; L^q(\Omega),\; q\in [1, 1^*),\\
u_n^\gamma\rightharpoonup u^\gamma&\text{ in }\; L^{1^*}(\Omega),
\end{eqnarray*}
as $n\to \infty$.
\end{lemma}

\dem By Green formula, we have that
$$
\int_{\Omega}({\bf z}_n, Du_n)+\int_{\Omega}u_n {\rm div}\,{\bf z}_n\,dx=\int_{\partial\Omega}u_n[{\bf z}_n, \nu]\,d\mathcal{H}^{N-1}.
$$
From $(3)$ and $(4)$, it follows that
$$
\int_{\Omega}|Du_n|+\int_{\partial\Omega}u_n\,d\mathcal{H}^{N-1}=-\int_{\Omega}u_n{\rm div}\,{\bf z}_n\,dx.
$$
Since
\begin{align}
-\int_{\Omega}u_n{\rm div}\,{\bf z}\,dx&=\lambda\int_{\Omega}\frac{u_n}{|x|}\,dx+\int_{\Omega}\frac{f_n}{u_n^\gamma}u_n\,dx\\
&\leq \frac{\lambda^+}{N-1}\left(\int_{\Omega}|Du_n|+\int_{\partial\Omega}|u_n|\,d\mathcal{H}^{N-1}\right)+\left(\int_{\Omega} f_n\,dx\right)^{1/\gamma}\left(\int_{\Omega}\frac{{f_n}}{u_n^\gamma}\,dx\right)^{1-1/\gamma}\\
&\leq \frac{\lambda^+}{N-1}\left(\int_{\Omega}|Du_n|+\int_{\partial\Omega}|u_n|\,d\mathcal{H}^{N-1}\right)+\|f_n\|_{L^1(\Omega)}^{1/\gamma}\left(\left(1+\frac{\lambda^-}{N-1}\right)A\mathcal{H}^{N-1}(\partial\Omega)\right)^{1-1/\gamma},
\end{align}
where was used \eqref{L15}, one has
$$
\left(1-\frac{\lambda^+}{N-1}\right)\left(\int_{\Omega}|Du_n|+\int_{\partial\Omega}|u_n|\,d\mathcal{H}^{N-1}\right)\leq \|f_n\|_{L^1(\Omega)}^{1/\gamma}\left(\left(1+\frac{\lambda^-}{N-1}\right)A\mathcal{H}^{N-1}(\partial\Omega)\right)^{1-1/\gamma}.
$$
Hence $(u_n)_{n\geq 1}$ is bounded in $BV(\Omega)$ and so, there exists $u\in BV(\Omega)$ such that, up to subsequences,
\begin{eqnarray}
\nonumber Du_n\rightharpoonup Du&*-\hbox{weak in }\; BV(\Omega),\\
\nonumber u_n\to u&\hbox{ in }\; L^q(\Omega), \; q\in [1,1^*),\\
\nonumber u_n\rightharpoonup u&\hbox{in }\; L^{1^*}(\Omega),
\end{eqnarray}
as $n\to\infty$. Similarly, since $u_n^\gamma\in BV(\Omega)$ and it satisfies
$$
-\int_{\Omega}u_n^\gamma{\rm div}\,{\bf z}_n\,dx=\lambda\int_{\Omega}\frac{u_n^\gamma}{|x|}\,dx+\int_{\Omega}f_n\,dx,
$$
there exists $u^\gamma\in BV(\Omega)$ such that, up to subsequences,
\begin{eqnarray*}
Du_n^\gamma\rightharpoonup Du&*-\text{ weak in }\; BV(\Omega),\\
u_n^\gamma\to u^\gamma&\text{ in }\; L^q(\Omega),\; q\in [1, 1^*),\\
u_n^\gamma\rightharpoonup u^\gamma&\text{ in }\; L^{1^*}(\Omega),
\end{eqnarray*}
as $n\to \infty$.
\fim

{\bf Proof of the finite energy property in Theorem \ref{mainTh}.}
We shall show that $u$ as found above satisfies the condition $(a)-(e)$ of Definition \ref{DS}.

${\bf (a)}$ Since $\|{\bf z}_n\|_{L^\infty(\Omega, \mathbb{R}^N)}\leq 1$, there exists ${\bf z}\in L^\infty(\Omega,\mathbb{R}^N)$ such that, up to subsequences,
$$
{\bf z}_n\rightharpoonup {\bf z}\quad *-\hbox{ weak in }\; L^\infty(\Omega, \mathbb{R}^N),
$$
as $n\to \infty$, and $\|{\bf z}\|_{L^\infty(\Omega,\mathbb{R}^N)}\leq1$. But this and Fatou's lemma  in $(2)$ imply that
\begin{equation}
\label{R}
\int_{\Omega}{\bf z}\cdot \nabla \varphi\,dx\geq\lambda\int_{\Omega}\frac{\varphi}{|x|}\,dx+\int_{\Omega}\frac{f}{u^\gamma}\varphi\,dx\quad\text{ for all }\;0\leq \varphi\in C_c^1(\Omega).
\end{equation}
Hence $-{\rm div}\,{\bf z}$ is a Radon measure on $\Omega$ and $f/u^\gamma\in L^1_{\rm loc}(\Omega)$.

${\bf ( b)}$ Let $\delta>0$. Define $S_\delta\colon\mathbb{R}\to \mathbb{R}$ as
\begin{equation}
S_\delta(t)=\left\{
\begin{array}{lcr}
0&\text{ if }&t\leq \delta,\\
\frac{t-\delta}{\delta}&\text{ if }&\delta<t<2\delta,\\
1&\text{ if }&t\geq 2\delta.
\end{array}
\right.
\end{equation}
Let $0\leq \varphi\in C_c^1(\Omega)$. Use the fact that $({\bf z}_n, DS_\delta(u_n))=|DS_\delta(u_n)|$ and $(2)$ in Definition \ref{zDu} to obtain
$$
\int_{\Omega}\varphi|DS_\delta(u_n)|+\int_{\Omega}S_\delta(u_n){\bf z}_n\cdot \nabla\varphi\,dx=\lambda\int_{\Omega}\frac{S_\delta(u_n)\varphi}{|x|}\,dx+\int_{\Omega}\frac{f_n}{u_n^\gamma}S_\delta(u_n)\varphi\,dx.
$$
Letting $n\to\infty$ above, we obtain
$$
\int_{\Omega}\varphi|DS_\delta(u)|+\int_{\Omega}S_\delta(u){\bf z}\cdot\nabla\varphi\,dx\leq\lambda\int_{\Omega}\frac{S_\delta(u)\varphi}{|x|}\,dx+\int_{\Omega}\frac{f}{u^\gamma}S_\delta(u)\varphi\,dx,
$$
where was used the fact that $S_\delta(u_n)\to S_\delta(u)$ in $L^1(\Omega)$ as $n\to \infty$. Hence, passing to limit as $\delta\to0$, one has 
\begin{equation}
\label{E0}
\int_{\Omega}\varphi|D\chi_{\{u>0\}}|+\int_{\Omega}\chi_{\{u>0\}}{\bf z}\cdot\nabla \varphi\,dx\leq\lambda\int_{\Omega}\frac{\chi_{\{u>0\}}\varphi}{|x|}\,dx+\int_{\Omega}\frac{f}{u^\gamma}\chi_{\{u>0\}}\varphi\,dx.
\end{equation}
But this, in turn, implies that $\chi_{\{u>0\}}\in BV_{\rm loc}(\Omega)$.

${\bf(c)}$ Note that, from \eqref{E0} one can infer that
\begin{equation}
\label{d1}
-\int_{\Omega}\chi_{\{u>0\}}^*\varphi{\rm div}\,{\bf z}\leq \lambda\int_{\Omega}\frac{\chi_{\{u>0\}}\varphi}{|x|}\,dx+\int_{\Omega}\frac{f}{u^\gamma}\chi_{\{u>0\}}\varphi\,dx.
\end{equation}
In order to estimate the contrary inequality use $\rho_{\epsilon}\ast\chi_{\{u>0\}}\varphi$ as test function in \eqref{R} to obtain, after of pass to the limit as $\epsilon\to0$,
\begin{equation}
\label{d2}
-\int_{\Omega}\chi_{\{u>0\}}^*\varphi{\rm div}\,{\bf z}\geq \lambda\int_{\Omega}\frac{\chi_{\{u>0\}}\varphi}{|x|}\,dx+\int_{\Omega}\frac{f}{u^\gamma}\chi_{\{u>0\}}\varphi\,dx.
\end{equation}
Thus, from \eqref{d1} and \eqref{d2}, $(c)$ holds.

${\bf(d)}$ Note that, since$({\bf z}_n, Du_n)=|Du_n|$, one has, by \cite[Proposition 2.8]{Anzellotti1983},
\begin{equation}
\label{zDgn}
({\bf z}_n, Dg(T_k(u_n)))=|Dg(T_k(u_n))|\quad\text{ as Radon measures on }\;\Omega,
\end{equation}
where $g$ is defined in \eqref{g}. In order to prove $(d)\; ({\bf z}, Du)=|Du|$ we shall show that
$$
({\bf z}, Dg(T_k(u)))=|Dg(T_k(u))|\quad\text{ as Radon measures on }\; \Omega.
$$
Indeed, by Definition \ref{zDu}, it follows that
\begin{align}
\int_{\Omega}\varphi|Dg(T_k(u_n))|&=-\int_{\Omega}g(T_k(u_n))^*\varphi {\rm div}\,{\bf z}_n\,dx-\int_{\Omega}g(T_k(u_n)){\bf z}_n\cdot\nabla\varphi\,dx\\
&=\lambda\int_{\Omega}\frac{g(T_k(u_n))\varphi}{|x|}\,dx+\int_{\Omega}\frac{f_n}{u_n^\gamma}g(T_k(u))\varphi\,dx-\int_{\Omega}g(T_k(u_n)){\bf z}_n\cdot\nabla \varphi\,dx.
\end{align}
Letting $n\to \infty$, we obtain
\begin{equation}
\label{Dg5}
\int_{\Omega}\varphi|Dg(T_k(u))|\leq\lambda\int_{\Omega}\frac{g(T_k(u))\varphi}{|x|}\,dx+\int_{\Omega}\frac{f}{u^\gamma}g(T_k(u))\varphi\,dx-\int_{\Omega}g(T_k(u)){\bf z}\cdot\nabla\varphi\,dx.
\end{equation}
and, as $({\bf z}, g(T_k(u)))\leq |Dg(T_k(u))|$, using Definition \ref{zDu}, we infer that
$$
-\int_{\Omega}g(T_k(u))^*\varphi {\rm div}\,{\bf z}\leq \lambda\int_{\Omega}\frac{g(T_k(u))\varphi}{|x|}\,dx+\int_{\Omega}\frac{f}{u^\gamma}g(T_k(u))\varphi\,dx.
$$
Use $\rho_\epsilon\ast g(T_k(u))\varphi$ as test function in \eqref{R} to obtain the contrary inequality. So, using \cite[Propositions 3.64 c), 3.69 c)] {AmbrosioFuscoPallara}
\begin{equation}
\label{gDiv}
-\int_{\Omega}g(T_k(u^*))\varphi {\rm div}\,{\bf z}= \lambda\int_{\Omega}\frac{g(T_k(u))\varphi}{|x|}\,dx+\int_{\Omega}\frac{f}{u^\gamma}g(T_k(u))\varphi\,dx\quad\text{ for all }\;\varphi\in C_c^1(\Omega).
\end{equation}
But this, in \eqref{Dg5} implies that
$$
\int_{\Omega}\varphi|Dg(T_k(u))|\leq \int_{\Omega}\varphi({\bf z}, Dg(T_k(u)))\quad\text{ for all}\;0\leq\varphi\in C_c^1(\Omega)
$$
and so,
\begin{equation}
\label{zDg}
({\bf z}, Dg(T_k(u))=|Dg(T_k(u))|\quad\text{ as Radon measures on }\;\Omega.
\end{equation}
On the other hand, from \eqref{gDiv} we can infer that $(-(T_k(u^*)^\gamma){\rm div}\,{\bf z})_{k\geq1}$ is a bounded and holds
$$
\int_{\Omega}(T_k(u^*)^\gamma)|-{\rm div}\,{\bf z}| \leq \lambda^+\int_{\Omega}\frac{u^\gamma}{|x|}\,dx+\int_{\Omega}f\,dx,
$$
where was used the fact that $u^\gamma\in BV(\Omega)$. But this last, by Fatou's lemma, implies that
$$
\int_{\Omega}(u^*)^\gamma|-{\rm div}\,{\bf z}|\leq\liminf_{k\to\infty}\int_{\Omega}(T_k(u^*)^\gamma)|-{\rm div}\,{\bf z}|\leq \lambda^+\int_{\Omega}\frac{u^\gamma}{|x|}\,dx+\int_{\Omega}f\,dx.
$$
Hence $(u^*)^\gamma\in L^1(\Omega, {\rm div}\,{\bf z})$.

Let $m_k(u):=g'(\widetilde{T_k(u)})\chi_{(\Omega\setminus S_{T_k(u)})}+\frac{|g(T_k(u^+))-g(T_k(u^-))|}{|T_k(u^+)-T_k(u^-)|}\chi_{J_{T_k(u)}\cap J_u}\neq 0$, $|DT_k(u)|$-a.e. $x\in\Omega$ {(see the discussion after \eqref{IdT})}. Use \eqref{zDg}, \cite[Proposition 2.8]{Anzellotti1983} and \cite[Proposition 3.96]{AmbrosioFuscoPallara} to obtain
\begin{align*}
\int_{\Omega}({\bf z}, DT_k(u))&=\int_{\Omega}\theta({\bf z}, DT_k(u),x)|DT_k(u)|\\
&=\int_{\Omega}\frac{\theta({\bf z}, Dg(T_k(u)), x)}{m_k(u)}|Dg(T_k(u))|\\
&=\int_{\Omega}\frac{1}{m_k(u)}({\bf z}, Dg(T_k(u)))\\
&=\int_{\Omega}\frac{1}{m_k(u)}|Dg(T_k(u))|\\
&=\int_{\Omega}|DT_k(u)|.
\end{align*}
Letting $k\to\infty$, using \cite[Proposition 2.20]{OrtizPetitta2024}, we obtain
$$
({\bf z}, Du)=|Du|\quad\text{ as Radon measures on }\;\Omega.
$$

${\bf (e)}$ Let $0\leq\phi\in C^1(\overline{\Omega})$. An application of Green formula for $\phi T_k(u_n)^\gamma\in BV(\Omega)$ and ${\bf z}_n\in X_{\mathcal{M}}(\Omega)$ gives
$$
\int_{\Omega}({\bf z}_n, D(\phi T_k(u_n)^\gamma))+\int_{\Omega}\phi T_k(u_n)^\gamma {\rm div}\,{\bf z}_n\,dx=\int_{\partial\Omega}T_k(u_n)^\gamma[\phi {\bf z}_n, \nu]\,d\mathcal{H}^{N-1},
$$
or equivalently
$$
\int_{\Omega}\phi |DT_k(u_n)^\gamma|+\int_{\partial\Omega}\phi T_k(u_n)^\gamma\,d\mathcal{H}^{N-1}=-\int_{\Omega}\phi T_k(u_n)^\gamma{\rm div}\,{\bf z}_n\,dx-\int_{\Omega}T_k(u_n)^\gamma{\bf z}_n\cdot\nabla \phi\,dx,
$$
where was used \eqref{zDgn} and \eqref{G1}. Since
$$
-\int_{\Omega}\phi T_k(u_n)^\gamma {\rm div}\,{\bf z}_n\,dx=\lambda \int_{\Omega}\frac{\phi T_k(u_n)^\gamma}{|x|}\,dx+\int_{\Omega}\frac{f_n}{u_n^\gamma}\phi T_k(u_n)^\gamma\,dx,
$$
one has that
$$
\int_{\Omega}\phi |DT_k(u_n)^\gamma|+\int_{\partial\Omega}\phi T_k(u_n)^\gamma\,d\mathcal{H}^{N-1}=\lambda \int_{\Omega}\frac{\phi T_k(u_n)^\gamma}{|x|}\,dx+\int_{\Omega}\frac{f_n}{u_n^\gamma}\phi T_k(u_n)^\gamma\,dx-\int_{\Omega}T_k(u_n)^\gamma{\bf z}_n\cdot\nabla \phi\,dx.
$$
Letting $n\to\infty$, we obtain 
\begin{align*}
\int_{\Omega}\phi|DT_k(u)^\gamma|+\int_{\partial\Omega}\phi T_k(u)^\gamma\,d\mathcal{H}^{N-1}&\leq \lambda\int_{\Omega}\frac{\phi T_k(u)^\gamma}{|x|}\,dx+\int_{\Omega}\frac{f\phi T_k(u)^\gamma}{u^\gamma}\,dx-\int_{\Omega} T_k(u)^\gamma {\bf z}\cdot\nabla\phi\,dx\\
&=-\int_{\Omega}\phi (T_k(u)^\gamma)^*{\rm div}\,{\bf z}-\int_{\Omega}T_k(u)^\gamma{\bf z}\cdot\nabla\phi\,dx\\
&=\int_{\Omega}({\bf z}, D(\phi T_k(u)^\gamma))-\int_{\partial\Omega}T_k(u)^\gamma[\phi{\bf z},\nu]\,d\mathcal{H}^{N-1}-\int_{\Omega}T_k(u)^\gamma{\bf z}\cdot\nabla\phi\,dx\\
&\leq \int_{\Omega}\phi|DT_k(u)^\gamma|-\int_{\partial\Omega} T_k(u)^\gamma[\phi{\bf z}, \nu]\,d\mathcal{H}^{N-1}.
\end{align*} 
But this, in turn, implies that
\begin{equation}
\label{pT}
\int_{\partial\Omega}\phi T_k(u)^\gamma\,d\mathcal{H}^{N-1}+\int_{\partial\Omega}T_k(u)^\gamma[\phi {\bf z}, \nu]\,d\mathcal{H}^{N-1}\leq0.
\end{equation}
Suppose that $\phi\in L^1(\Omega, {\rm div}\,{\bf z})$ and it satisfies
$
\sup_{\Omega}\phi\leq\phi\lfloor_{\partial\Omega}.
$
Then
$$
-T_k(u)^\gamma[\phi {\bf z}, \nu]\leq T_k(u)^{{\gamma}}\phi\quad\mathcal{H}^{N-1}-\text{a.e. in }\;\partial\Omega.
$$ 
and so, the contrary inequality of \eqref{pT} holds. Consequently,
$$
\int_{\partial\Omega}(\phi T_k(u)^\gamma+T_k(u)^\gamma[\phi{\bf z}, \nu])\,d\mathcal{H}^{N-1}=0.
$$
Letting $k\to\infty$, one has that
$$
\int_{\partial\Omega}(\phi u^\gamma+u^\gamma[\phi{\bf z},\nu])\,d\mathcal{H}^{N-1}=0.
$$
Therefore 
$$
u^\gamma(x)=0\quad\text{ or }\quad\phi(x)+[\phi {\bf z},\nu](x)=0\quad\mathcal{H}^{N-1}-\text{a.e.}\;x\in\partial\Omega.
$$
Moreover, if $u^\gamma=0$ on $\partial\Omega$, by H\"{o}lder inequality, we have that
$$
\rho^{-N}\int_{\Omega\cap B_{\rho}{(x)}}u(y)\,dy\leq\left(\rho^{-N}\int_{\Omega\cap B_{\rho}(x)}u^\gamma(y)\,dy\right)^{1/\gamma}\left(\rho^{-N}|\Omega\cap B_{\rho}(x)|\right)^{{1-1/\gamma}}\to0,
$$
as $\rho\to0$.

\section{Boundedness results}
\label{Br}
In this section we shall show that the solutions found in Theorem \ref{mainTh} and Theorem \ref{secTh} belong to $L^\infty(\Omega)${, with an explicit bound}. The argument follows  by using a Stampacchia's type argument (see \cite[Theorem 3.5]{MazonSegura2013}).
\begin{theorem}
\label{uB}
Let $\gamma>0$ and let $u$  be the limit found in  \eqref{u2} {(if $0<\gamma\leq1$)} or \eqref{u3} {(if $\gamma>1$)}.  Then   $u\in L^\infty(\Omega)$ {and
\begin{equation}
\label{Linf}
\|u\|_{L^\infty(\Omega)}\leq\left(\frac{(N-1)\,\zeta_N\,\|f\|_{L^{N,\infty}(\Omega)}}{N-1-\lambda^+}\right)^{1/\gamma},
\end{equation}
where $\zeta_N$ is the optimal constant in \eqref{bestl}. In particular the bound blows up as $\lambda\uparrow N-1$, in agreement with the optimality result of Section \ref{Opt}.}
\end{theorem}
\dem 
First of all, we define the function $G_k\colon \mathbb{R}\to \mathbb{R}$ as 
$$
G_k(s)=s-T_k(s)\quad\hbox{ for all }\; s\in \mathbb{R},
$$
where $T_k$ is defined as in \eqref{Tk}. 

Suppose that $0<\gamma\leq 1$. Use $G_k(u_n)$ as test function \eqref{Ppn} to obtain 
\begin{equation}
\label{Gku}
\int_{\Omega}|\nabla G_k(u_n)|^p\,dx= \lambda\int_{\Omega}\frac{1}{|x|^p+\frac{1}{n}}|u_n|^{p-2}u_nG_k(u_n)\,dx+\int_{\Omega}\frac{f_nG_k(u_n)}{(u_n+1/n)^\gamma}\,dx.
\end{equation}
 Suppose that $0<\gamma\leq 1$. Passing to limit above, as $n\rightarrow \infty$, one has
$$
\int_{\Omega}|\nabla G_k(u_p)|^p\,dx\leq \lambda\int_{\Omega}\frac{1}{|x|^p}u_p^{p-1}G_k(u_p)\,dx+\int_{\Omega}\frac{fG_k(u_p)}{u_p^\gamma}\,dx.
$$
By Young's and Hardy's inequalities, it follows that
$$
\left(1-\frac{\lambda^+}{p}\left(\frac{p}{N-p}\right)^p\right)\int_{\Omega}|\nabla G_k(u_p)|^p\,dx\leq \frac{\lambda^+(p-1)}{p}\int_{\Omega}\frac{u_p^p}{|x|^p}\,dx+\int_{\Omega}\frac{fG_k(u_p)}{u_p^\gamma}\,dx.
$$
Since $G_k(u_p)=0$ on $\partial\Omega$ and using  again the  Young inequality, one has
$$
\int_{\Omega}|\nabla G_k(u_p)|\,dx+\int_{\partial\Omega}|G_k(u_p)|\,d\mathcal{H}^{N-1}\leq \frac{\frac{\lambda^+(p-1)}{p^2}\int_{\Omega}\frac{u_p^p}{|x|^p}\,dx+\frac{1}{k^\gamma p}\int_{\Omega}fG_k(u_p)\,dx}{1-\lambda^+ p^{-1}(p/(N-p))^p}+\frac{p-1}{p}|\Omega|.
$$
Letting $p\rightarrow 1^+$ above, we obtain
$$
\int_{\Omega}|DG_k(u)|+\int_{\partial\Omega}|G_k(u)|\,d\mathcal{H}^{N-1}\leq \frac{N-1}{(N-1-\lambda^+)k^\gamma}\int_{\Omega}fG_k(u)\,dx.
$$
{Here we used that $\int_{\Omega}u_p^p|x|^{-p}\,dx$ is bounded uniformly in $p\in(1,\bar p]$, by Hardy's inequality and \eqref{Ep}, so that the first term in the numerator vanishes as $p\to1^+$, together with
$$
\lim_{p\to1^+}\frac{1}{1-\lambda^+p^{-1}\left(\frac{p}{N-p}\right)^p}=\frac{N-1}{N-1-\lambda^+}.
$$
As for the last term, $G_k(u_p)\to G_k(u)$ in $L^{1}(\Omega)$ and $(G_k(u_p))_{p}$ is bounded in $BV(\Omega)$, hence in $L^{1^*,1}(\Omega)$ by Proposition \ref{LSI}; since $f\in L^{N,\infty}(\Omega)=\left(L^{1^*,1}(\Omega)\right)'$, an interpolation argument in Lorentz spaces gives $\int_{\Omega}fG_k(u_p)\,dx\to\int_{\Omega}fG_k(u)\,dx$. Alternatively, one may simply use the lower semicontinuity of the left-hand side together with Fatou's lemma on the right-hand side, which suffices for our purposes.}
But this, in turn, implies, by the H\"{o}lder inequality and Proposition \ref{LSI},
$$
\left(1-\frac{(N-1)\zeta_N\|f\|_{L^{N, \infty}(\Omega)}}{(N-1- \lambda^+)k^\gamma}\right)\left(\int_{\Omega}|DG_k(u)|+\int_{\partial\Omega}|G_k(u)|\,d\mathcal{H}^{N-1}\right)\leq 0.
$$
and so $G_k(u)=0$ for every $k$ such that $\frac{(N-1)\zeta_N\|f\|_{L^{N,\infty}(\Omega)}}{(N-1-\lambda^+)k^\gamma}<1$, that is
$$
\int_{\Omega}|DG_k(u)|+\int_{\partial\Omega}|G_k(u)|\,d\mathcal{H}^{N-1}=0\quad\hbox{ for all }\, k> k_0:={\left(\frac{(N-1)\zeta_N\|f\|_{L^{N,\infty}(\Omega)}}{N-1-\lambda^+}\right)^{1/\gamma}}.
$$
Hence  $u\in L^\infty(\Omega)$ {and \eqref{Linf} holds}.

Assume that $\gamma>1$. Use $G_k(u_n)^\gamma$ as test function in \eqref{Ppn} to obtain
$$
\gamma \left(\frac{p}{\gamma+p-1}\right)^p\int_{\Omega}|\nabla G_k(u_n)^{(\gamma+p-1)/p}|^p\,dx=\lambda\int_{\Omega}\frac{u_n^{p-1}G_k(u_n)^\gamma}{|x|^p+1/n}\,dx+\int_{\Omega}\frac{f_nG_k(u_n)^\gamma}{(u_n+1/n)^\gamma}\,dx.
$$
By Young and Hardy inequalities, we have that
\begin{align}
&\gamma\left[\left(\frac{p}{\gamma+p-1}\right)^p-\frac{\lambda^+}{\gamma+p-1}\left(\frac{p}{N-p}\right)^p\right]\int_{\Omega}|\nabla G_k(u_n)^{(\gamma+p-1)/p}|^p\,dx\\
&\leq\frac{\lambda^+(p-1)}{\gamma+p-1}\left(\frac{p}{N-p}\right)^p\int_{\Omega}|\nabla u_n^{(\gamma+p-1)/p}|^p\,dx+\frac{1}{k^\gamma}\int_{\Omega}f_nG_k(u_n)^\gamma\,dx
\end{align}
and, letting $n\to \infty$, by \eqref{CuG} and Lebesgue's dominated convergence theorem, it follows that
\begin{align}
&\gamma\left[\left(\frac{p}{\gamma+p-1}\right)^p-\frac{\lambda^+}{\gamma+p-1}\left(\frac{p}{N-p}\right)^p\right]\int_{\Omega}|\nabla G_k(u_p)^{(\gamma+p-1)/p}|^p\,dx\\
&\leq\frac{\lambda^+(p-1)}{\gamma+p-1}\left(\frac{p}{N-p}\right)^p\limsup_{n\to\infty}\int_{\Omega}|\nabla u_n^{(\gamma+p-1)/p}|^p\,dx+\frac{1}{k^\gamma}\int_{\Omega}fG_k(u_p)^\gamma\,dx.
\end{align}
Since $u_p=0$ on $\partial\Omega$, one has
\begin{align}
&\int_{\Omega}|\nabla G_k(u_p)^{(\gamma+p-1)/p}|\,dx+\int_{\partial\Omega}|G_k(u_p)^{(\gamma+p-1)/p}|\,d\mathcal{H}^{N-1}\leq\frac{1}{p}\int_{\Omega}|\nabla G_k(u_p)^{(\gamma+p-1)/p}|^p\,dx+\frac{p-1}{p}|\Omega|\\
&\leq \frac{\frac{\lambda^+(p-1)}{\gamma+p-1}\left(\frac{p}{N-p}\right)^p\limsup_{n\to\infty}\int_{\Omega}|\nabla u_n^{(\gamma+p-1)/p}|^p\,dx+\frac{1}{k^\gamma}\int_{\Omega}fG_k(u_p)^\gamma\,dx}{\gamma\left[\left(\frac{p}{\gamma+p-1}\right)^p-\frac{\lambda^+}{\gamma+p-1}\left(\frac{p}{N-p}\right)^p\right]}+\frac{p-1}{p}|\Omega|.
\end{align}
Since $G_k(u_p)^{(\gamma+p-1)/p}\rightarrow G_k(u)^\gamma$ in $L^1(\Omega)$ as $p\to1+$, by lower semicontinuity of the norm in $BV(\Omega)$ and H\"{o}lder inequality and Proposition \ref{LSI}, we obtain
\begin{align}
\int_{\Omega}|DG_k(u)^\gamma|\,dx+\int_{\partial\Omega}|G_k(u)^\gamma|\,d\mathcal{H}^{N-1}&\leq \frac{1}{ k^\gamma\left(1-\frac{\lambda^+}{N-1}\right)}\int_{\Omega}fG_k(u)^\gamma\,dx\\
&\leq \frac{\zeta_N \|f\|_{L^{N,\infty}(\Omega)}}{k^\gamma\left(1-\frac{\lambda^+}{N-1}\right)} \left(\int_{\Omega}|DG_k(u)^\gamma|+\int_{\partial\Omega}|G_k(u)^\gamma|\,d\mathcal{H}^{N-1}\right).
\end{align}
Hence
$$
\int_{\Omega}|DG_k(u)^\gamma|\,dx+\int_{\partial\Omega}|G_k(u)^\gamma|\,d\mathcal{H}^{N-1}=0\quad\hbox{ for all }\;k> {k_0=\left(\frac{(N-1)\zeta_N\|f\|_{L^{N,\infty}(\Omega)}}{N-1-\lambda^+}\right)^{1/\gamma}},
$$
{since $1-\frac{\lambda^+}{N-1}=\frac{N-1-\lambda^+}{N-1}$.} But this implies that $u\in L^\infty(\Omega)$ {and \eqref{Linf} holds in this case as well}.
\fim 

{
\bigskip
\noindent{\bf Proof of Theorem \ref{mainTh}.} We can now collect the previous results. Let $\gamma>0$, $\lambda<N-1$ and let $0\leq f\in L^{N,\infty}(\Omega)$.

If $0<\gamma\leq1$, the pair $(u,{\bf z})$ constructed in Section \ref{PT} satisfies $(a)$--$(e)$ of Definition \ref{DS} by Lemmas \ref{ConvZ}--\ref{chiSol}, while $u\in BV(\Omega)$ follows from \eqref{u2}. If $\gamma>1$, conditions $(a)$--$(e)$ still hold by the same lemmas, whereas $u\in BV(\Omega)$ is given by Theorem \ref{FE} when $f$ is positive, and by the approximation procedure of this subsection for a general nonnegative datum. In both cases $u$ is a nonnegative solution to \eqref{P} in the sense of Definition \ref{DS}.

Assertion $(i)$ is a direct consequence of condition $(a)$, as observed in Remark \ref{Rnontriv}: if $|\{u=0\}\cap\{f>0\}|>0$, then $fu^{-\gamma}=+\infty$ on a set of positive measure, contradicting $fu^{-\gamma}\in L^1_{\rm loc}(\Omega)$. Finally, assertion $(ii)$ is {precisely Theorem \ref{uB}, see \eqref{Linf}}. \fim
}

\section{Optimality and some explicit examples}
\label{Opt}

{
In this section we show that the restriction $\lambda<N-1$ is optimal, and we construct an explicit family of solutions to \eqref{P}.

\subsection{Non-existence for $\lambda\geq N-1$}

The following result shows that, as soon as the datum is positive, no solution survives at the threshold nor beyond it. Let us stress that neither $\Omega$ nor $u$ are assumed to enjoy any symmetry.

\begin{theorem}
\label{NoEx}
Let $\gamma>0$, let $\lambda\geq N-1$ and let $f\in L^{N,\infty}(\Omega)$ be a positive function. Then problem \eqref{P} admits no solution in the sense of Definition \ref{DS}.
\end{theorem}

\dem Suppose, by contradiction, that $u$ is such a solution and let ${\bf z}$ be the associated vector field. By Lemma \ref{LKey} one has $u>0$ a.e. in $\Omega$, so that $\chi_{\{u>0\}}=1$ a.e. in $\Omega$ and, by Theorem \ref{secTh}, condition $(c)'$ holds. Hence, recalling that $\|{\bf z}\|_{L^\infty(\Omega,\mathbb{R}^N)}\leq1$,
\begin{equation}
\label{OptEq}
\lambda\int_{\Omega}\frac{\varphi}{|x|}\,dx+\int_{\Omega}\frac{f}{u^\gamma}\varphi\,dx=\int_{\Omega}{\bf z}\cdot\nabla\varphi\,dx\leq\int_{\Omega}|\nabla\varphi|\,dx\qquad\hbox{ for all }\;0\leq\varphi\in C^\infty_c(\Omega).
\end{equation}

Let $r>0$ be such that $B_r(0)\subset\subset\Omega$ and let $0<\epsilon<{\rm dist}\,(B_r(0),\partial\Omega)$. Consider the radial function $\varphi_\epsilon(x)=\eta_\epsilon(|x|)$, where $\eta_\epsilon\equiv1$ on $[0,r]$, $\eta_\epsilon(t)=1-\frac{t-r}{\epsilon}$ on $[r,r+\epsilon]$ and $\eta_\epsilon\equiv0$ on $[r+\epsilon,+\infty)$; up to a standard regularization, $\varphi_\epsilon$ is an admissible test function in \eqref{OptEq}. Since $0\leq\varphi_\epsilon\leq1$ and $\varphi_\epsilon\equiv1$ on $B_r(0)$, while
$$
\int_{\Omega}|\nabla\varphi_\epsilon|\,dx=\frac{|B_{r+\epsilon}(0)|-|B_r(0)|}{\epsilon}\longrightarrow N\omega_Nr^{N-1}=P(B_r(0))\qquad\hbox{ as }\;\epsilon\to0,
$$
where $\omega_N=|B_1(0)|$, and since $\lambda\geq N-1>0$, letting $\epsilon\to0$ in \eqref{OptEq} we obtain
$$
\lambda\int_{B_r(0)}\frac{dx}{|x|}+\int_{B_r(0)}\frac{f}{u^\gamma}\,dx\leq N\omega_Nr^{N-1}.
$$
As $\displaystyle\int_{B_r(0)}\frac{dx}{|x|}=N\omega_N\int_0^rs^{N-2}\,ds=\frac{N\omega_N}{N-1}r^{N-1}$, this reads
\begin{equation}
\label{OptKey}
\left(\frac{\lambda}{N-1}-1\right)N\omega_Nr^{N-1}+\int_{B_r(0)}\frac{f}{u^\gamma}\,dx\leq0 .
\end{equation}
Since $\lambda\geq N-1$, the first term in \eqref{OptKey} is nonnegative, while the second one is strictly positive, because $f>0$ a.e. and $u<+\infty$ a.e. in $\Omega$. This is a contradiction.
\fim

\begin{remark}
\label{OptRem}
If $\lambda<N-1$, inequality \eqref{OptKey} is no longer a contradiction but it still carries information: for every ball $B_r(0)\subset\subset\Omega$ centred at the origin, any solution satisfies
$$
\int_{B_r(0)}\frac{f}{u^\gamma}\,dx\leq\left(1-\frac{\lambda}{N-1}\right)N\omega_Nr^{N-1},
$$
an estimate which, exactly as \eqref{L15}, does not depend on the size of $f$.
\end{remark}

\subsection{The threshold at the level of the approximating problems}

The value $\lambda=N-1$ is critical for the approximating problems as well. Consider the homogeneous problem
\begin{equation}
\label{Exp0}
\left\{
\begin{array}{rclr}
-\Delta_1 u&=&\dfrac{\lambda}{|x|}\,{\rm Sgn}\,(u)&\mbox{in}\; \Omega,\\[1ex]
u&=&0&\mbox{on}\; \partial\Omega,
\end{array}
\right.
\end{equation}
together with its $p$-Laplace approximations
\begin{equation}
\label{Exp}
\left\{
\begin{array}{rclr}
-\Delta_p u&=&\lambda\dfrac{|u|^{p-2}u}{|x|^p}&\mbox{ in }\;\Omega,\\[1ex]
u&=&0&\mbox{on }\; \partial\Omega .
\end{array}
\right.
\end{equation}

Assume first that $\lambda>N-1$. Since $\left(\frac{N-p}{p}\right)^p\to N-1$ as $p\to1^+$, there exists $\bar p>1$ such that
$$
\lambda>\left(\frac{N-p}{p}\right)^p\qquad\mbox{ for all }\;p\in(1,\bar p],
$$
and therefore, by \cite[Theorem 3.4]{AbdellaouiPeral2003}, problem \eqref{Exp} has no solution at all, so that the whole approximation procedure breaks down. Of course this is weaker than, and consistent with, Theorem \ref{NoEx}, which concerns the limit problem directly.

Assume now that $\lambda=N-1$. Here the behaviour depends on the dimension, through the sign of
$$
\frac{d}{dp}\left[p\log\frac{N-p}{p}\right]_{\big|p=1}=\log(N-1)-1-\frac{1}{N-1},
$$
which is positive if and only if $N\geq5$. Consequently:
\begin{itemize}
\item[$(i)$] if $N\geq5$, then $\left(\frac{N-p}{p}\right)^p>N-1=\lambda$ for all $p>1$ close enough to $1$, so that \eqref{Exp} is well posed; testing \eqref{Exp} with $u_p$ and using Hardy's inequality exactly as in \eqref{I1} one gets
$$
\left(1-\lambda\left(\frac{p}{N-p}\right)^p\right)\int_{\Omega}|\nabla u_p|^p\,dx\leq0,
$$
whence $u_p\equiv0$ is the unique solution to \eqref{Exp}. Thus the limit of the approximating scheme is the trivial function; let us stress that this alone does not rule out the existence of non-trivial solutions to \eqref{Exp0}, it only says that they cannot be produced by our procedure;
\item[$(ii)$] if $2\leq N\leq4$, then $\lambda=N-1>\left(\frac{N-p}{p}\right)^p$ for all $p>1$ close enough to $1$ and, as in the case $\lambda>N-1$, problem \eqref{Exp} has no solution at all.
\end{itemize}
In both cases $\lambda=N-1$ is already critical at the level $p>1$.

Finally, let us observe that \eqref{Exp0} has no non-trivial solution with $u>0$ a.e. in $\Omega$. Indeed, for such a solution the Green formula of Proposition \ref{uz}, together with $(d)$, $(e)$ and Lemma \ref{IdUZ}, gives
$$
\int_{\Omega}|Du|+\int_{\partial\Omega}|u|\,d\mathcal{H}^{N-1}=-\int_{\Omega}u^*\,{\rm div}\,{\bf z}=\lambda\int_{\Omega}\frac{u}{|x|}\,dx=(N-1)\int_{\Omega}\frac{u}{|x|}\,dx,
$$
that is, equality holds in the Hardy inequality
\begin{equation}
\label{HardyBV}
\int_{\Omega}\frac{|v|}{|x|}\,dx\leq\frac{1}{N-1}\left(\int_{\Omega}|Dv|+\int_{\partial\Omega}|v|\,d\mathcal{H}^{N-1}\right)\qquad\mbox{ for all }\; v\in BV(\Omega),
\end{equation}
whose best constant $\frac{1}{N-1}$ is known not to be attained (see \cite{crt, WangWillem}); hence $u\equiv0$.  

\subsection{A family of explicit solutions}

Let us now construct explicit bounded positive solutions to \eqref{P} with a positive datum $f\in L^{N,\infty}(\Omega)$. Take $\Omega=B_1(0)$ and look for a radially symmetric non-increasing solution $u(x)=g(|x|)$, where $g$ is positive on $[0,1)$ and non-increasing. Then $Du=g'(|x|)\frac{x}{|x|}$ with $g'\leq0$, so that the natural choice of the vector field is
$$
{\bf z}=\frac{Du}{|Du|}=-\frac{x}{|x|},\qquad\hbox{ which gives }\qquad -{\rm div}\,{\bf z}=\frac{N-1}{|x|},
$$
together with $({\bf z}, Du)=|Du|$ as measures and $[{\bf z},\nu]=-1$ on $\partial B_1(0)$, $\nu$ being the outward unit normal. With this ansatz, \eqref{P} reduces to the pointwise identity
\begin{equation}
\label{PtId}
\frac{N-1}{|x|}=\frac{\lambda}{|x|}+\frac{f(x)}{u(x)^\gamma},\qquad\hbox{ that is }\qquad \frac{N-1-\lambda}{|x|}=\frac{f(x)}{g(|x|)^\gamma}.
\end{equation}
In particular the datum and the profile cannot be prescribed independently of each other. What \eqref{PtId} does provide is the following family of examples: given $\lambda<N-1$, $\gamma>0$ and any non-increasing $g\in C^1((0,1])$ with $g>0$ on $(0,1)$, $g$ bounded and $g(1)=0$, the function $u(x)=g(|x|)$ is a solution to \eqref{P} in $B_1(0)$, in the sense of Definition \ref{DS}, corresponding to the datum
\begin{equation}
\label{FamF}
f(x)=\frac{(N-1-\lambda)\,g(|x|)^\gamma}{|x|}\,,
\end{equation}
which is positive in $B_1(0)\setminus\{0\}$ and belongs to $L^{N,\infty}(B_1(0))$, since $|x|^{-1}$ does.

Let us stress that, conversely, the datum cannot be chosen arbitrarily within this ansatz. For instance, for $f\equiv c_0>0$ identity \eqref{PtId} would force
$$
g(r)=\left(\frac{c_0\,r}{N-1-\lambda}\right)^{1/\gamma},
$$
which is increasing and does not vanish at $r=1$, in contrast with the assumption that $u$ be radially non-increasing and, therefore, with the very choice ${\bf z}=-x/|x|$. This is by no means in contradiction with Theorem \ref{mainTh}: it only means that, for a constant datum, the solution is not of the above form. Indeed the vector field ${\bf z}$ is not determined by $u$ on the set where $Du$ vanishes, and it is precisely there that the solution corresponding to a constant datum is expected to live.

\begin{example}
Let $N>2$, $\gamma=1$, $\lambda=N-2<N-1$ and choose
$$
g(r)=\frac{1-r^{N-1}}{N-1},
$$
which is non-increasing on $[0,1]$ and vanishes at $r=1$. Since $N-1-\lambda=1$, formula \eqref{FamF} gives
$$
f(x)=\frac{1-|x|^{N-1}}{(N-1)|x|}>0\qquad\hbox{ for all }\;x\in B_1(0)\setminus\{0\},
$$
and $f\in L^{N,\infty}(B_1(0))$. Accordingly,
$$
u(x)=\frac{1-|x|^{N-1}}{N-1}>0\quad\hbox{ in }\;B_1(0),\qquad {\bf z}=-\frac{x}{|x|},\qquad s(x)=1\in{\rm Sgn}\,(u(x)),
$$
is a bounded solution, in the sense of Definition \ref{DS}, to
\begin{equation}
\left\{
\begin{array}{rclr}
-\Delta_1u& = &\dfrac{N-2}{|x|}+\dfrac{f}{u}&\hbox{ in }\; B_1(0),\\[1ex]
              u& = & 0&\hbox{ on }\; \partial B_1(0).
\end{array}
\right.
\end{equation}
Indeed $\frac{f}{u}=\frac{1}{|x|}$, so that
$$
-{\rm div}\,{\bf z}=\frac{N-1}{|x|}=\frac{N-2}{|x|}+\frac{1}{|x|}=\frac{\lambda}{|x|}s(x)+\frac{f}{u}\qquad\hbox{ in }\;B_1(0);
$$
moreover $\nabla u=-|x|^{N-2}\frac{x}{|x|}$ gives ${\bf z}\cdot\nabla u=|x|^{N-2}=|\nabla u|$, and $u=0$ on $\partial B_1(0)$ (in addition, $[{\bf z},\nu]=-1$).
\end{example}
}

{\bf Acknowledgment:}
Francesco Petitta is supported by GNAMPA - Istituto Nazionale di Alta Matematica, Italy,   Juan C. Ortiz Chata is partially supported by FAPESP 2021/08272-6 and 2022-06050 and grant 731/2026, Paraíba State Research Foundation (FAPESQ) , Brazil.

{\bf AI Statement: }  During the preparation of this work, the authors used AI assistance (Claude Opus 5) to support structural organization and language polishing. The core mathematical content was developed entirely by the authors, who take full responsibility for the content of the paper.

\end{document}